\documentclass[letter,12pt]{amsart}

\usepackage[T1]{fontenc}
\usepackage[latin1]{inputenc}
\usepackage{pgfplots}
\usepgfplotslibrary{fillbetween}
\usepackage{enumitem}
\usepackage{subcaption} %side by side diagrams
\usepackage{caption} %customises captions
\usepackage{floatrow}

\usepackage{setspace}
\usepackage{marginnote}
\newcommand\note[1]%
                {$^\dagger$\marginnote{\footnotesize{$^\dagger${#1}}}}

\usepackage{geometry}
\usepackage[hidelinks]{hyperref}

\usepackage{amsthm}
\usepackage{amsmath}
\usepackage{amsfonts}
\usepackage{amssymb}
\usepackage{mathdots}
\usepackage{stmaryrd}

\usepackage{times}
\usepackage{bm}
\usepackage{dutchcal}

\usepackage{tikz-cd}
\usepackage{fancybox}
\usepackage{caption}
\usepackage{subcaption}
\usepackage{graphicx}
\usepackage{tikz}
\usetikzlibrary{arrows.meta}
\tikzset{%
  >={Latex[width=2mm,length=2mm]},
            base/.style = {rectangle, rounded corners, draw=black,
                           minimum width=4cm, minimum height=1cm,
                           text centered, font=\sffamily},
  activityStarts/.style = {base, fill=blue!30},
       startstop/.style = {base, fill=red!30},
    activityRuns/.style = {base, fill=green!30},
         process/.style = {base, minimum width=2.5cm, fill=orange!15,
                           font=\sffamily},
}

\usepackage[toc,page]{appendix}
\usepackage{todonotes}

\DeclareMathOperator{\Vol}{Vol}
\DeclareMathOperator{\Hom}{Hom}

\DeclareMathOperator{\im}{Im}%%image
\DeclareMathOperator{\id}{Id}

\DeclareMathOperator{\pr}{pr}

\DeclareMathOperator{\Lie}{Lie}
\DeclareMathOperator{\Ad}{Ad}

\DeclareMathOperator{\gr}{gr}

\DeclareMathOperator{\spec}{Spec}
\newcommand{\n}{^{-1}}

\DeclareMathOperator{\ann}{ann}

\DeclareMathOperator{\cone}{cone}

\DeclareMathOperator{\hull}{hull}
\DeclareMathOperator{\fix}{Fix}

\newcommand\lie{\mathfrak}

\newcommand{\g}{\lie{g}}

\newcommand{\hhh}{\lie{h}}

\newcommand{\nnn}{\lie{n}}

\newcommand{\ttt}{\lie{t}}
\newcommand{\kk}{\lie{k}}

\newcommand\bb{\mathbb}

\newcommand\N{\bb{N}}
\newcommand\Z{\bb{Z}} 
\newcommand\Q{\bb{Q}}
\newcommand\R{\bb{R}} 
\newcommand\C{\bb{C}}

\renewcommand\H{\bb{H}}
\newcommand\T{\bb{T}}

\newcommand\ii{\mathbf{i}}

\newcommand\X{\mathfrak{X}}

\usepackage{theoremref}
\newtheorem{theorem}{Theorem}[section]

\newtheorem{proposition}[theorem]{Proposition}
\newtheorem{lemma}[theorem]{Lemma}
\newtheorem{corollary}[theorem]{Corollary}

\theoremstyle{definition}
\newtheorem{definition}[theorem]{Definition}
\newtheorem{data}[theorem]{Data}
\newtheorem{example}[theorem]{Example}

\theoremstyle{remark}
\newtheorem{remark}[theorem]{Remark}

\makeatletter
\@namedef{subjclassname@2020}{\textup{2020} Mathematics Subject Classification}
\makeatother

\makeatletter
\newcommand\footnoteref[1]{\protected@xdef\@thefnmark{\ref{#1}}\@footnotemark}
\makeatother

\date{}

\begin{document}

\title{Stratified gradient Hamiltonian vector fields and collective integrable systems} 

\author{Benjamin Hoffman}
\address{Earth Species Project}
\email{benjamin@earthspecies.org, benjaminsshoffman@gmail.com}

\author{Jeremy Lane}
\address{Amazon\footnote{Work completed before joining Amazon}}
\email{planjere@amazon.com}

%    General info

\date{\today}

% \dedicatory{This paper is dedicated to our advisors.}


\maketitle

\begin{abstract}

We construct completely integrable systems on the dual of the Lie algebra of any compact Lie group $K$ with respect to the standard Lie-Poisson structure. These systems generalize key properties of Gelfand-Zeitlin systems: A) the pullback to any Hamiltonian $K$-manifold defines a Hamiltonian torus action on an open dense subset, B) if the $K$-manifold is multiplicity-free, then the resulting torus action is \textit{completely} integrable, and C) the collective moment map has convexity and fiber connectedness properties. These systems generalize the relationship between Gelfand-Zeitlin systems and Gelfand-Zeitlin canonical bases via geometric quantization by a real polarization. 

To construct these systems, we generalize Harada and Kaveh's construction of integrable systems by toric degeneration on smooth projective varieties to singular quasi-projective varieties. Under certain conditions, we show that the stratified-gradient Hamiltonian vector field of such a degeneration, which is defined piece-wise, has a flow whose limit exists and defines continuous degeneration map.

\end{abstract}

% \setcounter{tocdepth}{1}
% \tableofcontents

\section{Introduction}

This paper lies at the intersection of two topics in Hamiltonian mechanics: collective integrable systems and geometric quantization. Collective integrable systems arise when constructing commuting Hamiltonian symmetries from non-commuting Hamiltonian symmetries of a symplectic manifold. Given a Hamiltonian action of non-abelian compact Lie group $K$ on a symplectic manifold $(M,\omega)$ with equivariant moment map $\mu$, assume there exists a completely integrable system on $\Lie(K)^*$ with respect to the standard Poisson structure\footnote{We define completely integrable systems on constant rank Poisson manifolds in Definition \ref{def; completely integrable system on constant rank Poisson manifold}. By a completely integrable system on $\Lie(K)^*$ we mean a set of continuous functions that restricts to a completely integrable system on each maximal constant rank Poisson submanifold of $\Lie(K)^*$, i.e., the orbit-type strata of the coadjoint action.}. Pulling this system back by $\mu$ yields a \textit{collective} integrable system\footnote{An \emph{integrable system} on a smooth connected symplectic manifold of dimension $2n$ is a collection of $k \leq n$ continuous functions that are smooth on an open dense subset and are functionally independent and pairwise Poisson commute there (see the discussion preceding Definition \ref{def; integrable system on singular symplectic space} for more detail). If $k=n$ then it is a \emph{completely integrable system}.} on $(M,\omega)$. This is useful because completely integrable systems have tractable local and global structures, as demonstrated by the Liouville-Arnold theorem \cite{arnoldMathematicalMethods} and the Delzant classification of compact symplectic toric manifolds \cite{delzantHamiltoniensPeriodiquesImages1988}.

Geometric quantization is a procedure that maps classical observables (functions on a symplectic manifold) to quantum observables (unitary operators on a Hilbert space). In the version of geometric quantization developed by Kostant and Souriau, classical observables are mapped to operators on the space of sections of a pre-quantum line bundle over the symplectic manifold~\cite{woodhouse1992geometric}. These sections are required to be flat with respect to a choice of polarization, which ensures the geometric quantization recovers a Hilbert space of the expected dimension.  For example, geometric quantization of the coadjoint action of $K$ on one of its integral coadjoint orbits via K\"ahler polarization produces the irreducible $K$-module corresponding to the same highest weight~\cite{guilleminGeometricQuantizationMultiplicities1982}. 
% There are various methods to produce polarizations. Cases where two different polarizations result in the same quantization, as expected by the physical motivation, are known as ``independence of polarization'' results  \cite{guilleminGelfandCetlinSystemQuantization1983, jeffrey1992bohr}.

In~\cite{guilleminGelfandCetlinSystemQuantization1983}, Guillemin and Sternberg extended this approach to collective integrable systems and canonical bases by constructing a family of completely integrable systems, called \emph{Gelfand-Zeitlin systems}, on the duals of the Lie algebras of unitary Lie groups\footnote{Their construction also applies to orthogonal Lie groups.}. When a Gelfand-Zeitlin system is restricted to an integral coadjoint orbit of $U(n)$ it defines a completely integrable system whose fibers define a (singular) real polarization of the prequantum line bundle. Modulo the boundary of the polytope, the integer lattice points of the Gelfand-Zeitlin system's moment map image parametrize Bohr-Sommerfeld leaves of the real polarization, which in turn parameterize a basis of the geometric quantization. The same integer lattice points also parametrize a Gelfand-Zeitlin canonical basis of the irreducible $U(n)$-module that arises during the K\"ahler quantization. Guillemin and Sternberg interpreted this an example of the principle of ``independence of polarization'': the geometric quantization of a symplectic manifold should not depend on the choice of polarization.

When a Gelfand-Zeitlin system is pulled back via the moment map of a Hamiltonian $U(n)$-action on a symplectic manifold, one obtains a collective integrable system. These systems can be used to prove independence of polarization results, such as for the cotangent bundle of $U(n)$~\cite{CrooksWeitsman3}. Gelfand-Zeitlin systems have also been applied to symplectic topology \cite{pabiniakGromovWidthNonregular2013, nishinouToricDegenerationsGelfand2010}, abelianization of symplectic quotients  \cite{CrooksWeitsman}, and non-abelian Duistermaat Heckman measures \cite{CrooksWeitsman2}. These applications are possible because Gelfand-Zeitlin systems have several remarkable properties: 

\begin{enumerate}[label=(\Alph*)]

\item The pullback of a Gelfand-Zeitlin system defines a Hamiltonian torus action on a dense open subset of any Hamiltonian $U(n)$-manifold \cite{guilleminGelfandCetlinSystemQuantization1983}; 
\item The resulting torus action is completely integrable if the Hamiltonian $U(n)$-manifold is multiplicity-free \cite{guilleminCollectiveCompleteIntegrability1983}; and 
\item The moment map of a collective Gelfand-Zeitlin system has convexity and fiber connectedness properties if the manifold  is compact\cite{laneConvexityThimmTrick2018}.
\end{enumerate}

% Gelfand-Zeitlin systems are completely integrable systems on $\Lie(U(n))^*$ introduced by Guillemin and Sternberg
% \cite{guilleminGelfandCetlinSystemQuantization1983}. After identifying $\Lie(U(n))^*$ with the space of Hermitian $n\times n$ matrices in the standard manner, the coordinates of a Gelfand-Zeitlin system can be defined as the ordered eigenvalues of a sequence of nested principal minors of consecutive dimensions. The image of $\Lie(U(n))^*$ under any Gelfand-Zeitlin system is the convex polyhedral cone defined by the interlacing inequalities of the eigenvalues, also known as a Gelfand-Zeitlin cone. Gelfand-Zeitlin systems are so-called because the integer lattice points of the Gelfand-Zeitlin cone -- i.e., tuples of interlacing integer eigenvalues of consecutive principal minors -- correspond to elements of the Gelfand-Zeitlin canonical bases of irreducible unitary representations \cite{GelfandCetlin}. A similar construction yields completely integrable systems with the same properties for orthogonal groups. 

Motivated by the applications of Gelfand-Zeitlin systems and the connection to geometric quantization, one may ask if there exist collective integrable systems for other compact Lie groups that share properties (A) -- (C). The construction used by Guillemin and Sternberg is clever and simple but does not produce enough independent commuting functions to form a completely integrable system on the dual of the Lie algebra of other groups. There are general constructions of completely integrable systems on the dual of the Lie algebra of arbitrary compact Lie groups  \cite{MischenkoFomenko,haradaSymplecticGeometryGel2006}, but these have not been shown to to share properties (A) or (C). Property (A) is non-trivial due to nutation effects \cite[Section 4]{guilleminGelfandCetlinSystemQuantization1983}. Another approach using toric degeneration to construct completely integrable systems \cite{nishinouToricDegenerationsGelfand2010, haradaIntegrableSystemsToric2015} has produced new independence of polarization results \cite{HaradaHamiltonKaveh} and been applied in symplectic topology \cite{fangSimplicesNewtonOkounkovBodies2018}, symplectic geometry of projective varieties \cite{kavehToricDegenerationsSymplectic2019}, and symplectic cohomological rigidity \cite{pabiniakSymplecticCohomologicalRigidity2020}. However, the toric degeneration constructions in the literature that precede this work do not produce integrable systems on all of $\Lie(K)^*$, and thus cannot produce collective integrable systems on arbitrary Hamiltonian $K$-manifolds. Yet another approach to this problem uses Poisson-Lie theory and tropicalization to construct completely integrable systems on  $\Lie(K)^*$ as a limit of Ginzburg-Weinstein isomorphism \cite{ alekseevConcentrationSymplecticVolumes2018, alekseevActionangleCoordinatesCoadjoint2020,alekseevGelfandZeitlinSystemTropical2018}. This approach also falls short of constructing collective integrable systems on all of $\Lie(K)^*$ with properties (A) -- (C), with the notable  exception of $U(n)$  where the approach recovers the Gelfand-Zeitlin systems \cite{alekseevGelfandZeitlinSystemTropical2018}.

This paper presents a construction of collective integrable systems for arbitrary compact Lie groups that generalizes properties (A) -- (C) of Gelfand-Zeitlin systems\footnote{We prove slightly more general results. For example, our construction extends to Hamiltonian $K$-spaces and our convexity result extends to proper Hamiltonian $K$-manifolds. See Propositions \ref{proposition; properties of toric contraction I} -- \ref{proposition; properties of toric contraction IV}.}. 

\begin{theorem}\label{intro main theorem}
Let $K$ be a compact connected Lie group. Let $\mathbb{T}$ be a compact torus of dimension $\frac{1}{2}(\dim_\R(K) + \dim_\R(T))$, where $T$ is the maximal torus of $K$. There exists a completely integrable system\footnote{See Definition \ref{def; completely integrable system on constant rank Poisson manifold}.} $F \colon \Lie(K)^* \to \Lie(\T)^*$ such that, for any connected symplectic manifold $M$  equipped with a Hamiltonian action of $K$ with equivariant moment map $\mu \colon M \to \Lie(K)^*$:
    \begin{enumerate}[label=(\Alph*)]
        \item  The composition
\[
    \mu_\T \colon M \xrightarrow{\mu} \Lie(K)^* \xrightarrow{F} \Lie(\T)^*
\]
is a moment map for a Hamiltonian action of $\T$ on a connected, open, dense subset $D \subset M$.
        \item The complexity\footnote{See Equation \eqref{def; complexity part 1} and the following text for definition of complexity of Hamiltonian group actions.} of the $\T$-action on $D$ equals the complexity of the $K$-action on $M$. In particular, the action of $K$ on $M$ is multiplicity-free (complexity 0) if and only if the action of $\T$ on $D$ is completely integrable (complexity 0).
        \item If $M$ is compact, then the fibers of $\mu_\T$ are connected and $\mu_\T(M)$ is a rational convex polytope that projects linearly onto the Kirwan polytope of $\mu$. In this case, $D$ equals the preimage under $\mu_\T$ of the smooth locus of the rational polytope $\mu_\T(M)$.\footnote{For a definition of the smooth locus of a polyhedral set, see Section~\ref{section; convex geometry}.}
    \end{enumerate}
\end{theorem}

Except for Lie types A, B, and D, for which the Gelfand-Zeitlin systems are completely integrable, this is the first general existence theorem for completely integrable systems on $\Lie(K)^*$ with these properties for arbitrary compact Lie group $K$. In the course of our proof, which is constructive, we demonstrate how the moment map images of these systems are related to convex polyhedral cones generated by value semigroups associated with the affine closure of base affine space (Example \ref{main example}, Example \ref{ex; flp valuations}, Remark \ref{rem; other valuations}). One well known example of these is the extended string cones, whose integer lattice points correspond to elements of a Kashiwara-Lusztig crystal basis (Example \ref{main example 1.2}).

Following the same logic as  \cite{guilleminGelfandCetlinSystemQuantization1983}, these collective integrable systems produce (singular) real polarizations of arbitrary, possibly non-compact, multiplicity-free spaces equipped with $K$-invariant prequantum data.  This can produce new ``independence of polarization'' results as we demonstrate in Example~\ref{main example 3}. 

\subsection*{Outline of Proof} Our proof of Theorem \ref{intro main theorem} starts with a general result about the toric degeneration construction of integrable systems. We generalize the known construction of \cite{haradaIntegrableSystemsToric2015}, which applies to smooth projective varieties, to the setting of quasi-projective varieties that are possibly singular\footnote{One must be careful to define what is meant by \emph{completely integrable system} in this more general setting where $X$ is not smooth. Our construction yields completely integrable systems in the sense of our Definition \ref{def; integrable system on singular symplectic space}. In particular, the resulting functions restrict to a completely integrable system on each piece of a decomposition of  $X$ into smooth pieces. See Section \ref{section; decomposed and stratified spaces} for definitions of decomposed spaces and varieties.} . We use a stratified version of the gradient Hamiltonian vector field and introduce several technical assumptions that are sufficient to integrate them to time-1 (Theorem~\ref{thm; toric degenerations main theorem}, Figures~\ref{fig1} and~\ref{fig2}). We then turn our attention to affine varieties, and identify a set of conditions under which we can apply this construction (Theorem~\ref{maintheorem}).

\begin{figure}[h]
\includegraphics[width=15cm]{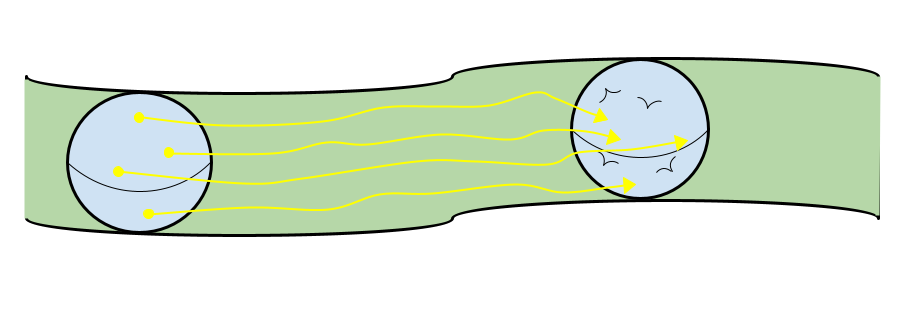}
\caption{In Harada-Kaveh the time-1 limit of the gradient Hamiltonian flow defines a continuous map from the 1-fiber (left), a smooth projective variety, to the possibly singular toric fiber of the degeneration (right).}\label{fig1}
\end{figure}

\begin{figure}[h]\includegraphics[width=15cm]{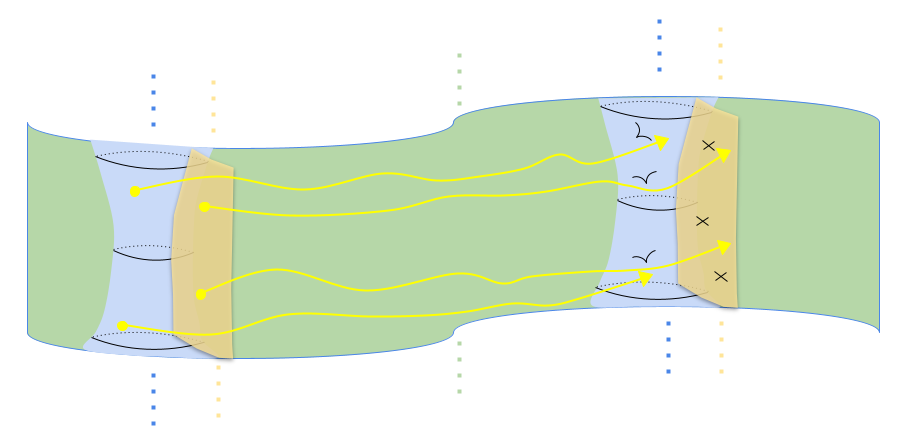}
\caption[]{In our work, the 1-fiber (left) is not necessarily smooth or projective. Instead we suppose it has a decomposition
into smooth pieces (illustrated in blue and orange).
We define a piecewise gradient Hamiltonian vector field which we call the \emph{stratified gradient Hamiltonian vector field}. Under some assumptions the time-1 limit of its flow exists and defines a continuous map to the toric fiber (right).}\label{fig2}
\end{figure}

Next, we apply our toric degeneration result to construct a completely integrable system on the affine closure of base affine space,  $G\sslash N$, of any complex semisimple Lie group $G$. This construction applies to a large family of toric degenerations of $G\sslash N$ that arise from certain previously-studied valuations on $\C[G]^N$ (Theorem~\ref{maintheoremGmodN} and Example~\ref{main example}). 

For a fixed choice of valuation, the resulting integrable system on $G\sslash N$ descends under a natural quotient map to the collective integrable system $F\colon \Lie(K)^*\to\Lie(\mathbb{T})^*$ in Theorem~\ref{intro main theorem}. We use several tricks of symplectic geometry -- such as the identification of $G\sslash N$ with the symplectic implosion of $K$ -- to bootstrap this construction to any Hamiltonian $K$-manifold of a compact Lie group $K$. Properties (A) -- (C) of Theorem~\ref{intro main theorem} are proved in Section~\ref{section proof of main theorem}.

\subsection*{Applications}

% In the case where $M$ is a compact, multiplicity-free $K$-manifold, Theorem~\ref{intro main theorem} (C) can be used to put a toric chart on the open dense subset $D$ via the classification of proper toric symplectic manifolds \cite{karshonNonCompactSymplecticToric2015}. This toric chart is the product of a certain open dense convex subset of the moment map image $F \circ \mu(M)$ and the torus $\mathbb{T}$, equipped with standard action-angle coordinates. The moment map image $F \circ \mu(M)$ is described in Proposition \ref{proposition; properties of toric contraction III}.\footnote{However, applying this result is non-trivial in most cases because $F \circ \mu(M)$ is the intersection of two convex sets that can both be difficult to describe explicitly: the convex polyhedal cone $F(\Lie(K)^*)$ which is the convex hull of the value semigroup of $G\sslash N$ resulting from the valuation used to construct $F$, and the pre-image of the Kirwan polytope of $M$ under a certain linear projection. This description of  $F \circ \mu(M)$ is very similar to Alexeev and Brion's description of  the toric moment polytopes resulting from toric degenerations of smooth spherical varieties \cite[Theorem 3.5]{alexeevToricDegenerationsSpherical2005}. However, our setting is considerably more general: compact multiplicity-free Hamiltonian $K$-manifolds may not even admit a compatible $K$-invariant complex structure \cite{woodwardMultiplicityfreeHamiltonianActions1998}.}

We provide two applications of Theorem~\ref{intro main theorem} where the moment map image is not difficult to describe:  coadjoint orbits (Example \ref{main example 2}) and cotangent bundles (Example \ref{main example 3}) of arbitrary compact connected Lie groups. The former is related to the study of Gromov width and the latter to independence of polarization. Finally, we discuss connections with the recent work of Crooks and Weitsman on non-abelian Duistermaat-Heckman measures.

Our first application is to computing the Gromov width of coadjoint orbits of compact Lie groups (see Example \ref{main example} and Section \ref{ss; gromov width}). Karshon and Tolman have conjectured a simple formula for the Gromov width of coadjoint orbits of semisimple compact Lie groups and  tight upper bounds are known in all cases \cite{castroUpperBoundGromov2016}. Tight lower bounds were proven for all orbits in type A, B, D using Gelfand-Zeitlin systems \cite{pabiniakGromovWidthNonregular2013}. Tight lower bounds were later generalized to most coadjoint orbits of arbitrary Lie type using the results of Harada-Kaveh in \cite{fangSimplicesNewtonOkounkovBodies2018}, but a gap in the generalization to arbitrary Lie type remains (see the end of Section \ref{ss; gromov width} for additional details).  This gap highlights the difference between Gelfand-Zeitlin systems and the construction of Harada-Kaveh~\cite{haradaIntegrableSystemsToric2015}: The Harada-Kaveh construction cannot produce integrable systems on coadjoint orbits that are not parameterized by a scalar multiple of a dominant integral weight (i.e. irrational orbits) because their symplectic forms cannot be written as a constant multiple of a Fubini-Study K\"ahler form on the orbit.  Although one might na\"ively expect that these results can somehow extend by a continuity argument to the irrational coadjoint orbits, there is no obvious way that this works\footnote{In fact, an early pre-print of \cite{fangSimplicesNewtonOkounkovBodies2018} attempted to do exactly this using a Moser trick argument which was later found to be incorrect. The integrable systems constructed by Harada and Kaveh on rational orbits do not have any obvious continuity properties since Harada-Kaveh treats each orbit as a separate projective variety. In some sense, the na\"ive intuition that Harada-Kaveh should fit together continuously is precisely the motivation for our development of stratified gradient Hamiltonian vector fields in Part I. Our continuity theorem for stratified gradient Hamiltonian vector fields can be viewed as a realization of this intuition.}.  In Section \ref{ss; gromov width} we show how the integrable systems we construct on arbitrary coadjoint orbits in arbitrary type (Example \ref{main example}) eliminate this technical roadblock.  As a result, we reduce the Karshon--Tolman conjecture to a purely algebraic conjecture regarding the existence of good birational orderings with certain algebraic properties (Theorem~\ref{thm;gwidth}). 

Our second application is independence of polarization in geometric quantization. We show how completely integrable torus actions can be constructed on an open dense subset of $T^*K$, for any compact Lie group $K$, by applying our construction to the natural multiplicity-free $K\times K$-action on $T^*K$ (Example \ref{main example 3}).  This illustrates how our construction can be applied to non-compact Hamiltonian manifolds. Conveniently, this is also a case where the moment map image is relatively easy to describe. If a string valuation is used in the construction, then the integer lattice points in the moment map image can be identified with elements of a crystal basis for $G = K^\C$. Modulo singularities, this recovers a graded isomorphism between the Bohr-Sommerfeld fibers of our completely integrable system on $T^*K$ and a basis of the K\"ahler polarization. This generalizes the recent independence of polarization result for $T^*U(n)$ that was proven using collective Gelfand-Zeitlin systems and their connection to Gelfand-Zeitlin bases \cite{CrooksWeitsman3}. 

Finally, we highlight the connection of our construction to the recent work of Crooks and Weitsman on abelianization of symplectic quotients and non-abelian Duistermaat-Heckman measures.  Crooks and Weitsman introduce the notion of \emph{Gelfand-Zeitlin data}, an abstract generalization of the collective complete integrability properties of Gelfand-Zeitlin systems \cite[Definition 1]{CrooksWeitsman}. Unitary and orthogonal Gelfand-Zeitlin systems are examples of Gelfand-Zeitlin data.  Crooks and Weitsman show there are canonical isomorphisms between generic symplectic reductions of Hamiltonian manifolds with respect to $K$, and with respect to the densely defined big torus action generated by a Gelfand-Zeitlin datum \cite[Main Theorem]{CrooksWeitsman}. This result allows them to relate the Radon-Nikodym derivatives of the Duistermaat-Heckman measures on $\Lie(K)^*$ and $\Lie(\T)^*$ \cite{CrooksWeitsman2}.  In comparison, the other well-known method to study Duistermaat-Heckman measures of Hamiltonian $K$-manifolds is to consider the projection of the Duistermaat-Heckman measure on $\Lie(K)^*$ onto the positive Weyl chamber. The completely integrable systems we construct in Theorem \ref{thm; liek completely integrable system} proves the existence of Gelfand-Zeitlin data in all Lie types. Thus, our construction extends the results of Crooks and Weitsman to any compact Lie group $K$.

% \begin{figure}[h]
% \includegraphics[width=15cm]{cartoon1}
% \caption{In Harada-Kaveh the time-1 limit of the gradient Hamiltonian flow defines a continuous map from the 1-fiber (left), a smooth projective variety, to the possibly singular toric fiber of the degeneration (right).}\label{fig1}
% \end{figure}

% \begin{figure}[h]\includegraphics[width=15cm]{cartoon2}
% \caption[]{In our work, the 1-fiber (left) is not necessarily smooth nor projective. Instead we suppose it has a decomposition\footnoteref{decomposed_footnote} 
% into smooth pieces (illustrated in blue and orange).
% We define a piecewise gradient Hamiltonian vector field which we call the \emph{stratified gradient Hamiltonian vector field}. Under some assumptions the time-1 limit of its flow exists and defines a continuous map to the toric fiber (right).}\label{fig2}
% \end{figure}

\subsection*{Acknowledgments}

The authors would like to thank Anton Alekseev, Peter Crooks, Mark Hamilton, Megumi Harada, Yael Karshon, Kiumars Kaveh, Allen Knutson, Chris Manon, Daniele Sepe, and Reyer Sjamaar for helpful suggestions and conversations throughout the course of this project. J.L. would like to thank the
Fields Institute and the organizers of the thematic program on Toric Topology and Polyhedral Products for
the support of a Fields Postdoctoral Fellowship during writing of this paper.

\section{Setup} \label{background section}

This section establishes terminology, conventions, and notation and collects several useful lemmas.

\subsection{Decomposed spaces and varieties}\label{section; decomposed and stratified spaces}
Following \cite[Definition 1.1.1]{pflaumAnalyticGeometricStudy2001}, a \emph{decomposed space} is a paracompact, Hausdorff, countable  topological space $X$ equipped with a locally finite partition by locally closed subspaces (called \emph{pieces}) such that: 
\begin{enumerate}
    \item Every piece is equipped with the structure of a smooth manifold that is compatible with the subspace topology;
    \item (Frontier condition) If $Y \cap \overline{Z} \ne \emptyset$ for a pair of pieces $Y, Z$ of $X$, then $Y\subset \overline{Z}$.
\end{enumerate} 
We often denote the pieces of a decomposed space $X$ by $X^\sigma$, where $\sigma$ is an element of an indexing set $\Sigma$. Let $\prec$ denote the partial order by inclusion with respect to closures on the set of pieces of a decomposed space. We also equip the indexing set with this partial order.

By a \emph{variety} we mean an irreducible quasi-projective variety over the complex numbers.  
Let $\mathbf{I}(X)\subset \C[\mathbb{C}^n]$ denote the ideal of functions vanishing on $X \subset \C^n$. Let $\mathbf{V}(I)\subset \mathbb{C}^n$ denote the vanishing locus of an ideal $I\subset \C[\mathbb{C}^n]$. All algebras we consider will be  integral domains. Given an algebra $A$ we make no distinction between the affine scheme $\spec A$ and its set of closed points. A \emph{decomposed variety} is a variety $X$, equipped with a partition by finitely many smooth irreducible subvarieties $X^\sigma $, which endows $X$ with the structure of a decomposed space with respect to its analytic topology. 

\begin{definition}\label{def; decompsed Kahler variety}
A \emph{decomposed K\"ahler variety} is a tuple $(X,M,\omega)$ where: (i) $X$ is a decomposed variety, (ii) $M$ is a smooth variety equipped with a K\"ahler form $\omega$ (compatible with the complex structure on $M$), and (iii) $X$ is equipped with an algebraic embedding into $M$ as a (not necessarily closed) subvariety. 
\end{definition}
The embedding is implicit in our tuple notation for decomposed K\"ahler varieties.
If $M$ is a vector space and $\omega$ is linear, then we say that $(X,M,\omega)$ is a \emph{decomposed affine K\"ahler variety}.

\subsection{Convex geometry}\label{section; convex geometry} A \emph{lattice} is a free $\Z$-module of finite rank. Given a lattice $L$, denote $L^\vee=\Hom(L,\Z)$ and  $L_\R = L\otimes_\Z \R$.  Given a set $A \subset L$, let $\cone(A) \subset L_\R$ denote the cone generated by $A$. We import well-known terminology  from \cite{coxToricVarieties2011}.

A \emph{locally rational polyhedral set} is a set $V \subset L_\R$ such that for all $p \in V$, there is a neighborhood $U_p$ of $p$ in $L_\R$ and a rational convex polyhedron $P$ such that $U_p \cap V = U_p \cap P$. A point $p\in V$ is a \emph{smooth point} of $V$ if it is a smooth point of this rational convex polyhedron $P$. The \emph{smooth locus} of  $V$ is the set of smooth points of $V$ and the \emph{singular locus} is its complement.

\subsection{Lie theory}
\label{Lie theory section} Unless stated otherwise, $K$ denotes a compact connected Lie group and $G$ denotes the complex form of $K$. Fix a maximal complex algebraic torus $H\subset G$ and let $T$ be the maximal compact torus $H\cap K$ in $K$. We denote Lie algebras with fraktur letters, e.g.~$\Lie(G)=\g$. 

Let $\Lambda \subset \ttt^*$ denote the lattice of real weights of $T$. We use the convention that each $\lambda\in \Lambda$ corresponds to the character $T \to S^1$, $t = \exp(\xi) \mapsto t^\lambda = e^{\sqrt{-1}\langle \lambda,\xi\rangle}$, for all $\xi \in \ttt$. 

Fix a set of positive roots  $R_+$ and let $\Lambda_+ \subset \Lambda$ denote the semigroup of dominant real weights. Let $m$ denote the number of positive roots, let $n$ denote the rank of $K$, and let $r = \dim_\R T$. Let $\g_\alpha$ denote the $\alpha$-weight subspace of $\g$. Let $N$ and $N_-$ be the opposite unipotent radical subgroups of $G$ with Lie algebras $\nnn = \oplus_{\alpha \in R_+} \g_\alpha$ and $\nnn_- = \oplus_{\alpha \in R_+} \g_{-\alpha}$ respectively.

Let  $\ttt^*_+\subset \ttt^*$ denote the positive Weyl chamber that is generated as a rational polyhedral cone by $\Lambda_+$. 
Throughout, $\ttt^*$ is identified canonically with the subspace $(\kk^*)^T \subset \kk^*$ of points fixed by the coadjoint action of $T$. The quotient map for the coadjoint action of $K$, $\kk^* \to \kk^*/K$, defines a homeomorphism $\ttt_+^* \cong \kk^*/K$. The \emph{sweeping map} is the continuous map  $\mathcal{S} \colon \kk^* \to \kk^*/K \cong \ttt_+^*$.

\subsection{Moment maps and singular symplectic spaces}\label{section; ham grp actions}   
All symplectic manifolds are presumed to be connected. Given a (left) action of $K$ on a smooth manifold $M$ and $\xi \in \kk$,  define the fundamental vector field $\xi_M \in \mathfrak{X}(M)$ with the sign convention so that $\xi \mapsto \xi_M$ is a Lie algebra anti-homomorphism. Our sign convention for the moment map equation of a  Hamiltonian $K$-action on a symplectic manifold $(M,\omega)$ is
$\omega( \xi_M,\cdot) = d\langle \mu,\xi\rangle$.  All moment maps are equivariant. When working with symplectic representations, we always use the quadratic moment map with $\mu(0) = 0$. When working with unitary representations on a finite dimensional Hermitian inner product space $(E,h)$, we always use the linear symplectic form $\omega = -\Im h$. We record the following elementary lemma which will be useful later on.

\begin{lemma}\label{lemma; unitary torus action proper moment map}
    Let a compact torus $T$ act by unitary transformations on a finite dimensional Hermitian inner product space $(E,h)$ and let $\mu \colon E \to \Lie(T)^*$ be the quadratic moment map with $\mu(0) = 0$. 
    \begin{enumerate}[label=(\roman*)]
        % \item  For any connected subtorus $T' \subset T$,
        % let $\ann(\Lie(T')) \subset \Lie(T)^*$ denote the annihilator of $\Lie(T')$ and let $\fix(E,T')$ denote the set of fixed points  for the action of $T'$ on $E$. Then, $\fix(E,T') %= \mathbf{V}(\bigoplus_{\lambda \not\in W} (E^*)_\lambda) 
        %     = \mu\n(\ann(\Lie(T')))$. (In the notation of Section \ref{section; affine G-varieties}, $\mu\n(\overline{\sigma}) = E^\overline{\sigma}$.)
        \item \label{lemma2.2i} Assume that $\mu(E)$ is strongly convex\footnote{I.e., it does not contain any non-trivial subspaces.}. Given a (closed) face $\overline{\sigma}$ of the polyhedral cone $\mu(E)$, let $T^{\overline{\sigma}}$ denote the connected subtorus with $\ann(T^{\overline{\sigma}}) = \mathrm{span}_\R(\overline{\sigma})$ and let $\fix(E,T^{\overline{\sigma}})$ denote the set of fixed points for the action of $T^{\overline{\sigma}}$ on $E$. Then, $\fix(E,T^{\overline{\sigma}}) = \mu\n(\overline{\sigma})$.
        \item \label{lemma2.2ii} The quadratic moment map $\mu\colon E \to \Lie(T)^*$ is proper %\footnote{A map of topological spaces $f \colon X \to Y$ is proper if $f\n(C)$ is compact for every $C \subset Y$ compact.} 
        if and only if $0$ cannot be written as a non-trivial linear combination (with non-negative coefficients) of the real weights of the representation. Equivalently, $\mu(E)$ is strongly convex.
    \end{enumerate}

\end{lemma}

\begin{proof} To establish~\ref{lemma2.2i}: Let $E = \oplus_\lambda E_\lambda$ be the weight decomposition of $E$ with respect to $T$. If $v \in E$ has weight decomposition $\sum_\lambda v_\lambda$, $v_\lambda\in E_\lambda$, then $\mu(v) = \frac{1}{2}\sum_\lambda \lambda ||v_\lambda||^2$. It follows from this formula and strong convexity of $\mu(E)$ that $\mu^{-1}(\overline{\sigma}) = \oplus_{\lambda \in \overline{\sigma}} E_\lambda$. On the other hand, $\fix(E,T^{\overline{\sigma}}) = \oplus_{\lambda \in \ann(T^{\overline{\sigma}})} E_\lambda$. The result follows by definition of $T^{\overline{\sigma}}$. The proof of~\ref{lemma2.2ii} is a straightforward exercise.
\end{proof}

%Let $(X,M,\omega)$ be a decomposed K\"ahler variety and suppose that $(M,\omega)$ is equipped with a Hamiltonian $K$-action with moment map $\mu$. If the action of $K$ restricts to a smooth action on each stratum of $X$, then we say that the restricted action of $K$ on $X$ is  \emph{Hamiltonian} and $\mu\vert_X$ is a \emph{moment map} for the action. Hamiltonian actions on decomposed K\"ahler varieties are examples of Hamiltonian actions on singular symplectic spaces (cf.~Section \ref{section 5}).

Throughout we will deal with Hamiltonian actions on singular spaces.   In this paper, a \emph{singular symplectic space} is a locally compact, paracompact, Hausdorff, second countable topological space $X$ equipped with a locally finite partition by locally closed subspaces $X^\sigma \subset X$ such that each $X^\sigma$ is equipped with the structure of a connected symplectic manifold (and the manifold structure is compatible with the subspace topology). For example, every decomposed K\"ahler variety is a singular symplectic space. Unlike decomposed spaces, our singular symplectic spaces need not satisfy the frontier condition!\footnote{This allows the statement of our constructions and results in Section \ref{section 5} to be much cleaner.}

We denote the symplectic structure on $X^\sigma$ by $\omega^\sigma$. The subspaces $X^\sigma$ are the  \emph{symplectic pieces} of $X$. A \emph{Hamiltonian $K$-space} is a pair $(X,\mu)$ where $X$ is singular symplectic space equipped with a continuous action of $K$ and $\mu \colon X \to \kk^*$ is a continuous map such that for all $\sigma \in \Sigma$, the action of $K$ preserves $X^\sigma$, and $(X^\sigma,\omega^\sigma,\mu\vert_{X^\sigma})$ is a Hamiltonian $K$-manifold. 
% We choose to work with this relatively weak notion of singular symplectic spaces because it has the convenient feature that it is closed under  symplectic reduction, implosion, and contraction.  This feature is useful for the constructions in Section~\ref{section 5}.

\subsection{Affine $G$-varieties}\label{section; affine G-varieties}
Let $G$ be a reductive algebraic group as above. Given an affine $G$-variety $X$,  let $\Lambda(X) \subset \Lambda$ denote the semigroup of highest weights of the $G$-module $\C[X]$ and let $\Gamma(X) = \cone\Lambda(X)$. By~\cite[Corollary 2.9]{brionIntroductionActionsAlgebraic2010}, the semigroup $\Lambda(X)$ is finitely generated. Let $\Sigma(X)$ denote the set of pieces of $\Gamma(X)$ with respect to its decomposition by relative interiors of closed faces, i.e.~each $\sigma \in \Sigma(X)$ is the relative interior of the closed face $\overline{\sigma} \subset \Gamma(X)$. Abbreviate $\Gamma = \Gamma(X)$ and $\Sigma = \Sigma(X)$ when the meaning is clear from context. If $G = H$ is a torus, then let $|f|_\Lambda \in \Lambda$ denote the degree of a $\Lambda$-homogeneous element of $\C[X]$. Abbreviate $|f| = |f|_\Lambda$ when the torus action is clear from context.  

For the rest of this subsection, suppose $G = H$ is a torus. In this case, the decomposition of $\Gamma(X)$ is related to a decomposition of $X$ which will be important in our constructions. Given $\sigma\in \Sigma(X)$, let $T^{\overline{\sigma}} \subset T$ denote the connected subtorus such that $\operatorname{span}_\R (\sigma) = \ann(\Lie(T^{\overline{\sigma}}))$ and let $H^{\overline{\sigma}}$ denote the algebraic subtorus of $H$ with maximal compact torus equal to  $T^{\overline{\sigma}}$. Let $X^{\overline{\sigma}}$ denote the subvariety $\fix(X,H^{\overline{\sigma}})$ of points fixed by $H^{\overline{\sigma}}$ and let
\begin{equation}\label{eq; partition of H variety def}
    X^\sigma = X^{\overline{\sigma}}\backslash \bigcup_{\tau\prec \sigma} X^{\overline{\tau}}.
\end{equation}
Assume that  $X$ is embedded $H$-equivariantly as a closed subvariety of a finite dimensional $H$-module $E$. If $\Lambda(E) = \Lambda(X)$, then it follows by definition that $X^{\overline{\sigma}} = X \cap E^{\overline{\sigma}}$ as varieties for all $\sigma \in \Sigma(X) = \Sigma(E)$. In fact, slightly more is true.

\begin{lemma}\label{lem; scheme theoretic intersection is reduced}
Let $X$, $E$, and $H$ be as in the previous paragraph. Let $\overline{\sigma}$ be a closed face of the polyhedral cone $\Gamma(X)$ associated with the action of the torus $H$ on $X$. Then, the scheme-theoretic intersection $X \cap E^{\overline{\sigma}}$ is reduced, i.e.~$X^{\overline{\sigma}} = X \cap E^{\overline{\sigma}}$ as schemes.
\end{lemma}
\begin{proof}
We want to show that ${\bf I}(X) + {\bf I}(E^{\overline{\sigma}})$ is a radical ideal. Let $f \in \C[E]$ such that  $f^k\in {\bf I}(X) + {\bf I}(E^{\overline{\sigma}})$ for some $k$. 
We must show that $f \in {\bf I}(X) + {\bf I}(E^{\overline{\sigma}})$. 
It is a straightforward exercise to show that it suffices to prove the Lemma for $\Lambda$-homogeneous $f$.

Let us fix a basis $\mathcal{K} = \{z_1,\dots, z_J\}$ of $E^*$ consisting of $\Lambda$-homogeneous elements. Note that 
\begin{equation}\label{useful fact about subspaces}
    \mathbf{I}(E^{\overline{\sigma}}) = (z_j \in \mathcal{K} \mid |z_j| \notin \overline{\sigma}).
\end{equation}
Suppose $f^k$ is homogeneous. Then $f^k = g + h$, for some homogeneous $g \in {\bf I}(X)$ and $h \in {\bf I}(E^{\overline{\sigma}})$. Write $f^k = \sum_{l=1}^{L} M_{l}$ where each $M_{l}\in \C[E]$ is a homogeneous monomial in $z_1,\dots,z_n$, of degree $|M_{l}| = |f^k|$. If $|f^k|\in \overline{\sigma}$, then by~\cite[Lemma 1.2.7]{coxToricVarieties2011} it follows that each $z_j$ appearing in each $M_l$ satisfies $|z_j| \in \overline{\sigma}$. But, by~\eqref{useful fact about subspaces}, we see that the only way this can happen is if $h=0$, so $f^k \in {\bf I}(X)$. Since ${\bf I}(X)$ is radical, we have that $f\in {\bf I}(X)\subset {\bf I}(X) + {\bf I}(E^{\overline{\sigma}})$. On the other hand, if  $|f^k|\notin \overline{\sigma}$ then, since $\overline{\sigma}$ is closed under addition, each $M_l$ of $f^k$ contains some $z_j$ with $|z_j| \notin \overline{\sigma}$. By~\eqref{useful fact about subspaces}, it follows that $f^k\in {\bf I}(E^{\overline{\sigma}})$. This ideal is also radical, so $f\in {\bf I}(E^{\overline{\sigma}}) \subset {\bf I}(X)+{\bf I}(E^{\overline{\sigma}})$.
\end{proof}

\subsection{Affine toric varieties} We record some notation regarding affine toric varieties. An \emph{affine semigroup} is a subset of a lattice $L$ that is closed under addition, contains $0$, and is finitely generated. Given an affine semigroup $\mathsf{S}$, let  $\C[\mathsf{S}]$ denote the semigroup algebra of $\mathsf{S}$ and let  $X_\mathsf{S}$ denote the associated affine toric variety.

Let $H$ denote the complex algebraic torus whose lattice of real weights is $L$, and let $T$ be the maximal compact torus in $H$. Let $C\subset L_\R$ be a $\operatorname{rank}(L)$-dimensional rational convex polyhedral cone and let $\mathsf{S} = C\cap L$. Then $X_\mathsf{S}$ is normal.  Let $F \subset C$ be a closed face, let $T^F\subset T$ denote the connected subtorus with $\operatorname{span}_\R (F) = \ann(\Lie(T^F))$, and let $H^F$ denote the complex subtorus of $H$ with maximal compact torus equal to $T^F$.  The subvariety $X^F_\mathsf{S} = \fix(X_\mathsf{S},H^F)$ is an $H$ orbit-closure. As a toric variety, $X^F_\mathsf{S}$ is isomorphic to $X_{\mathsf{S}^F}$ where $\mathsf{S}^F = \mathsf{S}\cap F$.  
% The following can be proved by combining Lemma~\ref{lemma; unitary torus action proper moment map}\ref{lemma2.2i}, \cite[Theorem~4.9]{sjamaarConvexityPropertiesMoment1998}, and \cite[Proposition~11.1.2]{coxToricVarieties2011}.

% \begin{lemma}\label{lemma; toric variety moment map} With everything  defined as in the preceding paragraph, suppose additionally that: $C$ is strongly convex, $H$ acts linearly on a Hermitian inner product space $(E,h)$; $X_\mathsf{S}$ is embedded $H$-equivariantly as a subvariety of $E$; and the action of the maximal compact torus $T \subset H$ on $(E,h)$ is unitary. Let $\Psi$ denote the quadratic moment map for this action with $\Psi(0) = 0$. Then: 
% \begin{enumerate}[label=(\roman*)]
%     \item $\Psi(X^{F}_\mathsf{S})= F$ and $X^{F}_\mathsf{S} = \Psi\n(F)\cap X_\mathsf{S}$.
%     \item The image under $\Psi$ of the smooth locus of $X^{F}_\mathsf{S}$ is the smooth locus of  $F$.
% \end{enumerate}
% \end{lemma}

\section{Stratified gradient Hamiltonian flows}\label{section; Toric degenerations and gradient Hamiltonian flows}

This section concerns a gradient Hamiltonian vector field, placed upon on a degeneration of a decomposed K\"ahler variety. We describe a set of conditions \ref{assumeA}--\ref{assumeG} under which this vector field may be integrated to a flow, and under which this flow may be extended continuously across the special fiber of the degeneration. There are two main ways in which our approach extends previous results. First, we do not assume that we work with compact varieties, and so the flow of this vector field may blow up before it reaches the special fiber. Second, because our varieties are decomposed, the gradient Hamiltonian vector field is defined piece-wise, and  the associated flow is not continuous \emph{a priori}. The main result of this section (Theorem \ref{thm; toric degenerations main theorem}) describes how to overcome both of these difficulties.

We briefly summarize the key ideas and associated notation for Theorem \ref{thm; toric degenerations main theorem}, which may be used as a road map for reading Section \ref{section; Toric degenerations of stratified affine varieties}. We start with a toric degeneration $\X$ whose 1-fiber, $\X_1$, is a decomposed K\"ahler variety with pieces denoted by $\X_1^\sigma$ for $\sigma$ in an index set $\Sigma$. This degeneration is required to satisfy assumptions \ref{assumeA} -- \ref{assumeG}. From this we define a piecewise gradient Hamiltonian vector field on $\X$. For each $\sigma$ this determines a flow $\varphi_t^\sigma$ which takes points in $\X_1^\sigma$ to points in $\X_{1-t}^\sigma$. These assemble into a continuous function $\varphi_t$ from $\X_1$ to $\X_{1-t}$. Although the time-1 flow $\varphi_1$ is not necessarily defined at all points on $\X_1$ because of the singular nature of $\X_0$, the \emph{limit} as $t\to 1^-$ of $\varphi_t$ exists for every point in $\X_1$ and defines a continuous map to $\X_0$. The image of each piece $\X_1^\sigma$ under this limiting map lies in a toric subvariety $\X_0^\sigma \subset \X_0$. The toric subvarieties $\X_0^\sigma$ define a partition of $\X_0$ but unlike $\X_1^\sigma$ they are not necessarily smooth. Provided that the toric action on $\X_0$ is generated by a moment map, its pullback defines an integrable system on $\X_1$ (see Corollary \ref{thm; toric degenerations main theorem corollary}). 
% \begin{equation}\label{eq; idea of section 3}
%     \begin{tikzcd}
%         D^\sigma \ar[d, "\varphi_1^\sigma"] \ar[r, hookrightarrow] & \X_1^\sigma\ar[rr, hookrightarrow]\ar[d, "\lim_{t\to 1^-}\varphi_t^\sigma"] && \X_1 \ar[d, "\lim_{t\to 1^-}\varphi_t"] \\
%         U_0^\sigma \ar[r, hookrightarrow] & \X_0^\sigma \ar[rr, hookrightarrow] && \X_0
%     \end{tikzcd}
% \end{equation}

This construction has an additional feature which is crucial for our applications in Section \ref{section 5}. For each piece $\X_1^\sigma$ there is an open dense subset $D^\sigma \subset \X_1^\sigma$ on which the flow $\varphi_1^\sigma$ \emph{is} defined and yields a \emph{symplectic isomorphism} onto the smooth locus $U_0^\sigma$ of the toric subvariety $\X_0^\sigma$ (illustrated in the bottom square of the following diagram).
\begin{equation}\label{eq; idea of section 3}
    \begin{tikzcd}
        \X_1 \ar[rr, "\lim_{t\to 1^-}\varphi_t"] && \X_0 \\
        \X_1^\sigma\ar[u, phantom, sloped, "\subset"]\ar[rr, "\lim_{t\to 1^-}\varphi_t^\sigma"] && \X_0^\sigma \ar[u, phantom, sloped, "\subset"] \\
        D^\sigma \ar[rr, "\varphi_1^\sigma"] \ar[u, phantom, sloped, "\subset"] && U_0^\sigma \ar[u, phantom, sloped, "\subset"]
    \end{tikzcd}
\end{equation}
Thus, whereas Harada-Kaveh produces a symplectic isomorphism on an open dense subset of their smooth variety, our result produces a symplectic isomorphism on an open dense subset of \emph{each piece} of our decomposed variety. In particular, the integrable system resulting from pullback defines a toric action on each $D^\sigma$ (see Definition \ref{def; integrable system on singular symplectic space} where we define integrable systems on decomposed K\"ahler varieties). 

This section is organized as follows. Section \ref{section; gradient Hamiltonian vector fields and toric degenerations} gives basic notions about gradient Hamiltonian vector fields in the stratified setting, and Section~\ref{section;degenerations background} recalls the definition of a degeneration. 
The main result of this section, Theorem \ref{thm; toric degenerations main theorem}, is described in Section \ref{section; Toric degenerations of stratified affine varieties}.
Its proof is given in Appendix \ref{section; proof of toric deg theorem}. Section~\ref{section;GH flows on affine varieties} establishes several results about gradient Hamiltonian vector fields on affine varieties which will be used in the sequel.

\subsection{Gradient Hamiltonian vector fields}\label{section; gradient Hamiltonian vector fields and toric degenerations}

We recall the definition and elementary properties of gradient Hamiltonian vector fields.  We refer the reader to \cite[Section 2.2]{haradaIntegrableSystemsToric2015} for more details.  

Let $M$ be a K\"ahler manifold and let $\pi \colon M \to \C$ be a holomorphic submersion. Let $\Re \pi, \Im \pi\colon M\to \R$ denote the real and imaginary parts of $\pi$, i.e.~$\pi = \Re \pi + \sqrt{-1}\Im \pi$. Let $\nabla(\Re\pi)$ denote the gradient vector field of $\Re\pi$ with respect to the K\"ahler metric and let $X_{\Im\pi}$ denote the Hamiltonian vector field of $\Im\pi$ with respect to the K\"ahler form.  Since $\pi$ is holomorphic and $M$ is K\"ahler, it follows that $\nabla(\Re\pi) = - X_{\Im\pi}$.
The \emph{gradient Hamiltonian vector field} of $\pi$ is
\begin{equation}\label{eq; grad ham vf def}
	V_\pi := \frac{X_{\Im \pi}}{||X_{\Im \pi}||^2} = -\frac{\nabla \Re\pi}{||\nabla \Re\pi||^2}
\end{equation}
where $||\cdot ||$ denotes norm with respect to K\"ahler metric. The vector field $V_\pi$ is defined everywhere on $M$ since $\pi$ is a submersion, and therefore $\nabla(\Re\pi)$ is non-vanishing.

Let $M_z = \pi\n(z)$ denote the fiber of $\pi$ over $z\in \C$.  Let $\varphi_t(x)$ denote the flow of $V_\pi$ through $x\in M$ at time $t$. If $x \in M_z$ and $\varphi_t(x)$ is defined, then $\varphi_t(x) \in M_{z-t}$. Thus, if $\varphi_t(x)$ is defined for all $x \in M_z$, then it determines a map 
$
    \varphi_t\colon M_z\to M_{z-t}.
$

Since $\pi$ is a holomorphic submersion, each $M_z$ is a smooth K\"ahler submanifold of $M$ and $\dim_\C M_z = \dim_\C M -1$. The following fact is well-known, see e.g.~\cite[Proposition 2.3]{haradaIntegrableSystemsToric2015}. 

\begin{lemma}\label{lemma; flow is symplectic}
If it is defined, the map $\varphi_t\colon M_z\to M_{z-t}$ is symplectic with respect to the restricted K\"ahler forms on $M_z$ and $M_{z-t}$.
\end{lemma}

We now extend the definition of gradient Hamiltonian vector fields to a stratified setting. Let $M$ be a K\"ahler manifold as before. Let $Y$ be a decomposed space and suppose $Y$ is embedded in $M$ so that each smooth piece $Y^\sigma$ is a submanifold. The \emph{stratified tangent bundle} of $Y$ is the disjoint union of tangent bundles
$
    TY = \bigcup_{\sigma \in \Sigma} TY^\sigma.
$
It inherits a subspace topology from $TM$. See e.g.~\cite[Section 2.1]{pflaumAnalyticGeometricStudy2001} for more details.

Suppose $\pi \colon M \to \C$ is a holomorphic map, each $Y^\sigma$ is a K\"ahler submanifold of $M$, and each restricted map $\pi\colon Y^\sigma \to \C$ is a submersion. For each $\sigma \in \Sigma$, let  $V_\pi^\sigma \colon Y^\sigma \to TY^\sigma$ denote the gradient Hamiltonian vector field of the holomorphic submersion $\pi\colon Y^\sigma \to \C$.  
\begin{definition} The \emph{stratified gradient Hamiltonian vector field} on $Y$ is the section
\begin{equation}\label{eq; general stratified gh vf definition}
    V_\pi \colon Y \to  TY, \quad V_\pi(x) = V_\pi^\sigma(x) \text{ for } x \in Y^\sigma.
\end{equation}
\end{definition}
Note that $V_\pi$ may fail to be continuous.  It is a \emph{stratified vector field} in the sense of \cite[2.1.5]{pflaumAnalyticGeometricStudy2001}.

For $x\in Y$, consider the initial value problem 
\begin{equation}\label{eq; IVP}
    \frac{d}{dt}\varphi_t(x) = V_\pi(\varphi_t(x)), \quad \varphi_0(x) = x.
\end{equation}
A solution of this initial value problem is given by combining the flows of the vector fields $ V_\pi^\sigma$ in a piecewise manner. Let $\varphi_t^\sigma$ denote the flow of $V_\pi^\sigma$. Define 
\begin{equation}\label{eq; stratified flow def}
   \varphi_t(x) = \varphi_t^\sigma(x) \quad \text{if  $\varphi_t^\sigma(x)$ is defined}.
\end{equation}
% Here $(Y^{\sigma})^{sm}$ denotes the smooth locus of $Y^\sigma$.
We call $\varphi_t(x)$ the \emph{stratified gradient Hamiltonian flow}. If $x\in Y \cap M_z$ and $\varphi_t(x)$ is defined, then $\varphi_t(x) \in Y \cap M_{z-t}$.
%Provided that each  $\varphi_t^\sigma$ can be integrated to time $t$, this defines a map (of sets) $\varphi_t\colon Y_z \to Y_{z-t}$. 

Because each stratum $Y^\sigma$ may be non-compact, it will become necessary to establish in some cases that the flows $\varphi_t^\sigma$ exist. This becomes possible in the presence of additional group symmetry. The key lemma is a straightforward application of Noether's theorem. Similar versions of this lemma previously appeared in, e.g. \cite[Section 2.6]{haradaIntegrableSystemsToric2015} and \cite[Lemma 5.1]{hilgertContractionHamiltonianKSpaces2017}.

\begin{lemma}\label{lemma; ham grp action preserving pi}
	Let $M$ be a K\"ahler manifold with K\"ahler form $\omega$ and K\"ahler metric $g$, and let $\pi \colon M \to \C$ be a holomorphic submersion. Assume there is a Hamiltonian action of a connected Lie group $K$ on $(M,\omega)$ with moment map $\psi\colon M \to \kk^*$ such that the action of $K$ preserves the fibers of $\pi$ and the K\"ahler metric $g$. 
	Then, as long as it exists, the flow of $V_\pi$ is $K$-equivariant and preserves fibers of $\psi$.
\end{lemma}

\subsection{Degenerations} \label{section;degenerations background}

Let $\X$ be a variety and let $\pi\colon \X \to \C$ be a morphism.  Denote the fiber of $\pi$ over $z\in \C$ by $\X_z$. Let $Z \subset \X$ denote the union of the singular set of $\X$ and the critical set of $\pi|_{\X^{sm}}$, where $\X^{sm}$ denotes the smooth locus of $\X$.  Then  $\X \setminus Z$ is a complex manifold and $\pi|_{\X\setminus Z}  \colon \X \setminus Z \to \C$ is a holomorphic map. For all $z \in \C$, denote $U_z = \X_z\setminus Z$. Since $\pi \colon \X \setminus Z \to \C$ is a submersion, each $U_z$ is a complex submanifold of $\X\setminus Z$ of complex codimension 1.

We recall the following definition from \cite[Definition 2]{haradaIntegrableSystemsToric2015}.

\begin{definition}\label{def; HK toric degeneration}
	A \emph{degeneration} of a variety $X$ is a map $\pi\colon \X \to \C$ such that:
	\begin{enumerate}[label=(\roman*)]
		\item There is an algebraic isomorphism $\rho\colon X\times\C^\times \to \X\setminus \X_0$ such that $\pi \circ \rho = \pr_2$. Such an isomorphism is called a \emph{trivialization away from 0}.
		\item The fiber $\X_0$ is non-empty.
		\item $\pi$ is a flat family of varieties, i.e. it is a flat morphism and its fibers are all reduced as schemes.
	\end{enumerate}
	A degeneration $\pi\colon \X \to \C$ is a \emph{toric degeneration} if  $\X_0$ is a toric variety.
\end{definition}

If $\pi \colon \X \to \C$ is a degeneration, then $\pi|_{\X\setminus Z} \colon \X \setminus Z \to \C$ is a submersion onto $\C$. If the variety $X$ is smooth, then it follows by trivialization away from $0$ that $Z$ is contained in $\X_0$. 
In this case, $U_z = \X_z$ for all $z\neq 0$. More generally, we have the following.

\begin{proposition}\cite[Corollary 2.10]{haradaIntegrableSystemsToric2015}\label{cor; HK cor 2.10}
	Let $\pi \colon \X \to \C$ be a degeneration of $X$. Then for all $z\in \C$, $U_z$ is precisely the smooth locus of $\X_z$. In particular, $U_z$ is dense in $\X_z$.
\end{proposition}

\subsection{Stratified gradient Hamiltonian vector fields}\label{section; Toric degenerations of stratified affine varieties}

Let $(X,M,\omega)$ be a decomposed K\"ahler variety (Definition \ref{def; decompsed Kahler variety}) and let $\pi\colon \X \to \C$ be a degeneration of $X$. The trivialization $\rho$ and the decomposition of $X$ allow us to define subfamilies
\begin{equation}\label{eq; definition of subfamilies}
	\X^{\overline{\sigma}} := \overline{\rho(X^\sigma \times \C^\times)}, \quad  \X^{\sigma} := \X^{\overline{\sigma}} \setminus \bigcup_{\tau \prec \sigma} \X^{\overline{\tau}}
\end{equation}
where closure is with respect to the Zariski topology. The trivialization also yields isomorphisms $\X^{\overline{\sigma}} \setminus \X_0 \cong \overline{X^\sigma} \times \C^\times$ and $\X^{\sigma} \setminus \X_0 \cong X^\sigma \times \C^\times$. Denote 
\begin{equation}\label{eq; definition of subvarieties}
    \X_z^{\overline{\sigma}} := \X_z \cap \X^{\overline{\sigma}}, \quad \text{and} \quad \X_z^{\sigma} := \X_z \cap \X^{\sigma}
\end{equation}
for all $z\in \C$. Note that $\X^{\overline{\sigma}}$, $\X^\sigma$, $\X^{\overline{\sigma}}_z$, and  $\X^\sigma_z$ are subvarieties of $\X$ for all $\sigma$ and all $z\in \C$.  

We now give a list of assumptions \ref{assumeA}--\ref{assumeG} that we will place on our degenerations. 
These are partially inspired by Harada and Kaveh's assumptions (a)--(d)  \cite[p. 932]{haradaIntegrableSystemsToric2015}.  
\begin{enumerate}[label=(GH\arabic*)]
	\item \label{assumeA} For all $\sigma$, the restricted map $\pi\colon \X^{\overline{\sigma}} \to \C$ is a flat family of varieties.
	\item \label{assumeB} The family $\X$  is embedded into $M\times \C$ as a subvariety such that the map $\pi\colon \X \to \C$ coincides with  restriction to $\X$ of the projection $M\times\C \to \C$.
	\item \label{assumeC} The embedding $X = X\times \{1\} \cong_\rho \X_1 \subset M \times \{1\} =M$ given by Assumption \ref{assumeB} coincides with the embedding $X \hookrightarrow M$ of the decomposed K\"ahler variety $(X,M,\omega_M)$.
\end{enumerate}
For each $\sigma$, define $Z^\sigma \subset \X^\sigma$ to be the union of the critical set of $\pi\colon \X^\sigma \to \C$ and the singular set of $\X^\sigma$. Denote $U_0^\sigma = \X_0^\sigma \setminus  Z^\sigma$. It follows by Assumption \ref{assumeA} and Proposition \ref{cor; HK cor 2.10} that $U_0^\sigma$ is precisely the smooth locus of $\X_0^\sigma$, i.e.
$
    U_0^\sigma = (\X_0^\sigma)^{sm}
$
(see Lemma~\ref{lemma;immediate consequences}). 

Equip $\C$ with its standard K\"ahler structure and $M\times \C$ with the product K\"ahler structure. In particular, the K\"ahler form is the product symplectic structure $\omega = \omega_M\oplus \omega_{\rm{std}}$. 

\begin{enumerate}[label=(GH\arabic*), start=4]
	\item \label{assumeE} There is a Hamiltonian action of a compact torus $T$ on $M\times \C$ with moment map $\psi\colon M \times \C \to \ttt^*$ such that:
	\begin{enumerate}[label=\Roman*)]
		\item The action of $T$ on $M\times \C$ preserves each of the subvarieties $\X^{\sigma}\setminus Z^\sigma$, the fibers of $\pi$, and the K\"ahler metric on $M\times \C$.
		\item The map $(\pi,\psi)\colon \X \to \C\times \ttt^*$ is proper.
		\item The subfamilies $\X^\sigma$ are saturated\footnote{A subset $A \subset X$ is \emph{saturated} by a map $f\colon X \to Y$ if it is a union of fibers of $f$.} by the restricted maps $\psi\colon \X \to \ttt^*$.
	\end{enumerate}
\end{enumerate}
Note that by assumptions \ref{assumeE}II) and \ref{assumeE}III), the map $(\pi,\psi)\colon \X^\sigma \to \C\times \ttt^*$ is proper as a map to its image for every subfamily $\X^\sigma$.

Denote the restriction of $\omega$ to the symplectic submanifolds $\X^\sigma\setminus Z^\sigma$, $\X^\sigma_z$ for $z\neq 0$, and $U_0^\sigma$ by $\omega^\sigma$, $\omega^\sigma_z$, and $\omega^\sigma_0$, respectively. The action of $T$ preserves  each of these submanifolds. The restricted action is therefore also Hamiltonian with moment map given by the restriction of $\psi$. The following assumption is sufficient to conclude that the time-1 gradient Hamiltonian flow on $\X^\sigma\setminus Z^\sigma$ defines a symplectomorphism from an open dense subset of $(\X^\sigma_1,\omega^\sigma_1)$ onto $(U_0^\sigma,\omega_0^\sigma)$.

\begin{enumerate}[label=(GH\arabic*), start=5]
	\item \label{assumeF} The Duistermaat-Heckman measures of $(U_0^\sigma,\omega_0^\sigma,\psi)$ and $(\X^\sigma_1,\omega^\sigma_1,\psi)$ are equal for all $\sigma \in \Sigma$. In particular, $U_0^\sigma$ is non-empty.
\end{enumerate}

Because $U_0^\sigma$ is nonempty for all $\sigma \in \Sigma$, each subfamily $\pi\colon \X^{\overline{\sigma}} \to \C$ is a degeneration.

The partition of $\X \setminus \X_0$  by the manifolds $\X^\sigma\setminus \X_0^\sigma = \rho(X^\sigma \times \C^\times)$ is a decomposition. Each $\X^\sigma\setminus \X_0^\sigma$ is a K\"ahler submanifold of $M\times \C$ and the restricted maps  $\pi\colon \X^\sigma\setminus \X_0^\sigma \to \C$ are holomorphic submersions.  
Thus, we may define a stratified gradient Hamiltonian vector field as in Section \ref{section; gradient Hamiltonian vector fields and toric degenerations},
\begin{equation}\label{eqn; stratified gradient hamiltonian vector field}
    V_\pi \colon \X \setminus \X_0 \to  T(\X \setminus \X_0), \quad V_\pi(x) = V_\pi^\sigma(x) \text{ for } x \in \X^\sigma\setminus \X_0^\sigma.
\end{equation}
The following is sufficient to prove that its gradient Hamiltonian flow is continuous.
\begin{enumerate}[label=(GH\arabic*), start=6]
	\item \label{assumeG} The stratified vector field $V_\pi\colon  \X \setminus \X_0 \to  T(\X \setminus \X_0)$ is continuous.
\end{enumerate}

A first consequence of these axioms is the following, the proof of which is an easy application of Lemma~\ref{lemma; ham grp action preserving pi}, together with \ref{assumeE}.
 
 \begin{proposition}
 \label{proposition;flow exists}
The flow $\varphi^\sigma_t \colon \X^\sigma_1 \to \X^\sigma_{1-t}$ exists, for all $0<t<1$ and all $\sigma \in \Sigma$.
 \end{proposition}

We now state the main result of this section. Its proof is given in Appendix \ref{section; proof of toric deg theorem}.  

\begin{theorem}\label{thm; toric degenerations main theorem}
	Let $\pi\colon \X \to \C$ be a  degeneration of a decomposed K\"ahler variety $(X,M,\omega_M)$ that   satisfies assumptions \ref{assumeA}--\ref{assumeG} above. Let $T$ and $\psi$ be the compact torus and moment map of Assumption \ref{assumeE}. 
	
	Then, for all $x\in \X_1$, the limit
	\[
	    \phi(x) = \lim_{t \to 1^-}\varphi_t(x)
	\]
	exists and defines a continuous, $T$-equivariant, proper, surjective  map $\phi \colon \X_1 \to \X_0$. Moreover:
	\begin{enumerate}[label=(\alph*)]
	    \item For all $\sigma\in \Sigma$ and $x \in U_0^\sigma$, the flow $\varphi_{-1}^\sigma(x)$ exists, $D^\sigma := \varphi_{-1}^\sigma(U_0^\sigma)$ is a dense open subset of $\X^\sigma_1$, and 
	    \[
	        \phi\vert_{D^\sigma} = \varphi_1^\sigma\colon (D^\sigma,\omega_1^\sigma) \to (U_0^\sigma,\omega_0^\sigma)
	   \]
	    is a symplectomorphism. 
	    \item $\psi\circ \phi = \psi$.
	\end{enumerate} 
\end{theorem}

\begin{remark}
    The construction of the map $\phi$ has two main components: proving that $\varphi_t\colon \X_1 \to \X_{1-t}$ is continuous when $t<1$, and  proving that the limit $\lim_{t \to 1^-}\varphi_t(x)$ exists and is continuous.
    The proof that $\varphi_t$ is continuous for $t<1$ relies  on assumptions \ref{assumeE} and \ref{assumeG}. % and an application of Noether's theorem (Lemma \ref{lemma; ham grp action preserving pi}).  
    There are various other frameworks for studying flows of vector fields on stratified spaces (such as Mather's control theory %\cite{matherNotesTopologicalStability2012} 
    or the notion of rugose vector fields), %\cite[Definition (1.4)]{verdierStratificationsWhitneyTheoreme1976}
     which we did not use in our proof.
    The proof that the limit $\lim_{t \to 1^-}\varphi_t(x)$ exists and is continuous follows the same outline as the proof of \cite[Theorem 2.12]{haradaIntegrableSystemsToric2015}, along with an application of Noether's theorem (Lemma \ref{lemma; ham grp action preserving pi}).
\end{remark}

Our primary application of Theorem \ref{thm; toric degenerations main theorem} is to toric degenerations. If the zero fiber of the degeneration $\X$ carries a Hamiltonian action of a torus $\T$ generated by a moment map $\Psi\colon \X_0 \to \Lie(\T)^*$, then the composition of $\Psi$ with $\phi\colon X = \X_1 \to \X_0$ generates a torus action on an open dense subset of each piece of $X$ as follows.

\begin{corollary} \label{thm; toric degenerations main theorem corollary}
    Let $\pi\colon \X \to \C$ be a toric degeneration of a decomposed K\"ahler variety $(X,M,\omega_M)$ that satisfies assumptions \ref{assumeA}--\ref{assumeG}, let $\phi \colon \X_1 \to \X_0$ denote the map constructed as in Theorem \ref{thm; toric degenerations main theorem}, and let $\T$ denote the compact torus of the toric variety $\X_0$. Assume that:
    \begin{enumerate}[label=(\roman*)]
        \item The action of $\T$ on $\X_0$ extends to an action on $M\times\{0\} = M$ that is Hamiltonian with moment map $\Psi\colon M \to \Lie(\T)^*$.
        \item The subvarieties $U_0^\sigma$ are $\T$-invariant.
    \end{enumerate}
	Then, for each $\sigma \in  \Sigma$, 
	    the restriction of $\Psi\circ \phi$ to $D^\sigma$ is a moment map for a complexity 0 Hamiltonian $\T$-action on $(D^\sigma,\omega_1^\sigma)$. 
\end{corollary}

In \cite[Definition 2.1]{haradaIntegrableSystemsToric2015}, Harada and Kaveh give a somewhat non-standard definition of  \emph{completely integrable systems} on singular varieties whose smooth locus is equipped with a symplectic structure.  We now give a similar definition that is more suitable to the present setting. First, we define a collection of real valued continuous functions $f_1,\dots, f_n$ on a smooth connected symplectic  manifold $M$ to be a \emph{completely integrable system} if: 
    \begin{enumerate}[label=(\roman*)]
        \item There exists an open dense subset $D \subset M$ such that the restricted functions $f_i\vert_{D}$ are all smooth and the rank of the Jacobian of $F = (f_1,\dots,f_n)$ equals $\frac{1}{2}\dim(M)$ on a dense subset of $D$.
        \item The restricted functions $f_i\vert_{D}$ pairwise Poisson commute, i.e.~$\{f_i\vert_{D},f_j\vert_{D}\} = 0$ for all $1\leq i,j\leq n$.
    \end{enumerate}
If the functions $f_1, \dots, f_n$  satisfy condition (i) but not condition (ii), then they form an \emph{integrable system}.
Corollary \ref{thm; toric degenerations main theorem corollary} produces a completely integrable system on a decomposed K\"ahler variety in the following sense. 

\begin{definition}\label{def; integrable system on singular symplectic space}
   A collection of real valued continuous functions $f_1,\dots ,f_n$ on a decomposed K\"ahler variety (or, more generally, a singular symplectic space, cf.~Section~\ref{section 5}) $X$ is a \emph{(completely) integrable system} if their restriction to each piece of $X$ defines an (completely) integrable system in the sense defined above. 
\end{definition}

\begin{remark}\label{rem; hk as special case}
    Suppose $X$ is a smooth variety equipped the trivial decomposition, $M$ is a projective space equipped with the Fubini-Study K\"ahler form,  the torus $T$ is trivial, and $\X$ is a toric degeneration. Then assumptions \ref{assumeE}--\ref{assumeG} are satisfied automatically. In particular, \ref{assumeF} reduces to the statement that the symplectic volumes of $U_0$ and $X$ are equal, which follows by flatness of $\pi\colon\X\to\C$ (see the proof of \cite[Corollary 2.11]{haradaIntegrableSystemsToric2015}).
    Thus, Corollary \ref{thm; toric degenerations main theorem corollary} reduces to  \cite[Theorem A]{haradaIntegrableSystemsToric2015}.
\end{remark}

\subsection{Gradient Hamiltonian flows on decomposed affine K\"ahler varieties} \label{section;GH flows on affine varieties}

We now turn our attention to gradient Hamiltonian flows on affine varieties. The first main result of this section is Proposition~\ref{CTSvfprop}, which is a useful tool for verifying that~\ref{assumeG} holds. An application of this result is Theorem~\ref{GHmoser}, which says that often the symplectic structure on a decomposed affine K\"ahler variety is independent of the embedding into an ambient affine space.  Throughout this section, $H$ is an algebraic torus with maximal compact torus $T$ and real weight lattice $\Lambda$.

We briefly recall the Whitney A condition from~\cite[1.4.3]{pflaumAnalyticGeometricStudy2001}, which we will need below. Given a $k$-dimensional submanifold $N$ of a smooth manifold $M$ and a sequence of points $\{x_i\}_{i\in \N}\subset N$, let $\lim_{i\to \infty}T_{x_i}N$ denote the limit of tangent spaces in the Grassmannian of $k$-dimensional subspaces of $TM$. A pair $(N',N)$ of submanifolds of $M$ satisfies the Whitney condition (A) if:
\begin{enumerate}[label=(A)]
    \item For any sequence $\{x_i\}_{i\in \N} \subset N$ that converges to some $x \in N'$, if $\lim_{i\to \infty}T_{x_i}N$ exists, then $T_xN' \subset \lim_{i\to \infty}T_{x_i}N$.\label{Whitney A}
\end{enumerate}
A decomposed space satisfies the \emph{Whitney condition (A)} if each pair of its pieces satisfies the Whitney condition (A).

\subsubsection{Continuity of gradient Hamiltonian vector fields}
Let $E$ be a finite dimensional $H$-module. Extend the action of $H$ to $E\times \C$ by letting $H$ act trivially on $\C$. Let $\X$ be a $H$-invariant closed subvariety of $\X \subset E\times \C$. Import all the notation from Section \ref{section; affine G-varieties}. 
Let $\pi$ denote the projection $E\times \C \to \C$ as well as its restriction to $\X$. Denote $\X_0 = \pi\n(0) \cap \X$. Assume:
\begin{enumerate}[label=(D\arabic*)]
\item The partition of $\X \backslash \X_0$ into $\X^\sigma\backslash \X_0$, $\sigma\in \Sigma(\X)$, gives $\X\backslash \X_0$ the structure of a decomposed variety. \label{D1}
% \end{enumerate}

% We also assume:

% \begin{enumerate}[label=(D\arabic*), start=2]
\item The decomposition of $\X\setminus\X_0 $ in \ref{D1} satisfies the Whitney condition (A) with respect to the embedding into $E\times \C$. \label{D2}
\end{enumerate}

Suppose $E$ is equipped with  a Hermitian inner product $h_E$ and the action of $T$ is unitary. Equip $E\times \C$ with the Hermitian inner product $h_E\oplus h_\C $ (where $h_\C$ is the standard Hermitian inner product). This equips the submanifolds $\X^\sigma \setminus \X_0$, $\sigma \in \Sigma$, with a K\"ahler structure. Assume that the restricted maps $\pi \colon \X^\sigma \setminus \X_0 \to \C$ are submersions for all $\sigma \in \Sigma$. Then, we may define a stratified gradient Hamiltonian vector field  $V_\pi\colon \X\backslash \X_0 \to T(\X\backslash \X_0)$ as in~\eqref{eq; general stratified gh vf definition}.  The goal of this section is to prove the following.

\begin{proposition} \label{CTSvfprop} In the preceding context (everything from the beginning of this subsection, including the assumptions~\ref{D1} and~\ref{D2}),
 $V_\pi\colon \X\backslash \X_0 \to T(\X\backslash \X_0)$ is continuous.
\end{proposition}

We first establish a preliminary result. Throughout, identify $E\times \C = T_x (E\times \C)$ for all $x\in E\times \C$.  Limits of tangent spaces are taken within the appropriate Grassmannian.

\begin{lemma} \label{lemma;decomplemma}
Let $\sigma, \tau\in \Sigma(\X)$ with $\sigma\prec\tau$. Let $\{x_j\} \subset \X^\tau\setminus \X_0$ be a sequence of points converging to $x\in \X^\sigma\setminus \X_0$.  If $\lim_{j\to \infty} T_{x_j} \X^\tau$ exists, then
$
\lim_{j\to \infty} T_{x_j} \X^\tau \subset T_x \X^{\sigma} \oplus (E^{\overline{\sigma}}\times \C)^\perp.
$
\end{lemma}

\begin{proof} Let $\hat{g}_1,\dots,\hat{g}_J$ be a set of $\Lambda$-homogeneous generators of ${\bm I} (\X)$. The tangent space of $\X$ at $x$ is
\[
    T_x\X = \{v \in E\times \C \mid (d\hat{g}_j)_x(v) = 0 \text{ for all }j\in [1,J]\}.
\]
Since 
$
    \lim_{j\to \infty} T_{x_j} \X^\tau \subset T_x \X,
$
it suffices to show $T_x \X\subset T_x \X^{\sigma} \oplus (E^{\overline{\sigma}}\oplus \C)^\perp$. 

By  Lemma \ref{lem; scheme theoretic intersection is reduced} (and since $\X^\sigma\setminus \X_0$ is an open subset of $\X^{\overline{\sigma}}\setminus \X_0$), 
\[
    T_x\X^\sigma = T_x \X^{\overline{\sigma}} = T_x\X \cap (E^{\overline{\sigma}}\times \C).
\]  
It follows by  \eqref{useful fact about subspaces} that if $|\hat{g}_j|\notin \overline{\sigma}$, then $\hat{g}_j$ vanishes on $E^{\overline{\sigma}} \times \C$. Thus,
\begin{align*}
T_x\X \cap (E^{\overline{\sigma}}\times \C) =& \{v \in E\times \C \mid (d\hat{g}_j)_x(v) = 0, \text{ for all }j\in [1,J] \text{ such that }|\hat{g}_j|\in \overline{\sigma}\} \cap (E^{\overline{\sigma}}\times \C).
\end{align*}
Combining these equalities, we arrive at
\[
    T_x\X^\sigma = \{v \in E\times \C \mid (d\hat{g}_j)_x(v) = 0, \text{ for all }j\in [1,J] \text{ such that }|\hat{g}_j|\in \overline{\sigma}\} \cap (E^{\overline{\sigma}}\times \C).
\]

Finally, let $v\in T_x \X$ and write $v = v' + v''$, where $v'\in E^{\overline{\sigma}} \times \C$ and $v''\in  (E^{\overline{\sigma}}\times\C)^\perp$. If $|\hat{g}_j|\in \overline{\sigma}$, then $\hat{g}_j$ vanishes on $(E^{\overline{\sigma}} \times \C)^\perp$ and so $(d\hat{g}_j)_x(v'') = 0$. It follows from the description of $T_x\X$ and $T_x\X^\sigma$ above that $v' = v-v''\in T_x \X^\sigma$. Thus, $v\in T_x \X^\sigma \oplus (E^{\overline{\sigma}}\times \C)^\perp$.
\end{proof}

\begin{proof}[Proof of Proposition~\ref{CTSvfprop}]
The gradient of $\Re \pi|_{\X^\sigma\backslash \X_0}$ is computed with respect to the metric on $\X^\sigma\backslash \X_0$ which is the restriction of the fixed K\"ahler metric on $E\times \C$.
 Let $x\in \X^\sigma\setminus \X_0$ and take a sequence $\{x_j\}_{j\in \N}\in \X\backslash \X_0$ converging to $x$. By passing to a subsequence, we may assume that $\{x_j\}_{j\in \N} \subset \X^\tau\setminus \X_0$ for some $\tau\in \Sigma(\X)$ with $\sigma \prec \tau$, and that $\lim_{j\to \infty} T_{x_j} \X^\tau$ exists.
By Definitions \eqref{eq; grad ham vf def} and \eqref{eq; general stratified gh vf definition} and Lemma~\ref{dumbsubsequencelemma}, it suffices to show that 
$
\lim_{j\to \infty} \nabla_{x_j} (\Re \pi|_{\X^\tau}) = \nabla_x (\Re \pi|_{\X^\sigma}).
$

In what follows, if $W$ is a linear subspace of $E\times \C$, then $\pr_W$ denotes orthogonal projection to $W$. We have
\begin{align*}
\lim_{j\to \infty} \nabla_{x_j} (\Re \pi|_{\X^\tau}) & = \lim_{j\to \infty} \pr_{T_{x_j} \X^\tau} \left(-\frac{\partial}{\partial t}\right) \\
& = \pr_{ \lim_{j\to \infty} T_{x_j} \X^\tau} \left(-\frac{\partial}{\partial t}\right)  \\
& =  \pr_{ \lim_{j\to \infty} T_{x_j} \X^\tau}\circ \pr_{T_x \X^\sigma \oplus (E^{\overline{\sigma}}\times \C)^\perp}\left(-\frac{\partial}{\partial t}\right) \\
& \qquad \qquad \left(\text{by Lemma \ref{lemma;decomplemma}}\right) \\
& = \pr_{ \lim_{j\to \infty} T_{x_j} \X^\tau}\left( \pr_{T_x \X^\sigma}\left(-\frac{\partial}{\partial t}\right) + \pr_{(E^{\overline{\sigma}}\times\C)^\perp}\left(-\frac{\partial}{\partial t}\right)\right) \\
& =  \pr_{ \lim_{j\to \infty} T_{x_j} \X^\tau}\left( \pr_{T_x \X^\sigma}\left(-\frac{\partial}{\partial t}\right)\right) \\
& =  \pr_{T_x \X^\sigma}\left(-\frac{\partial}{\partial t}\right)\\
&\qquad\qquad \left(\text{by Assumption \ref{D2}}\right) \\
& = \nabla_x (\Re \pi|_{\X^\sigma}).
\end{align*}
This proves the claim.
\end{proof}

\subsubsection{Isomorphisms between decomposed affine K\"ahler varieties} Let $X$ be an affine $H$-variety and import the notation from Section \ref{section; affine G-varieties}. Suppose that $X$ embeds as a $H$-invariant closed subvariety of a $H$-module $E$. We introduce the following conditions, which are analogues of~\ref{D1} and~\ref{D2} for $X$:
\begin{enumerate}[label=(D\arabic*')]
\item The partition of $X$ into $X^\sigma$, $\sigma\in \Sigma(X)$, gives $X$ the structure of a decomposed variety. \label{D1'}
\item The decomposition of $X$ in \ref{D1'} satisfies the Whitney condition~\ref{Whitney A} with respect to the embedding into $E$. \label{D2'}
\end{enumerate}

We now give the first application of Proposition~\ref{CTSvfprop}.  Let $(E,h_E)$ and $(E',h_{E'})$ be two unitary $T$-modules, with symplectic forms $\omega_E, \omega_{E'}$ and $T$-equivariant moment maps $\mu,\mu'$, respectively. Let $i\colon X\hookrightarrow E$ and $i' \colon X\hookrightarrow E'$ be two closed $H$-equivariant embeddings. 
We assume without loss of generality that $X$ is not contained in any proper affine subspace of $E$ or $E'$. That is, we assume that no $f\in E^*\backslash\{0\}$ restricts to a constant function on $X$. In particular, because $X$ is irreducible $\Lambda(E) = \Lambda(X) =\Lambda(E')$ (see also Remark~\ref{WLOGremark}). The embeddings $i$ and $i'$ each endow $X$ with the structure of a Hamiltonian $T$-space. Denote these Hamiltonian $T$-spaces $(X_E,\mu)$ and  $(X_{E'},\mu')$ respectively.

 \begin{theorem} \label{GHmoser}
 Assume $X$ satisfies \ref{D1'} and the embeddings $i,i'$ both satisfy \ref{D2'}.
 If the moment maps $\mu\colon E\to \ttt^*$ and $\mu'\colon E'\to \ttt^*$ are proper, then $(X_E,\mu)$ is isomorphic to $(X_{E'},\mu')$ as a Hamiltonian $T$-space.
 \end{theorem}

\begin{proof}
Consider the unitary $T$-module $(E \times E' \times \C, h = h_E\oplus h_{E'} \oplus h_\C)$ with symplectic structure $\omega = -\im h$, where $T$ acts trivially on $\C$. The action of $T$ is Hamiltonian with moment map $\psi = \mu\circ \pr_E + \mu' \circ \pr_{E'}$.
Consider the trivial degeneration $\X = X\times \C$ of $X$. Embed $\X$ into $E \times E' \times \C$ according to the map 
\begin{align*}
X\times \C & \to E\times E' \times \C \\
(x,t) & \mapsto (t i(x), (1-t) i'(x), t).
\end{align*}
Its image is a closed $H$-invariant subvariety of $E\times E'\times \C$. It has smooth pieces $\X^\sigma = X^\sigma \times \C = (X\times \C) \cap (E^\sigma \times{E'}^\sigma \times \C)$.

We apply Theorem~\ref{thm; toric degenerations main theorem} to the degeneration $X\times \C\to \C$. To do so, we need to check the conditions \ref{assumeA}--\ref{assumeG}. The conditions \ref{assumeA}, \ref{assumeB}, and \ref{assumeC}, and \ref{assumeE}I) are satisfied automatically.  

To show~\ref{assumeE}II), it suffices to show that $(\psi,\pi) \colon E\times E' \times \C\to \ttt^*\times \C$ is proper.  
By Lemma~\ref{lemma; unitary torus action proper moment map}\ref{lemma2.2ii}, because $\mu$ and $\mu'$ are assumed to be proper, the cone $\Gamma(E)=\Gamma(E')$ is strongly convex. It follows that the map $\psi = \mu\circ\pr_E + \mu\circ \pr_{E'}$ is proper. Thus $(\psi,\pi) \colon E\times E' \times \C\to \ttt^*\times \C$ is proper. The condition~\ref{assumeE}III) holds because 
\[
\X^{\overline{\sigma}} = \X \cap (E^{\overline{\sigma}} \times {E'}^{\overline{\sigma}} \times \C) = \X \cap \psi\n(\overline{\sigma}). 
\]
Here the last equality is a consequence of Lemma \ref{lemma; unitary torus action proper moment map}\ref{lemma2.2i}.

In contrast with the general setup of Theorem~\ref{thm; toric degenerations main theorem}, there are no singular points of $X^\sigma \times \C$. As a result, the stratified gradient Hamiltonian flow $\varphi_t$ is defined for all $t\in \R$, for all points of $X\times \C$.
The condition~\ref{assumeF} is then satisfied automatically.

Finally, it remains to verify condition~\ref{assumeG}. To do this, we apply Proposition~\ref{CTSvfprop}. To apply Proposition~\ref{CTSvfprop}, we need to check that~\ref{D1} and~\ref{D2} hold for $\X\subset E\times E' \times \C$. This is a straightforward consequence of the fact that, by assumption, the partition of $X$ into $X^\sigma$ satisfies~\ref{D1'}, and that each embedding $X\hookrightarrow E$ and $X\hookrightarrow E'$ satisfies~\ref{D2'}.

After applying Theorem~\ref{thm; toric degenerations main theorem}, we have a $T$-equivariant continuous map $\phi  \colon X_E \to X_{E'}$ which, for each $\sigma\in \Sigma$, restricts to a symplectomorphism 
$
\varphi_1^\sigma \colon D^\sigma \to U_0^\sigma
$
from an open dense subset $D^\sigma$ of $X_E^\sigma$ to an open dense subset $U_0^\sigma $ of $ X_{E'}^\sigma$. Since $ X_{E'}^\sigma$ is smooth,  $D^\sigma = X_E^\sigma$ and  $U_0^\sigma = X_{E'}^\sigma$. In other words, $\varphi_1^\sigma$ defines a symplectomorphism of $X_E^\sigma$ and $X_{E'}^\sigma$. What is more, $\psi\circ \phi = \psi$. Thus, we  have a map of Hamiltonian $T$-spaces $\phi\colon (X_E,\mu)\to(X_{E'},\mu')$. The map $\varphi_{-1}$ is an inverse to $\phi$, so $\phi$ is an isomorphism of Hamiltonian $T$-spaces.
\end{proof}

\section{From valuations to stratified gradient Hamiltonian flows}
\label{good valuation section}

\label{section;valuation GH}

This section provides a general recipe for constructing  integrable systems on a decomposed affine K\"ahler variety $X$.  This is achieved by constructing toric degenerations of $X$ that satisfy all the assumptions of the framework set out in Section~\ref{section; Toric degenerations of stratified affine varieties}. The construction of toric degenerations given here is similar to \cite[Section 3.2, 3.3]{haradaIntegrableSystemsToric2015}, which deals with the case where $X$ is a smooth projective variety. Our construction is slightly more detailed, due to the possibly singular nature of $X$. Moreover, many details of the construction cited above do not carry over directly  because of differences between the affine and projective settings.

The main ingredients of our toric degeneration construction are a torus action on $X$ and a valuation on the coordinate ring of $X$. These ingredients must satisfy some compatibility conditions which we package in our definition of a \emph{good valuation} (Definition \ref{definition;good}). The heart of this section is devoted to showing that suitable toric degenerations can be constructed from good valuations. In particular, we show that  toric degenerations constructed from good valuations can be embedded into affine space such that the K\"ahler structure, the decomposition of $X$, the torus action on $X$, and the big torus action on the toric fiber are all compatible with one another (cf.~Propositions \ref{prop; embed nicer}, \ref{proposition;GramSchmidt}, and \ref{proposition; degen from good valuation}). This is achieved using properties of Khovanskii bases and our symplectic isotopy theorem for decomposed affine K\"ahler varieties (Theorem \ref{GHmoser}).

The general idea of this section is as follows. We start with a torus $T_\mathsf{c}$, which we call the \emph{control torus}, together with an action of $T_\mathsf{c}$ on an affine variety $X$. The variety $X$ must be $T_\mathsf{c}$-equivariantly embedded in a finite dimensional $T_\mathsf{c}$-module, and this embedding must be such that the partition defined in \eqref{eq; partition of H variety def} makes $X$ a decomposed affine K\"ahler variety.  The pieces $X^\sigma$ of this decomposition are indexed by the relative interiors $\sigma$ of faces of the weight cone for the action of $T_{c}$.
Next, we assume the variety $X$ comes with a valuation $\mathbf{v}$ with values in the weight lattice of a bigger torus $\mathbb{T}$. We require the grading of $\C[X]$ by $T_\mathsf{c}$-weights to factor as  $\mathsf{c} \circ \mathbf{v}$, where $\mathsf{c} \colon  \Lie(\mathbb{T})^* \to \mathfrak{t}^*_{\mathsf{c}}$ is dual to an embedding of $T_\mathsf{c}$ in $\mathbb{T}$. Using the valuation, we realize $X$ as the 1-fiber of a degeneration to the affine toric variety $X_\mathsf{S}$ associated with the value semigroup $\mathsf{S}$ of $\mathbf{v}$. By applying the results from the previous section, we arrive at commuting diagrams
\begin{equation}\label{eq; idea of section 4}
    \begin{tikzcd}
        X \ar[d, "\psi_{\mathsf{c}}"] \ar[r, "\phi"] & X_{\mathsf{S}}\ar[d, "\Psi"] && X^\sigma \ar[d, "\psi_{\mathsf{c}}\vert_{X^\sigma}"] \ar[r, "\phi^\sigma"] & X_{\mathsf{S}}^\sigma\ar[d, "\Psi\vert_{X_{\mathsf{S}}^\sigma}"]\\
        \mathfrak{t}^*_\mathsf{c}  & \Lie(\mathbb{T})^* \ar[l, "\mathsf{c}"'] && \sigma & \mathsf{c}^{-1}(\sigma)\ar[l, "\mathsf{c}"']
    \end{tikzcd}
\end{equation}
where the right diagram is defined for all $\sigma$. In these diagrams $\psi_\mathsf{c}$ is the quadratic moment map for the $T_c$-action on $X$ and $\Psi$ is the quadratic moment map for the $\mathbb{T}$-action on $X_\mathsf{S}$. %The right diagram is obtained by restricting maps in the left diagram to subspaces associated with $\sigma$.  
%The pieces of $X$ satisfy $X^\sigma = \psi_\mathsf{c}^{-1}(\sigma)$.
The toric variety $X_\mathsf{S}$ is partitioned by toric subvarieties $X_\mathsf{S}^\sigma$ that satisfy $X_\mathsf{S}^\sigma= \Psi^{-1}(\mathsf{c}^{-1} (\sigma))$. 
For each face $\sigma$ there is a limiting gradient Hamiltonian flow map $\phi^\sigma\colon X^\sigma \to X_\mathsf{S}^\sigma$ that restricts to a symplectomorphism from an open dense subset of $X^\sigma$ onto the smooth locus of $X_\mathsf{S}^\sigma$.
The maps $\phi^\sigma$ assemble into the continuous map $\phi$ in the left diagram which is the time-1 limit of the stratified gradient Hamiltonian flow. The composition $\Psi \circ \phi$ defines a completely integrable system on $X$ in the sense of Definition~\ref{def; integrable system on singular symplectic space}.

The section is organized as follows. Section~\ref{section;valuationsanddegenerations} provides  relevant properties of valuations, Khovanskii bases, and the Rees algebra construction of toric degenerations.  Section~\ref{subsection; kahler form on toric variety} contains our definition of good valuations along with Propositions~\ref{prop; embed nicer} and~\ref{proposition;GramSchmidt}.  Section~\ref{main construction section} combines results from the previous sections to produce a toric degeneration with an embedding into affine space that has all the desired properties (Proposition~\ref{proposition; degen from good valuation}). We also pause in Section~\ref{subsection; hamiltonian torus action} to collect some facts about the resulting toric moment maps which will be useful later, in Section~\ref{section 5}. Finally, everything is combined to produce Theorem~\ref{maintheorem}, which is the main result.   

\subsection{Toric degenerations from valuations} \label{section;valuationsanddegenerations}

Although the Rees algebra construction for affine varieties is well-known, we will require (in Sections~\ref{subsection; kahler form on toric variety} and~\ref{main construction section}) that the resulting toric degeneration has several additional properties, the combination of which is possibly less well-known.  
In particular, we will consider Rees algebras constructed from valuations and the resulting degenerations will be embedded into affine space using a Khovanskii basis. Moreover, the toric degeneration and the embedding into affine space will also need to be compatible with a given torus action on $X$. We provide a package of assumptions for valuations on the coordinate ring of $X$, labelled~\ref{v1}--\ref{v5}, which are sufficient for the resulting toric degeneration to have the desired properties. In relation to the fact that we will deal with decomposed varieties,  we  give a careful algebraic description of certain subvarieties of $X$ and associated subfamilies of the toric degeneration in Proposition \ref{prop;subfamilies}.

\subsubsection{Valuations} \label{subsection; valuations} 
\begin{definition} \label{def; valuation} Let $L$ be a lattice equipped with a total order $>$ which respects addition. Let $A$ be an algebra over $\C$. 
A function $\mathbf{v} \colon A\setminus\{0\} \to L$ is a \emph{valuation} if, for all nonzero $f, g \in A$, it satisfies the following:
\begin{enumerate}
    \item $\mathbf{v}(f+g) \le \max \{\mathbf{v}(f),\mathbf{v}(g)\}$;
    \item $\mathbf{v}(cf) = \mathbf{v}(f)$ for all nonzero $c \in \C$;
    \item $\mathbf{v}(fg) = \mathbf{v}(f) + \mathbf{v}(g)$.
\end{enumerate}
\end{definition}
Let $L$ and $A$ be as in Definition~\ref{def; valuation}, and $\mathbf{v} \colon A\setminus\{0\} \to L$ be a valuation on $A$.
%\footnote{We define valuations as in \cite{}, except we use the convention that $\mathbf{v}(f+g) \le \max \{\mathbf{v}(f),\mathbf{v}(g)\}$.} 
The \textit{value semigroup} of $\mathbf{v}$ is the image $\mathbf{v}(A\backslash\{0\})$. Denote the value semigroup $\mathsf{S}_\mathbf{v}$, or $\mathsf{S}$ when $\mathbf{v}$ is clear from context. Throughout Section \ref{section;valuation GH}, assume that  $\mathsf{S}_\mathbf{v}$ generates $L$ as a $\Z$-module. In particular, we only consider valuations such that $\mathsf{S}_\mathbf{v} \subset L$ is \textit{saturated}, i.e., for all $k\in \mathbb{N}\backslash\{0\}$ and $s\in L$, $ks\in \mathsf{S}_\mathbf{v}$ implies $s \in \mathsf{S}_\mathbf{v}$. For $s\in  \mathsf{S}$, denote
\begin{align*}
    A_{\le s} & = \{ f \in A \mid \mathbf{v}(f) \le s \text{ or } f = 0\}, \quad A_{< s} = A_{\le s}\setminus \mathbf{v}\n(s).
\end{align*} 
We say that $\mathbf{v}$ \textit{has one-dimensional leaves} if  $A_{\le s}/A_{< s}$ is at most one-dimensional for all $s\in \mathsf{S}$. Let $\gr A$ denote the associated graded algebra. If $A$ is an integral domain and $\mathbf{v}$ has one-dimensional leaves, then $\gr A$ is isomorphic to the semigroup algebra $\C[\mathsf{S}]$~\cite[Remark~4.13]{gubeladzePolytopesRingsKTheory2009}.

Assume that $\mathbf{v}$ has one-dimensional leaves. A \emph{Khovanskii basis} for $A$ is a set $\mathcal{K} \subset A\backslash \{0\}$ such that $\mathbf{v}(\mathcal{K})$ generates  $\mathsf{S}_\mathbf{v}$ \cite{kavehKhovanskiiBasesHigher2019}. Assume that $\mathcal{K}$ is finite and there exists a lattice $\Lambda$ and a surjective linear map $\mathsf{w}\colon L \to \Lambda$ such that:
\begin{enumerate}[label=(v\arabic*)]
    \item\label{v1} The total order $>$ on $L$ descends under $\mathsf{w}$ to a total order on $\Lambda$.
    %\note{This order on $\Lambda$ respects addition} 
    \item\label{v2} The image $\mathsf{w}(\mathsf{S}_\mathbf{v})$ contains a minimal element with respect to this order. %\note{The minimal element must be 0.}
    \item\label{v3} The fibers $\mathsf{w}\n(\lambda) \cap \mathsf{S}_\mathbf{v}$ are all finite.%\note{This implies  $\mathsf{w}\n(0) \cap \mathsf{S} = \{0\}$}
\end{enumerate}
With these assumptions in place, we have the following proposition which is a straightforward generalization of \cite[Proposition 3.12]{haradaIntegrableSystemsToric2015}. The ``subduction algorithm'' proof given there adapts easily to this setting.

\begin{proposition}\label{prop; K generates A}
If assumptions \ref{v1}-\ref{v3} hold, then $\mathcal{K}$ generates $A$ as an algebra.
\end{proposition}

In particular, this situation produces an embedding $\spec(A) \hookrightarrow E$, where $E$ is dual to the subspace $\operatorname{span}_\C \mathcal{K}\subset A$. Assumptions \ref{v1}-\ref{v3} also have the following consequence which will be useful later on. Its proof is  a straightforward exercise.

\begin{lemma}\label{lemma; cones are strongly convex}
If assumptions \ref{v1}-\ref{v3} hold, then $\cone(\mathsf{S}_\mathbf{v})$ and $\cone(\mathsf{w}(\mathsf{S}_\mathbf{v}))$ are strongly convex.
\end{lemma}

\subsubsection{The Rees algebra construction}\label{Rees Section} Let $H$ be an algebraic torus with (real) weight lattice $\Lambda$ and let $X$ be an affine $H$-variety with coordinate algebra $A$.  
As in the previous section, equip $A$ with a valuation $\mathbf{v}$ with values in $(L,>)$. Assume that $\mathbf{v}$ has one-dimensional leaves and $\mathsf{S}_\mathbf{v}$ is finitely generated.  Assume there exists a linear map $\mathsf{w} \colon L \to \Lambda$ that satisfies  \ref{v1}-\ref{v3} and the following.
\begin{enumerate}[label=(v\arabic*), start=4]
    \item\label{v4} If $f\in A$ is homogeneous, then $|f|_\Lambda = \mathsf{w}(\mathbf{v}(f))$.
    \item\label{v5} The tuple $(\mathsf{S},>,\mathsf{w})$ is \emph{refinable}\footnote{See \cite{alexeevToricDegenerationsSpherical2005, calderoToricDegenerationsSchubert2002} for sufficient conditions that  $(\mathsf{S},>,\mathsf{w})$ be refinable.}, i.e.~for any two finite sets 
    $
        \{a_1,\dots,a_N\}, \{b_1,\dots,b_N\}\subset \mathsf{S}
    $
    such that $\mathsf{w}(a_i) = \mathsf{w}(b_i)$ and $a_i>b_i$ for all $i\in [1,N]$, there exists a linear map $e\colon L \to \Z$ such that
\begin{equation}\label{fix e}
e(\mathsf{S})\subset \N, \text{ and }  e(a_i)>e(b_i) \text{ for all } i\in [1,N].
\end{equation}
\end{enumerate}

Fix an $H$-equivariant embedding of  $X $ as a closed subvariety of a finite dimensional $H$-module $E$. Let $\{z_i\}_{i=1}^n$ be a system of $\Lambda$-homogeneous linear coordinates on $E$. Let $f_i \in A$ denote the restriction of $z_i$ to $X$. Assume that  $\mathcal{K} = \{f_i\}_{i=1}^n$ is a Khovanskii basis of $A$ (in particular, $f_i \neq 0$).

Let $\H$ denote the algebraic torus with (real) weight lattice $L$.  Define an $\H$-module structure on $E^*$ by letting $ h \cdot z_i = h^{\mathbf{v}(f_i)}z_i$ for all $h \in \H$ and $i= 1,\dots ,n$. Equip $E$ with the dual $\H$-module structure.\footnote{The linear map $\mathsf{w}$ is dual to a homomorphism  $H \to \H$. Along with the $\H$-module structure, this homomorphism equips $E$ with a $H$-module structure. By \ref{v4}, this coincides with the original $H$-module structure on $E$.}

The discussion above produces surjective algebra homomorphisms:
\begin{align}
    \C[E] & \to A=\C[X]; &  \C[E] & \to \gr A \label{embedinE} \\
    z_i & \mapsto f_i; &  z_i & \mapsto f_i \mod A_{<\mathbf{v}(f_i)}. \nonumber
\end{align}
The first is dual to the embedding $X \hookrightarrow E$. It is a map of $\Lambda$-graded algebras. The second is a map of $L$-graded algebras  dual to an $\H$-equivariant embedding $\spec(\gr A) \hookrightarrow  E$. The proof of the following lemma is a direct analogue of the first half of the proof of \cite[Theorem 3.13]{haradaIntegrableSystemsToric2015}.

\begin{lemma} \label{lemma;kernel} 
Let $\overline{g}_1,\dots,\overline{g}_J\in \C[E]$ be $L$-homogeneous generators of the ideal $\ker(\C[E] \to \gr A)$. Then, there exist $\Lambda$-homogeneous generators $g_1,\dots, g_J\in \C[E]$ of the ideal $ \ker(\C[E] \to A)$ which have the form
\begin{equation}
    \label{niceformg}
    g_j = \overline{g}_j + p_{j},\quad    \mathbf{v}(\overline{g}_j(f_1,\dots,f_n))>\mathbf{v}(p_{j}(f_1,\dots,f_n)).
\end{equation}
The $\Lambda$-homogeneous degree of $g_j$ is $|g_j|_\Lambda = \mathsf{w}(|\overline{g}_j|_L) = \mathsf{w}(\mathbf{v}(\overline{g}_j(f_1,\dots,f_n)))$.
\end{lemma}

Let $\overline{g}_j$, $g_j$, and $p_j$, $j = 1,\dots, J$ be as in Lemma~\ref{lemma;kernel}. Let $s_j = |\overline{g}_j|_L$ denote the $L$-homogeneous degree of $\overline{g}_j$ as an element of $\C[E]$. Write $p_j = \sum_{l=1}^{L_j} M_{j,l}$ where each $M_{j,l}\in \C[E]$ is a $\Lambda$-homogeneous monomial in $z_1,\dots,z_n$, of degree $|M_{j,l}|_\Lambda = \mathsf{w}(s_j)$. Fix a $\Z$-linear map $e\colon L\to \Z$ such that:
\begin{equation}
\label{fix e'}
e(\mathsf{S})\subset \N, \text{ and }  e(\mathbf{v}(\overline{g}_j(f_1,\dots,f_n)))>e(\mathbf{v}( M_{j,l}(f_1,\dots,f_n))) \, \forall j,l.
\end{equation} 
This exists by \ref{v5}.
For all $k\geq 0$, define
\[
    A_{\leq k } = \{ f \in A\mid e(\mathbf{v}(f)) \leq k \text{ or } f=0\}.
\]
The subspaces $A_{\leq k}$ define a $\N$-graded filtration of $A$. The \emph{Rees algebra} of this filtration is 
\begin{equation}\label{eqn; rees alg def}
    \mathcal{R} = \bigoplus_{k\geq 0} A_{\leq k} \otimes t^k \subset A\otimes\C[t].
\end{equation}
The algebra $\mathcal{R}$ inherits a $\Lambda$-grading from $A$ (where $t$ is defined to be homogeneous of degree 0). The following collects standard facts about $\mathcal{R}$; see for instance~\cite[Corollary~6.11]{eisenbudCommutativeAlgebra1995a}.

\begin{proposition}\label{prop; rees algebra properties}
Let $\mathcal{R}$ be as in \eqref{eqn; rees alg def}. Then,
	\begin{enumerate}
		\item $\mathcal{R}$ is finitely generated.
		\item The $\C$-algebra homomorphism $\C[t]\to \mathcal{R}$, $t \mapsto t$ makes $\mathcal{R}$ into a flat $\C[t]$-algebra. 
		\item $\mathcal{R}/t\mathcal{R}\cong \gr A \cong \C[\mathsf{S}]$.
		\item $\mathcal{R}[t\n]\cong A\otimes \C[t,t\n]$. 
	\end{enumerate}
\end{proposition}

Let $\X$ be the affine $H$-variety $\spec \mathcal{R}$.  Dualizing the map $\C[t]\to \mathcal{R}$ gives a flat morphism $\X \to \C$, which is a toric degeneration of $X$ to  $\spec \gr A$.

The variety $\X$ can be embedded $H$-equivariantly into $E\times \C$.
Define an algebra homomorphism
\begin{equation}\label{equation;dual to embedding of the degeneration}
    \C[E]\otimes \C[t] \to A\otimes \C[t], \quad z_i \mapsto  t^{e(\mathbf{v}(f_i))}f_i, \,  t\mapsto  t.
\end{equation}
It is a map of $\Lambda$-graded algebras ($t$ is homogeneous of degree 0 in both algebras). Its image is   $\mathcal{R}$. We then have an $H$-equivariant embedding $\X\hookrightarrow E\times \C.$
The image of $\X$ in $E\times \C$ is the subvariety cut out by the $\Lambda$-homogeneous polynomials 
\begin{equation*}\label{eqn; kernel of toric deg embedding}
    \hat{g}_j = \overline{g}_j(z_1,\dots,z_n) + \sum_{l=1}^{L_j} t^{m_{j,l}} M_{j,l}(z_1,\dots,z_n) \quad j = 1,\dots ,J.
\end{equation*}
where $m_{j,l} = e(\mathbf{v}(\overline{g}_j(f_1,\dots,f_n))-\mathbf{v}(M_{j,l}(f_1,\dots,f_n)))$.

We conclude by making some observations about subfamilies/subvarieties of $\X$, $X$, and the toric fiber. We adopt the notation of Section \ref{section; affine G-varieties}.

\begin{proposition} \label{prop;subfamilies}
Let $\overline{\sigma}$ be a closed face of $\Gamma(X)$.
\begin{enumerate}
    \item $\mathbf{I}(\X \cap (E^{\overline{\sigma}} \times \C))$ is generated by
\[
\{\hat{g}_j \mid \mathsf{w}(s_j) \in \overline{\sigma}\} \cup \{z_i \mid \mathsf{w}(\mathbf{v}(f_i)) \notin \overline{\sigma}\}.
\]
Consequently, $\C[\X\cap(E^{\overline{\sigma}} \times \C)]$ is isomorphic to the subalgebra of $\C[\X]= \mathcal{R}$ generated by $\{t\}\cup \{t^{e(\mathbf{v}(f_i))} f_i \mid \mathsf{w}(\mathbf{v}(f_i)) \in \overline{\sigma}\}.$ Additionally, $\C[\X\cap(E^{\overline{\sigma}}\times \C)]$ is flat as a $\C[t]$-module.
\item Identify $E\times \{1\} = E$. Then $\mathbf{I}(\X \cap (E^{\overline{\sigma}} \times \{1\}))\subset \C[E]$ is generated by
\[
\{g_j \mid \mathsf{w}(s_j) \in \overline{\sigma}\} \cup \{z_i \mid \mathsf{w}(\mathbf{v}(f_i)) \notin \overline{\sigma}\}.
\]
Consequently, $\C[\X\cap(E^{\overline{\sigma}} \times \{1\})]$ is isomorphic to the subalgebra of $\C[X]$ generated by $\{f_i \mid \mathsf{w}(\mathbf{v}(f_i)) \in \overline{\sigma}\}$.
\item Identify $E\times \{0\} = E$. Then $\mathbf{I}(\X \cap (E^{\overline{\sigma}} \times \{0\}))\subset \C[E]$ is generated by
\[
\{\overline{g}_j \mid \mathsf{w}(s_j) \in \overline{\sigma}\} \cup \{z_i \mid \mathsf{w}(\mathbf{v}(f_i)) \notin \overline{\sigma}\}.
\]
Consequently, $\C[\X\cap(E^{\overline{\sigma}} \times \{0\})]$ is isomorphic to the subalgebra of $\C[\mathsf{S}]$ generated by $\{\chi^{\mathbf{v}(f_i)} \mid \mathsf{w}(\mathbf{v}(f_i)) \in \overline{\sigma}\}$.
\end{enumerate}
\end{proposition}

\begin{proof}
By Lemma~\ref{lem; scheme theoretic intersection is reduced}, $\mathbf{I}(\X \cap (E^{\overline{\sigma}} \times \C)) = \mathbf{I}(\X) + \mathbf{I}(E^{\overline{\sigma}} \times \C))$. It follows by  \eqref{useful fact about subspaces} that if $|\hat{g}_j|\notin \overline{\sigma}$, then $\hat{g}_j$ vanishes on $E^{\overline{\sigma}} \times \C$. The description of $\mathbf{I}(\X \cap (E^{\overline{\sigma}} \times \C))$
immediately follows.
The $\C[t]$-module $\C[\X\cap(E^{\overline{\sigma}}\times\C)]$ is flat because $\C[\X\cap(E^{\overline{\sigma}}\times\C)]\subset \mathcal{R}$ is a torsion-free $\C[t]$-module, and $\C[t]$ is a principal ideal domain~\cite[Corollary~6.3]{eisenbudCommutativeAlgebra1995a}. The second and third items follow from the first, by putting $t=0$ and $t=1$.
\end{proof}

\subsection{Good valuations}\label{subsection; kahler form on toric variety} 

\begin{definition} \label{definition;good} Let $H_\mathsf{a}$ and $H_\mathsf{c}$ be algebraic tori with compact forms $T_\mathsf{a}$ and $T_\mathsf{c}$, respectively. Let $X$ be an affine $H_\mathsf{a}\times H_\mathsf{c}$-variety. A \emph{good valuation on $X$}
 is a tuple $(X, E, h_E, \mathbf{v}, \mathsf{a},\mathsf{c})$ consisting of:
\begin{enumerate}[label=(\roman*),start=1]
\item A finite dimensional Hermitian inner product space $(E,h_E)$ equipped with a unitary representation of $T_\mathsf{a}\times T_\mathsf{c}$ and a $H_\mathsf{a}\times H_\mathsf{c}$-equivariant embedding $X\hookrightarrow E$ of $X$ as a closed subvariety.
    \item A valuation $\mathbf{v}\colon A\backslash\{0\} \to L$ on $A= \C[X]$, with values in a lattice with total order $(L,>)$. We require that $\mathbf{v}$ has one-dimensional leaves, and that $\mathsf{S}=\mathsf{S}_\mathbf{v}$ is finitely generated.
    \item Surjective $\Z$-linear maps
 \[
 \mathsf{a} \colon  L \to \Lambda_{\mathsf{a}}, \qquad  \mathsf{c} \colon  L \to \Lambda_{\mathsf{c}}.
 \]
    where $\Lambda_{\mathsf{a}}$ (resp.~$\Lambda_{\mathsf{c}}$) is the character lattice of $H_\mathsf{a}$ (resp.~$H_\mathsf{c}$). 
 \end{enumerate}
Let $\Lambda(X)\subset \Lambda_\mathsf{c}$ be the semigroup of weights of the $H_\mathsf{c}$-module $\C[X]$, let $\Gamma=\Gamma(X) =\cone \Lambda(X)$, and let $\Sigma=\Sigma(X)$ be the face poset of $\Gamma$. 
The data must satisfy two compatibility conditions:
\begin{enumerate}[label=(GV\arabic*)]
     \item \label{GV1new} (Compatibility of the valuation) The valuation $\mathbf{v}$ and the map $\mathsf{c}$ satisfy conditions~\ref{v1}-\ref{v5} (with $\mathsf{w}=\mathsf{c}$). Additionally, the valuation $\mathbf{v}$ and the map $\mathsf{a}$ satisfy condition~\ref{v4} (with $\mathsf{w}=\mathsf{a}$). 
     \item \label{GV2new} (Compatibility of the decomposition) The partition of $X$ by the subvarieties $X^\sigma$, $\sigma\in \Sigma$, defined by the $H_\mathsf{c}$-action\footnote{See \eqref{eq; partition of H variety def} in Section \ref{section; affine G-varieties}.} equips $X$ with the structure of a decomposed variety~\ref{D1'}. The decomposition satisfies the Whitney condition~\ref{Whitney A} with respect to the embedding into $E$~\ref{D2'}. 
\end{enumerate}
\end{definition}

The actions of $T_\mathsf{a}$ and $T_\mathsf{c}$ on $E$ are Hamiltonian with quadratic moment maps $\psi_\mathsf{a}$ and $\psi_\mathsf{c}$.
Given a good valuation, the tuple $(X,E,\omega_E)$ is a decomposed affine K\"ahler variety (Definition \ref{def; decompsed Kahler variety}) with respect to the symplectic form $\omega_E = -\Im h_E$ and the decomposition of $X$ described in \ref{GV2new}. The action of $T_\mathsf{a} \times T_\mathsf{c}$ endows $X$ with the structure of a Hamiltonian $T_\mathsf{a} \times T_\mathsf{c}$-space.

\begin{remark}
    The letters $\mathsf{a}$ and $\mathsf{c}$ stand for \emph{auxiliary} and \emph{control}. The action of the control torus $H_\mathsf{c}$ is necessary for the construction in Section~\ref{main construction section}. We include the data of the auxiliary torus so that the construction in Section~\ref{main construction section} can be performed in the presence of an additional group action. This auxiliary group action is not necessary for the construction. If there is no additional group action to keep track of, one may put $H_\mathsf{a} = \{e\}$. 
\end{remark}

\begin{remark} \label{WLOGremark} In all that follows we assume without loss of generality that $X$ is not contained in any proper affine subspace of $E$. That is, we assume that no $f\in E^*\backslash\{0\}$ restricts to a constant function on $X$. In particular, because $X$ is irreducible the semigroups of $H_\mathsf{c}$-weights $\Lambda(X)$ and $\Lambda(E)$ are equal. 
\end{remark}

The remainder of the section is devoted to showing the following: Given a good valuation $\mathbf{v}$ on $A=\C[X]$, we can always assume that the ambient affine space $E$ has a system of linear coordinates $\mathcal{K}\subset E^*$ which restricts to a Khovanskii basis for $A$ and $\mathbf{v}$, and such that the dual basis of $\mathcal{K}$ is an orthonormal weight basis of $E$. 

\begin{lemma} \label{psic proper}
 Let $X$ be an affine $H_\mathsf{a}\times H_\mathsf{c}$-variety, and let $( X, E, h_E, \mathbf{v}, \mathsf{a},\mathsf{c})$ be a good valuation on $X$. Then the map $\psi_\mathsf{c} \colon E \to \Lie(T_\mathsf{c})^*$ is proper. 
\end{lemma}

\begin{proof} 
Let $\Pi$ denote the set of weights of the dual representation of $T_\mathsf{c}$ on $E^*$. By Lemma \ref{lemma; unitary torus action proper moment map}\ref{lemma2.2ii}, $\psi_\mathsf{c}$ is proper if and only if 0 cannot be written as a non-trivial linear combination of elements of $\Pi$ with non-negative coefficients. 

Let $\{z_i\}_{i=1}^n \subset E^*$ be a basis of $T_\mathsf{c}$-weight vectors. By our assumption that $\hull(X) = E$, the restricted functions $f_i = z_i\vert_X$ are all non-constant $\Lambda_\mathsf{c}$-homogeneous elements of $A$. The embedding $X \hookrightarrow E$ is $T_\mathsf{c}$-equivariant, and so by \ref{v4} the weight of $z_i \in E^*$ equals $\mathsf{c}(\mathbf{v}(f_i))$.  Thus, $\Pi$ is a subset of $\mathsf{c}(\mathsf{S}_\mathbf{v})$. Since $\cone(\mathsf{c}(\mathsf{S}_\mathbf{v}))$ is strongly convex (Lemma \ref{lemma; cones are strongly convex}), 0 can be written as a non-trivial linear combination of elements of $\Pi$ with non-negative coefficients if and only if $0 \in \Pi$.  

Finally, from Lemma \ref{lemma; cones are strongly convex} and \ref{v3} it follows that $\mathsf{c}\n(0)\cap \mathsf{S}_\mathbf{v} = \{0\}$. The $f_i$ are not constant and $\mathbf{v}$ has one-dimensional leaves. This implies that $0 \not\in \Pi$.
\end{proof}

\begin{proposition} \label{prop; embed nicer}
 Let $X$ be an affine $H_\mathsf{a}\times H_\mathsf{c}$-variety, and let $( X, E, h_E, \mathbf{v}, \mathsf{a},\mathsf{c})$ be a good valuation on $X$. Then, there exists an inner product space $(E',h_{E'})$ with unitary $T_\mathsf{a}\times T_\mathsf{c}$-action, and a $H_\mathsf{a}\times H_\mathsf{c}$-equivariant embedding $X\hookrightarrow E'$, so that
\begin{enumerate}
\item There exists a basis of ${E'}^*$ which restricts to a Khovanskii basis for $A=\C[X]$ and $\mathbf{v}$.
\item $(X, E', h_{E'}, \mathbf{v}, \mathsf{a},\mathsf{c})$ defines a good valuation on $X$.
\item Let $X_E$ and $X_{E'}$ denote the two Hamiltonian $T_\mathsf{a}\times T_\mathsf{c}$-space structures on $X$ coming from the embeddings $i$ and $i'$, respectively. Then $X_E$ is isomorphic to $X_{E'}$ as a Hamiltonian $T_\mathsf{a}\times T_\mathsf{c}$-space.
\end{enumerate}
\end{proposition}

\begin{proof}
 The image of the linear map $E^*\to \C[X]$ generates $\C[X]$ as an algebra. Pick a finite Khovanskii basis of $\C[X]$. By picking large enough $N$, one can ensure that the image of the natural map $\bigoplus_{k=1}^N \operatorname{Sym}^k(E^*) \to \C[X]$ contains this finite Khovanskii basis. Define $E' = \bigoplus_{k=1}^N \operatorname{Sym}^k(E)$. Then there is a natural $H_\mathsf{a}\times H_\mathsf{c}$-action on $E'$, and $(E')^*$ is canonically isomorphic to $\bigoplus_{k=1}^N \operatorname{Sym}^k(E^*)$. 
The linear map $\bigoplus_{k=1}^N \operatorname{Sym}^k(E^*) \to \C[X]$ determines a surjection of algebras $\C[E'] \to \C[X]$ and, in turn, an embedding $X\hookrightarrow E'$. This map is $H_\mathsf{a}\times H_\mathsf{c}$-equivariant.

 Put a $T_\mathsf{a}\times T_\mathsf{c}$-invariant inner product $h_{E'}$ on $E'$. We check that the tuple $( X , E', h_{E'}, \mathbf{v}, \mathsf{a},\mathsf{c})$ satisfies \ref{GV2new}. Condition~\ref{D1'} holds as it does not depend on the embedding of $X$. To check~\ref{D2'}, we take the natural $H_\mathsf{a}\times H_\mathsf{c}$-equivariant surjection $\C[E'] \to \C[E]$, which realizes $E$ as a smooth subvariety of $E'$. The map $\C[E'] \to \C[X]$ factors as $\C[E'] \to \C[E] \to \C[X]$. Since $X\subset E$ is Whitney A, and $E$ is a smooth subvariety  of $E'$, it follows that $X\subset E'$ is Whitney A.
 
 Finally, in order to eliminate elements of ${E'}^*$ which are constant on $X$, 
 we may replace $E'$ with a subspace of $E'$, as in Remark~\ref{WLOGremark}. 
 The resulting embedding $X\hookrightarrow E'$ then satisfies items 1 and 2.

 We will apply Theorem~\ref{GHmoser} in order to show that $X_E$ is isomorphic to $X_{E'}$. (The proof of equivariance with respect to $T_\mathsf{a}$ follows exactly as the proof for $T_\mathsf{c}$, since the Hamiltonian action of $T_\mathsf{a}$ on $E\times E' \times \C$ preserves $X\times \C$). We only need to verify that the moment maps $\psi_\mathsf{c} \colon E \to \Lie(T_\mathsf{c})^*$ and $\psi'_\mathsf{c} \colon E' \to \Lie(T_\mathsf{c})^*$ are proper. But this is Lemma~\ref{psic proper}. This proves item 3.
\end{proof}

\begin{proposition}\label{proposition;GramSchmidt}
Let $\mathbf{v}$ be a good valuation on $A=\C[X]$. Assume that there exists a basis $\mathcal{K}' = \{z_1',\dots,z_n'\}$ of $E^*$ which restricts to a Khovanskii basis for $A$ and $\mathbf{v}$.
Then, there exists a $H_\mathsf{a} \times H_\mathsf{c}$-weight basis $\mathcal{K}=\{z_1,\dots,z_n\}$ of $E^*$ such that: 1) $\mathcal{K}$ restricts to a Khovanskii basis for $A$ and $\mathbf{v}$, and 2) the dual basis of $\mathcal{K}$ is an orthonormal basis of $E$.
\end{proposition}

\begin{proof}
Let $\{y_1,\dots,y_n\}$ be a $H_\mathsf{a}\times H_\mathsf{c}$-weight basis of $E^*$. By~\ref{GV1new}, each $y_i$ has weight $(\mathsf{a},\mathsf{c})\circ \mathbf{v}(y_i)$. For each $\lambda\in \Lambda_\mathsf{a}\times \Lambda_\mathsf{c}$, let $I_\lambda = \{ i \in [1,n] \mid (\mathsf{a},\mathsf{c})\circ \mathbf{v}(y_i) = \lambda\}$. 
 Write
\[z_j' = \sum_{i = 1}^n a_i y_i =\sum_{\lambda\in \Lambda_\mathsf{a}\times \Lambda_\mathsf{c}} \sum_{i \in I_\lambda} a_i y_i , \qquad a_i \in \C.\]
Each term $\sum_{i\in I_\lambda} a_i y_i$ is a weight vector of weight $\lambda$. What's more,
$
\mathbf{v}(z_j') \le \max_\lambda \left\{ \mathbf{v}\left(\sum_{i \in I_\lambda} a_i y_i\right) \right\}.
$
By \ref{v4}, each term $\mathbf{v}\left(\sum_{i \in I_\lambda} a_i y_i\right)$ is contained in $(\mathsf{a},\mathsf{c})\n (\lambda)$, and so each of these terms is distinct. By elementary properties of valuations, it follows that
$
\mathbf{v}(z_j') = \max_\lambda \left\{ \mathbf{v}\left(\sum_{i \in I_\lambda} a_i y_i\right)\right\}.
$
By applying $(\mathsf{a},\mathsf{c})$ to both sides of this equation, we find that the right hand side must be contained in $(\mathsf{a},\mathsf{c})\n (\mathbf{v}(z_j') )$. This is impossible unless
$
\mathbf{v}(z_j') =  \mathbf{v}\left(\sum_{i \in I_{(\mathsf{a},\mathsf{c}) (\mathbf{v}(z_j') )}} a_i y_i\right).
$
Let 
$
    z''_j = \sum_{i \in I_{(\mathsf{a},\mathsf{c}) (\mathbf{v}(z_j') )}} a_i y_i.
$
Then by the preceding argument, the linear functions $z_j''$ satisfy $\mathbf{v}(z_j'') = \mathbf{v}(z_j')$ and therefore restrict to a Khovanskii basis for $A$ and $\mathbf{v}$. 

After possibly discarding functions from $\mathcal{K}'$ and reindexing we may assume that the functions $z_1',\dots,z_{n'}'$, where $n'\le n$, have values 
$\mathbf{v}(z_j')$ which are all distinct. (The functions $z_1',\dots,z_{n'}'$ may fail to be a basis for $E^*$, but they will still restrict to a Khovanskii basis for $A$ and $\mathbf{v}$).
Then, the values $\mathbf{v}(z_1''),\dots,\mathbf{v}(z_{n'}'')$ are all distinct, and therefore the functions $z_1'',\dots,z_{n'}''$ are linearly independent. Adding in weight  vectors $y_{n'+1}'',\dots, y_n''$ from $E^*$ as necessary, we arrive at a weight basis $\mathcal{K}''=\{z_1'',\dots,z_{n'}'',y_{n'+1}'',\dots,y_n''\}$ of $E^*$ which restricts to a Khovanskii basis for $A$ and $\mathbf{v}$.

 Finally, following the Gram-Schmidt argument of~\cite[Lemma 3.23]{haradaIntegrableSystemsToric2015}, one can replace $\mathcal{K}''$ with a basis $\mathcal{K}=\{z_1,\dots,z_n\}$ for $E^*$ which admits all the desired properties.
\end{proof}

\subsection{Good valuations and gradient Hamiltonian flows}
\label{main construction section}
Let $X$ be a decomposed affine K\"ahler variety with an algebraic action of $H_\mathsf{a}\times H_\mathsf{c}$. Assume that $X$ is equipped with a good valuation $(X, E,h_E,\mathbf{v},\mathsf{a},\mathsf{c})$. %(Remark \ref{rem; gv determines hamiltonian structure})
We import all the notation from Definition \ref{definition;good}. Given this, we may fix a basis $\mathcal{K} \subset E^*$ that satisfies the conclusions of Proposition \ref{proposition;GramSchmidt}. As demonstrated by Propositions~\ref{prop; embed nicer} and~\ref{proposition;GramSchmidt}, we may assume there exists  $\mathcal{K}$ with these properties without changing the underlying Hamiltonian $T_\mathsf{a} \times T_\mathsf{c}$-space structure on $X$. 

As in Section~\ref{Rees Section}, from $\mathbf{v}$ and $\mathcal{K}$ one can construct a toric degeneration of $X$ which embeds into $E\times \C$. We describe this degeneration in Section \ref{section; application of rees}.  We describe how this degeneration interacts with the symplectic structure in Section \ref{subsection; hamiltonian torus action}. This is combined with Theorem~\ref{thm; toric degenerations main theorem} in Section \ref{subsection; main result for good valuations} to construct an integrable system on $X$.

\subsubsection{Application of the Rees algebra construction}\label{section; application of rees}

Let $\mathsf{S} = \mathsf{S}_\mathbf{v}$ denote the value semigroup of $\mathbf{v}$ and let $X_{\mathsf{S}}$ denote the associated affine toric variety. Applying the Rees algebra construction of Section~\ref{Rees Section} to $\mathbf{v}$ and $\mathcal{K}$ produces a toric degeneration $\pi \colon \X \to \C$ of $X$ to $X_\mathsf{S}$ with the following properties.

\begin{proposition}
\label{proposition; degen from good valuation}
A toric degeneration $\pi \colon \X\to \C$ of $X$ to  $X_{\mathsf{S}}$, constructed from $\mathbf{v}$ and $\mathcal{K}$ as in Section~\ref{Rees Section}, has the following properties:
\begin{enumerate}
    \item \label{item;alg1} $\X$ is embedded as a closed subvariety of $E\times \C$ such that $\pi \colon \X \to \C$ coincides with the restriction of the projection $E\times \C \to \C$.
    \item \label{item;alg2} The fiber $\X_1 \subset E\times \{1\} \cong E$ coincides with the image of the embedding $X\hookrightarrow E$ of the good valuation. In other words, it is cut out by the kernel of the natural map $\C[E]\to \C[A]$ described in \eqref{embedinE}.
    \item \label{item;alg4} For each $\sigma\in \Sigma$, the subfamily $\X^{\overline{\sigma}}$ defined as in \eqref{eq; definition of subfamilies} using the decomposition of $X$ satisfies
    $
    \X^{\overline{\sigma}} = \X\cap (E^{\overline{\sigma}} \times \C).
    $
    \item \label{item;alg5} For each $\sigma\in \Sigma$, let $\mathsf{S}^{\overline{\sigma}} = \mathsf{c}\n(\overline{\sigma})\cap \mathsf{S}$. Then each subfamily $\X^{\overline{\sigma}}\to \C$ is a toric degeneration of $X^{\overline{\sigma}}$ to $X_{\mathsf{S}^{\overline{\sigma}}}$.
    \item \label{item;alg6}
    For each  $\sigma\in \Sigma$, the action of $H_\mathsf{a}\times H_\mathsf{c}$ on $E\times \C$ (where $H_\mathsf{a}\times H_\mathsf{c}$ acts trivially on $\C$) preserves  $\X^{\overline{\sigma}}$.
    \item \label{item;WhitneyA}
    $\X\backslash \X_0$ satisfies the Assumption \ref{D2}.
\end{enumerate}
\end{proposition}

\begin{proof} 
We follow the notation of Section~\ref{Rees Section}. Fix an enumeration $\mathcal{K} = \{z_1,\dots,z_n\} \subset E^*$, let $H = H_{\mathsf{a}}\times H_{\mathsf{c}}$, and let $\mathsf{w} = (\mathsf{a},\mathsf{c})$. Construct the Rees algebra $\mathcal{R}$ as in Section~\ref{Rees Section}, choosing a linear map $e\colon L\to \Z$ as in~\eqref{fix e'}. Let $\X = \spec \mathcal{R}$ and fix the $H$-equivariant embedding $\X \hookrightarrow E\times \C$ as in Section~\ref{Rees Section}. %\eqref{eq; family into ExC embedding}. 
Then  items~\ref{item;alg1} and~\ref{item;alg2} are satisfied by construction.

Let $\C^\times$ act on $E \times \C$ by \[
t\cdot (z_1,\dots,z_n,t') = (t^{e(\mathbf{v}(z_1))}z_1,\dots, t^{e(\mathbf{v}(z_n))} z_n,tt').\]
The trivialization away from zero  of the toric degeneration is written using this action as
\begin{align}
    \label{trivialization}
    \rho\colon X\times \C^\times \to \X\backslash \X_0,\qquad  \rho(z,t) = t\cdot (z,1).
\end{align}
This action of $\C^\times$ preserves $E^{\overline{\sigma}}$,  so $\rho(X^{\overline{\sigma}}\times \C^\times) = \X\cap (E^{\overline{\sigma}}\times \C^\times)$. Taking closures establishes item~\ref{item;alg4}. 
By Proposition~\ref{prop;subfamilies}, there is an isomorphism $\C[\X_0\cap E^{\overline{\sigma}}] \cong \C[\mathsf{S}^{\overline{\sigma}}]$ and $\X^{\overline{\sigma}}\to \C$ is a flat morphism. This establishes item~\ref{item;alg5}. 
Next, the map~\eqref{equation;dual to embedding of the degeneration} preserves the grading by $\Lambda_\mathsf{a}\times \Lambda_\mathsf{c}$, and so the action of $H_\mathsf{a} \times H_{\mathsf{c}}$ on $E\times \C$ preserves $\X$. This action also preserves $E^{\overline{\sigma}}$; putting these two facts together gives item~\ref{item;alg6}.

Finally, consider item~\ref{item;WhitneyA}. Using the trivialization away from zero~\eqref{trivialization}, it suffices to prove the analogous claim for $\X_1\cong X\subset E$. But this is precisely the Whitney condition~\ref{Whitney A} for the stratification of $X$, which holds by assumption~\ref{GV2new} of Definition~\ref{definition;good}.
\end{proof}

\subsubsection{Symplectic geometry of $\X$}\label{subsection; hamiltonian torus action}

Let $\pi\colon \X \to \C$ be a toric degeneration of $X$ constructed from $\mathbf{v}$ and $\mathcal{K}$ as in the previous subsection.  Recall that $\X$ is embedded as a subvariety of $E\times \C$. Let $h_\C$ be the standard Hermitian inner product on $\C$, let $h_E\oplus h_\C$ be the product Hermitian inner product on $E\times \C$ and let $\omega:=\omega_E \oplus \omega_\C$ be the associated symplectic structure. 

As in Section \ref{Rees Section}, let $\H$ denote the algebraic torus with (real) weight lattice $L$. Let $\T$ denote the maximal compact torus of $\H$. Then $\mathbf{v}$ and the basis $\mathcal{K} \subset E^*$ determine a representation of $\H = (\C^\times)^m$ on $E$ as described in Section \ref{Rees Section}. Since $\mathcal{K}$ satisfies the conclusions of Proposition~\ref{proposition;GramSchmidt}, the action of $\T$ on $(E,h_E)$ is unitary. 

Extend the action of $\H$ from $E$ to $E\times \C$ by letting $\H$ act trivially on the second factor. The action of $\T$  on $(E\times \C,\omega)$ is Hamiltonian with moment map
\begin{equation}\label{eqn; TT moment map def}
    \Psi\colon E\times \C \to \Lie(\T)^*,\quad \Psi(z) = \frac{1}{2}\sum_{i=1}^n |z_i|^2 \mathbf{v}(z_i).
\end{equation}
Recall from Section \ref{subsection; kahler form on toric variety}, that we similarly extended the actions of $H_\mathsf{a}$ and $H_\mathsf{c}$ to $E\times \C$. The actions of $T_\mathsf{a}$ and $T_\mathsf{c}$ are Hamiltonian, with moment maps $\psi_\mathsf{a}$ and $\psi_\mathsf{c}$ respectively.\footnote{These symbols were also used to denote moment maps on $E$. To avoid introducing new notation, here we use the same symbol to denote the composition of those maps with the natural projection $E\times \C \to E$.}

Extend $\mathsf{a}$ and $\mathsf{c}$ to $\R$-linear maps $\Lie(\T)^*\to \Lie(T_\mathsf{a})^*$ and $\Lie(\T)^*\to \Lie(T_\mathsf{c})^*$. These maps are dual to the homomorphisms  $T_\mathsf{a} \to \T$ and $T_\mathsf{c} \to \T$, respectively. As in Section \ref{Rees Section}, since $\mathbf{v}$ is compatible with $\mathsf{a}$ and $\mathsf{c}$ by \ref{GV1new}, the actions of $T_\mathsf{a}$ and $T_\mathsf{c}$ on $E\times \C$ coincide with the actions defined by the homomorphisms $T_\mathsf{a} \to \T$ and $T_\mathsf{c} \to \T$, and the action of $\T$ on $E\times \C$.
It follows that:
\begin{equation}
    \label{definition psi c}
    \psi_\mathsf{a}= \mathsf{a} \circ \Psi,\qquad \psi_\mathsf{c} = \mathsf{c} \circ \Psi.
\end{equation}

The action of $\H$ on $E\times \C$ preserves the fiber $\X_0 \subset E\times \{0\}$.
Recall from Proposition~\ref{proposition; degen from good valuation}, item~\eqref{item;alg5}, that the subvariety $\X^{\overline{\sigma}}_0 = \X_0 \cap \X^{\overline{\sigma}}$ is isomorphic to the toric variety $ X_{\mathsf{S}^{\overline{\sigma}}}$ associated to the semigroup  $\mathsf{S}^{\overline{\sigma}} = \mathsf{c}\n(\overline{\sigma})\cap \mathsf{S}$  for each $\sigma\in \Sigma$. 

\begin{proposition}\label{ReyerCor}
Let $\mathbf{v}$ be a good valuation as above, let $\Psi \colon E\times \C \to \Lie(\T)^*$ be the moment map~\eqref{eqn; TT moment map def}, and let $\sigma\in \Sigma$. Then:
\begin{enumerate}
\item  The action of $\H$ on $E\times \C$ preserves $\X_0^{\overline{\sigma}}$.
    \item \label{item;ReyerCorPart2} $\Psi(\X_0^{\overline{\sigma}})= \cone(\mathsf{S}^{\overline{\sigma}}) = \Psi(E^{\overline{\sigma}}\times \C)$ for all $\sigma\in \Sigma$.  In particular,  $\Psi(\X_0) = \cone(\mathsf{S}) = \Psi(E \times \C)$.
    \item $\Psi(E)$ is strongly convex. 
    \item The subspaces $E^{\overline{\sigma}} \times \C$ are saturated by $\Psi$ for all $\sigma\in \Sigma$.
    \item The fibers of the domain restricted map $\Psi\colon \X_0 \to \Lie(\T)^*$ are connected.
    \item The map $(\Psi,\pi)\colon E \times \C \to \Lie(\T)^*\times \C$ is proper. In particular, the domain restricted map $\Psi\colon \X_0 \to \Lie(\T)^*$ is proper.
    \item $ \X^{\overline{\sigma}} = \X\cap \psi_\mathsf{c}\n(\overline{\sigma}),
$ for each 
$\sigma\in \Sigma$.
\end{enumerate}

\end{proposition} 

\begin{proof}
\begin{enumerate}
    \item The map $\C[E] \to \gr A$ in~\eqref{embedinE} is a map of $L$-graded algebras, so $\H$ preserves $\X_0$. Also, $\H$ preserves $E^{\overline{\sigma}}\times \C$. By Proposition~\ref{proposition; degen from good valuation}, item~\ref{item;alg4}, $\X^{\overline{\sigma}}_0 = \X_0 \cap (E^{\overline{\sigma}}\times \C)$.  Thus, $\H$ preserves $\X^{\overline{\sigma}}_0$.
    
    \item As an $\H$-module, the semigroup algebra $\C[\mathsf{S}^{\overline{\sigma}}]$ splits as a direct sum $\C[\mathsf{S}^{\overline{\sigma}}] = \oplus_{\lambda \in \mathsf{S}^{\overline{\sigma}}} V(\lambda)$, where $V(\lambda)$ is the one-dimensional $\H$-module with weight $\lambda\in L$. Then $\cone(\mathsf{S}^{\overline{\sigma}}) = \Psi(\X^{\overline{\sigma}}_0)$ by \cite[Theorem~4.9]{sjamaarConvexityPropertiesMoment1998}. And $\cone(\mathsf{S}^{\overline{\sigma}}) = \Psi(E^{\overline{\sigma}}\times \C)$ by the same reasoning. 
    \item By the explicit description of $\Psi$, $\Psi(E) = \Psi(E\times \C)$. By the previous item, $\Psi(E\times \C) = \cone(\mathsf{S})$. Finally,  $\cone(\mathsf{S})$ is strongly convex by 
    Lemma~\ref{lemma; cones are strongly convex}.
    \item That the subspaces $E^{\overline{\sigma}} \times \C$ are saturated by $\Psi$ is a consequence of Lemma~\ref{lemma; unitary torus action proper moment map}\ref{lemma2.2i}.
    \item Since  $\mathsf{S}$ is saturated by assumption, the variety $\X_0\cong X_{\mathsf{S}}$ is normal~\cite[Theorem 1.3.5]{coxToricVarieties2011}. The claim then follows immediately by \cite[Corollary~4.13]{sjamaarConvexityPropertiesMoment1998}. 
    \item 
     It suffices to prove that the restricted map $\Psi \colon E = E\times \{0\}  \to \Lie(\T)^*$ is proper. This follows immediately from Lemma~\ref{psic proper}, since $\psi_\mathsf{c} = \mathsf{c} \circ \Psi$ and since the restricted map $\mathsf{c} \colon \cone(\mathsf{S}) \to \Lie(T_\mathsf{c})^*$ is proper by~\ref{v3}. Since $\X_0$ is closed in $E$, the domain restricted map $\Psi\colon \X_0 \to \Lie(\T)^*$ is also proper.
    
    \item By Proposition~\ref{proposition; degen from good valuation}, item~\ref{item;alg4}, it follows that $\X^{\overline{\sigma}} = \X\cap(E^{\overline{\sigma}} \times \C) = \X\cap \psi_\mathsf{c}\n(\overline{\sigma})$. \qedhere
\end{enumerate}  
\end{proof}

As in Section \ref{section; Toric degenerations of stratified affine varieties}, denote $\X^\sigma = \X^{\overline{\sigma}} \setminus \cup_{\tau \prec \sigma} \X^{\overline{\tau}}$ and $\X_0^\sigma = \X_0 \cap \X^\sigma$ for all $\sigma \in \Sigma$. It follows by Proposition~\ref{ReyerCor}, item~\ref{item;ReyerCorPart2}, that 
\[
    \Psi(\X_0^\sigma) = \Psi(\X^{\overline{\sigma}}) \setminus \bigcup_{\tau \prec \sigma} \Psi(\X^{\overline{\tau}}_0) = \cone(\mathsf{S}^{\overline{\sigma}}) \setminus \bigcup_{\tau \prec \sigma}\cone(\mathsf{S}^{\overline{\tau}}).
\]
For all $\tau \prec \sigma$, each $\cone(\mathsf{S}^{\overline{\tau}})$ is a face of $\cone(\mathsf{S}^{\overline{\sigma}})$. As a consequence the set $\Psi(\X_0^\sigma)$ is a convex, locally rational polyhedral set. In particular, the smooth locus of $\Psi(\X_0^\sigma)$ is well defined. 

\begin{corollary}\label{cor; smooth locus} As in Section \ref{section; Toric degenerations of stratified affine varieties}, let $U_0^\sigma \subset \X_0^\sigma$ denote the smooth locus of $\X_0^\sigma$. Then:
\begin{enumerate}
    \item The image $\Psi(U_0^\sigma)$ is the smooth locus of the locally rational polyhedral set $\Psi(\X_0^\sigma)$. In particular, $\Psi(U_0^\sigma)$ is a convex, locally rational polyhedral set.
    \item The restricted map $\Psi\colon U_0^\sigma \to \Psi(U_0^\sigma)$ is proper.
\end{enumerate} 
\end{corollary}

\begin{proof}
 The singular locus $(\X_0^{\overline{\sigma}} )^{sing} \subset \X_0^{\overline{\sigma}}$ is saturated by $\Psi \colon \X_0^{\overline{\sigma}} \to \Lie(\T)^*$.
 The image $\Psi((\X_0^{\overline{\sigma}} )^{sing})$ is the complement of the smooth locus of $\cone(S^{\overline{\sigma}})$.
 Since $\X_0^\sigma$ is an open subset of $\X_0^{\overline{\sigma}}$, one has $U_0^\sigma = \X_0^\sigma \setminus (\X_0^\sigma \cap(\X_0^{\overline{\sigma}} )^{sing})$.  Since $(\X_0^{\overline{\sigma}} )^{sing}$ is saturated by $\Psi$,
 this implies
    $
        \Psi(U_0^\sigma) = \Psi(\X_0^{\sigma}) \setminus \Psi((\X_0^{\overline{\sigma}} )^{sing}).
    $
 It follows that $\Psi(U_0^\sigma)$ is the smooth locus of $\Psi(\X_0^{\sigma})$. The convexity of $\Psi(U_0^\sigma)$ then follows from the convexity of $\Psi(\X_0^{\sigma})$. This proves the first claim.
    
   By Proposition \ref{ReyerCor}, for each $\tau \prec \sigma$ the restricted map $\Psi \colon \X_0^{\overline{\sigma}} \to \cone(S^{\overline{\sigma}})$ is proper and the closed subset $\X_0^{\overline{\tau}}$ is saturated by $\Psi$. It follows that
    $
        U_0^\sigma = \X_0^{\overline{\sigma}}\setminus \left(\cup_{\tau \prec \sigma} \X_0^{\overline{\tau}} \cup (\X_0^{\overline{\sigma}} )^{sing} \right)
    $
    is also saturated by $\Psi \colon \X_0^{\overline{\sigma}} \to \cone(S^{\overline{\sigma}})$. 
     Thus, the restricted map $\Psi\colon U_0^\sigma \to \Psi(U_0^\sigma)$ is proper. This proves the second claim.
\end{proof}

\subsubsection{Main result}\label{subsection; main result for good valuations}

Let $\mathbf{v}$ be a good valuation on $A=\C[X]$, and let $\X\subset E\times \C$ be the toric degeneration constructed as in Proposition~\ref{proposition; degen from good valuation}.  Let $T=T_\mathsf{c}$ and $\psi=\psi_\mathsf{c}$. Recall condition~\ref{assumeF} from Section~\ref{section; Toric degenerations of stratified affine varieties}:
\begin{enumerate}[label=(GH\arabic*), start=5]
	\item \label{assumeFrestate} The Duistermaat-Heckman measures of $((\X_0^\sigma)^{sm},\omega_0^\sigma,\psi)$ and $(\X^\sigma_1,\omega^\sigma_1,\psi)$ are equal for all $\sigma \in \Sigma$.
\end{enumerate}
The following theorem shows that good valuations produce integrable systems in the sense of Definition~\ref{def; integrable system on singular symplectic space}.

\begin{theorem} \label{maintheorem} Let $\X$ be as above, and assume condition~\ref{assumeFrestate} holds.
Then, there exists a continuous, surjective, $T_{\mathsf{a}}\times T_{\mathsf{c}}$-equivariant, proper map $\phi \colon \X_1 = X \to \X_0 = X_{\mathsf{S}}$ such that
	\begin{enumerate}[label=(\alph*)]
	    \item For all $\sigma\in \Sigma$,  $D^\sigma := \phi\n(U_0^\sigma)$ is a dense open subset of $\X^\sigma_1$ and 
	    \[
	        \phi\colon (D^\sigma,(\psi_\mathsf{a},\psi_\mathsf{c})) \to (U_0^\sigma,(\psi_\mathsf{a},\psi_\mathsf{c}))
	    \]
	    is an isomorphism of Hamiltonian $T_{\mathsf{a}}\times T_{\mathsf{c}}$-manifolds.
	    
	    \item For each $\sigma\in \Sigma$, 
	    the restricted map $\Psi\circ \phi\colon D^\sigma \to \Lie(\T)^*$ generates a complexity 0 Hamiltonian $\T$-action on $D^\sigma$.
	\end{enumerate}
\end{theorem}

\begin{proof}
We need only check the assumptions of Theorem \ref{thm; toric degenerations main theorem} and Corollary~\ref{thm; toric degenerations main theorem corollary}.
(The proof of equivariance with respect to $T_\mathsf{a}$ follows exactly as the proof for $T_\mathsf{c}$, since the Hamiltonian action of $T_\mathsf{a}$ on $E\times \C$ preserves $\X$, $\pi$, and the K\"ahler metric, cf.~Proposition \ref{prop; gh flow preserving dh measure}).  Condition~\ref{assumeA} holds by items~\eqref{item;alg4} and~\eqref{item;alg5} of Proposition~\ref{proposition; degen from good valuation}. Condition~\ref{assumeB} holds by item~\eqref{item;alg1} of Proposition~\ref{proposition; degen from good valuation} and the definition of the action of $\H$.  Condition~\ref{assumeC} holds by item~\eqref{item;alg2} of Proposition~\ref{proposition; degen from good valuation}, condition~\ref{GV2new} of Definition~\ref{definition;good}, and Lemma~\ref{lemma; unitary torus action proper moment map}. Condition~\ref{assumeE} follows from Lemmas~\ref{psic proper} and~\ref{lemma; unitary torus action proper moment map}, and item~\eqref{item;alg6} of Proposition~\ref{proposition; degen from good valuation}. Condition~\ref{assumeF} holds by assumption. Condition~\ref{assumeG} holds by item~\eqref{item;WhitneyA} of Proposition~\ref{proposition; degen from good valuation} and Proposition~\ref{CTSvfprop}. The additional assumptions of Corollary~\ref{thm; toric degenerations main theorem corollary} hold as described in Section \ref{subsection; hamiltonian torus action}, cf.~Equation~\eqref{eqn; TT moment map def} and Proposition~\ref{ReyerCor}.
\end{proof}

\begin{example}
    The following extends a special case of~\cite[Example 2.9]{pabiniakSymplecticCohomologicalRigidity2020}. It is included as a concrete demonstration of Theorem~\ref{maintheorem} not used in any of the results that follow. Consider the polyhedral cones
    \begin{align*}
        C = \{ p\in \R^4 \mid & 0\le \langle p, e_1\rangle, 0\le \langle p, e_2\rangle, 0\le \langle p, e_3\rangle, 0\le \langle p, e_4\rangle, \\
        & \langle p, e_3\rangle \le \langle p, e_1\rangle, \langle p,e_4\rangle \le \langle p, e_1 + e_2\rangle \}, \\
        C' = \{ p\in \R^4 \mid & 0\le \langle p, e_1\rangle, 0\le \langle p, e_2\rangle, 0\le \langle p, e_3\rangle, 0\le \langle p, e_4\rangle, \\
        & \langle p, e_3\rangle \le \langle p, e_1\rangle, \langle p,2e_3 + e_4 \rangle \le \langle p, 2e_1 + e_2\rangle \}.
    \end{align*}
    Let $\mathsf{S} = C\cap \Z^4$ and $\mathsf{S}' = C'\cap \Z^4$. Additionally, let $C_{\lambda_1,\lambda_2} = C\cap \{p\in \R^4 \mid \langle p,e_1\rangle = \lambda_1, \langle p, e_2 \rangle = \lambda_2\}$ and analogously for $C'_{\lambda_1,\lambda_2}$. We will show that 
    \begin{enumerate}
        \item There exists a good valuation on the toric variety $X_\mathsf{S}$, with value semigroup $\mathsf{S}'$; \label{ptexampleitem1}
        \item The resulting toric degeneration satisfies the hypotheses of Theorem~\ref{maintheorem}; and \label{ptexampleitem2}
        \item For fixed $\lambda_1\ge 0$ and $\lambda_2>0$, the toric manifold with polytope $C_{\lambda_1,\lambda_2}$ is isomorphic as a symplectic manifold ) to the toric manifold with polytope $C'_{\lambda_1,\lambda_2}$. This isomorphism is not necessarily equivariant for any torus action. \label{ptexampleitem3}
    \end{enumerate} This extends what is shown in~\cite[Example 2.9]{pabiniakSymplecticCohomologicalRigidity2020} because it does not assume that $\lambda_1,\lambda_2$ are rational numbers. We expect that the approach in this example can be extended to the other results of~\cite{pabiniakSymplecticCohomologicalRigidity2020}.
    
    Observe that $\mathsf{S}$ is generated by 
    \[
        v_1 = e_2, v_2 = e_2 + e_4, v_3 = e_1, v_4 = e_1+e_3, v_5 = e_1+e_4, v_6=e_1+e_3+e_4.
    \]
    Let $\pi \colon \C[y_1,\dots,y_6] \to \C[(\C^\times)^4]$ be defined by $\pi(y_i) = \chi^{v_i}$.
    By~\cite[Proposition 1.1.4]{coxToricVarieties2011}, the toric variety $X_\mathsf{S}$ can be identified with the subvariety of $E=\C^6$ cut out by $\ker\pi$; therefore as a set, \begin{equation}\label{example_embedding}
    X_\mathsf{S} = \{y \in E \mid y_3y_6=y_4y_5 \text{ and } y_1y_5=y_2y_3\}.
    \end{equation}
    Let $H_\mathsf{a} = \{1\}$ and $H_\mathsf{c} = (\C^\times)^2$. Let $L=\Z^4$ and define a $L$-grading on $\C[y_1,\dots,y_6] = \C[E]$ by $|y_i|_L = v_i$. This grading defines an action of $(\C^\times)^4$ on $E$ which restricts to an action on $X_\mathsf{S}$. Additionally, let $\mathsf{c}\colon L\to \Lambda_\mathsf{c} = \Z^2$ be the projection to the first two coordinates. This defines an action of $H_\mathsf{c}$ on $E$ and $X_\mathsf{S}$.

%     Define a $\Lambda_\mathsf{c}=\Z^2$-grading on $\C[y_1,\dots,y_6] = \C[E]$ by 
%     \[
%         |y_i|_{\Lambda_\mathsf{c}} = \begin{cases}
% e_1 & i \in \{3,4,5,6\}\\
% e_2 & i \in \{1,2\} 
% \end{cases}.
%     \]
%     This grading defines an action of $H_\mathsf{c}$ on $E$ which restricts to an action on $X_\mathsf{S}$.

    Consider the functions 
    \[
        u_1 = y_3, u_2 = y_1, u_3 = \frac{y_4-y_5}{y_3}, u_4 = \frac{y_5}{y_3}
    \]
    on $X_\mathsf{S}$; these functions form a coordinate system on an open dense subset $U$ of $X_\mathsf{S}$. Equip $L=\Z^4$ with the opposite order of the lexicographical order. Following~\cite[Section 2]{pabiniakSymplecticCohomologicalRigidity2020}, we define a valuation $\mathbf{v}\colon \C(X_\mathsf{S})\backslash \{0\} \to L$ 
    relative to this coordinate system. For a fixed $x\in U$, if $f$ is regular at $x$, then we can write it as a power series $f= \sum_{v\in \Z^4_{\ge 0}} c_v \prod_{i=1}^4 u_i^{v_i}$. In this case, $\mathbf{v}(f) = \max\{v | c_v\ne 0\}$. More generally, if $f$ and $g$ are regular at $x$, then $\mathbf{v}(f/g) = \mathbf{v}(f) -\mathbf{v}(g)$. As noted in~\cite{pabiniakSymplecticCohomologicalRigidity2020}, this valuation has one-dimensional leaves. Additionally, any $\Lambda_\mathsf{c}$-homogeneous function $f$ satisfies $|f|_{\Lambda_\mathsf{c}} = \mathsf{c}\circ \mathbf{v}(f)$. 
    
    % If we let $\mathsf{c}\colon L \to \Lambda_\mathsf{c}$ be the projection to the first two coordinates, we find that 
    
    Fix an inner product $h_E$ on $E$ that is invariant under the $(S^1)^4$-action induced by the $L$-grading of $\C[E]$. We then have the data $(X_\mathsf{S}, E, h_E, \mathbf{v}, \mathsf{a}\colon L\to \{0\}, \mathsf{c})$ of a good valuation. One immediately verifies that~\ref{GV1new} holds. The partition of $X_\mathsf{S}$ into orbit types of $H_\mathsf{c}$ clearly endows $X_\mathsf{S}$ with the structure of a decomposed variety. This decomposition satisfies \ref{D2'} since it is an orbit stratification for an algebraic group action with finitely many orbits. Therefore, $(X_\mathsf{S}, E, h_E, \mathbf{v}, \mathsf{a}\colon L\to \{0\}, \mathsf{c})$ satisfies~\ref{GV2new} and is indeed a good valuation.

    To establish~\eqref{ptexampleitem1} above, it remains to check that $\mathbf{v}(\C[X_\mathsf{S}]) = \mathsf{S}'$. Fix $\lambda_1,\lambda_2\in \Z_{\ge 0}$. For $p_3, p_4, t\in \Z_{\ge 0}$, define
    \begin{align*}
        f_{p_3,p_4,t} & = \left(\frac{y_4-y_5}{y_3}\right)^t\left(\frac{y_5}{y_3}\right)^{p_4-t}\left(\frac{y_4}{y_3}\right)^{p_3}y_3^{\lambda_1}y_1^{\lambda_2} \in \C(X_\mathsf{S})
    \end{align*}
    If $p_3=0$, $0\le p_4\le \lambda_1+\lambda_2$, and $(\lambda_1,\lambda_2,p_3+t,p_4-t)\in C$, then a direct computation using the relations~\eqref{example_embedding} shows that $f_{p_3,p_4,t}$ is regular on $X_\mathsf{S}$. Similarly, if $0\le p_3\le \lambda_1$, $p_4=\lambda_1+\lambda_2$, and $(\lambda_1,\lambda_2,p_3+t,p_4-t)\in C$, then $f_{p_3,p_4,t}$ is also regular on $X_\mathsf{S}$. Let $A(\lambda_1,\lambda_2)$ be the set of $f_{p_3,p_4,t}$ satisfying either of these conditions. Then, we can rewrite
    \begin{align*}
        f_{p_3,p_4,t} & = u_3^t u_4^{p_4-t} (u_3+u_4)^{p_3} u_1^{\lambda_1} u_2^{\lambda_2}, 
    \end{align*}
    and so $\mathbf{v}(f_{p_3,p_4,t}) = \lambda_1 e_1 +\lambda_2 e_2 + t e_3 + (p_3+p_4 - t)e_4$. One can then check that $\mathbf{v}(A(\lambda_1,\lambda_2)) = C'_{\lambda_1,\lambda_2} \cap \Z^4$, and so $C'_{\lambda_1,\lambda_2} \cap \Z^4 \subset \mathbf{v}(\C[X_\mathsf{S}]_{\lambda_1,\lambda_2})$, where $\C[X_\mathsf{S}]_{\lambda_1,\lambda_2}$ is the $(\lambda_1,\lambda_2)$-homogeneous part of $\C[X_\mathsf{S}]$. Moreover, because $\mathbf{v}$ has one-dimensional leaves,  the dimension of $\C[X_\mathsf{S}]_{\lambda_1,\lambda_2}$ is equal to the cardinality of  $\mathbf{v}(\C[X_\mathsf{S}]_{\lambda_1,\lambda_2})$. The dimension of $\C[X_\mathsf{S}]_{\lambda_1,\lambda_2}$ is also equal to the cardinality of $C_{\lambda_1,\lambda_2} \cap \Z^4$. But $|C_{\lambda_1,\lambda_2} \cap \Z^4| = |C'_{\lambda_1,\lambda_2} \cap \Z^4|$, and so $ C'_{\lambda_1,\lambda_2} \cap \Z^4 = \mathbf{v}(\C[X_\mathsf{S}]_{\lambda_1,\lambda_2})$. Since this applies to all values of $\lambda_1,\lambda_2$, we have that $\mathbf{v}(\C[X_\mathsf{S}]) = \mathsf{S}'$, as desired. This establishes claim~\eqref{ptexampleitem1}.

    To prove claim~\eqref{ptexampleitem2}, we must show that the toric degeneration associated with $\mathbf{v}$ satisfies~\ref{assumeF}. On $X_\mathsf{S}\subset\X_1$, the moment map for the $T_\mathsf{c}$-action factors through the moment map for the $(S^1)^4$-action induced by the original $L$-grading of $\C[E]$. Applying~\cite[Theorem~4.9]{sjamaarConvexityPropertiesMoment1998}, the symplectic volume of the reduced space over $(\lambda_1,\lambda_2) \in \R^2_{\ge 0}$ is, up to a constant multiple, the volume of the polytope $C_{\lambda_1,\lambda_2}$. Similarly, on $X_\mathsf{S'}\subset\X_0$, the moment map for the $T_\mathsf{c}$-action factors factors through the moment map for the action of $\mathbb{T}$, and the symplectic volume of the reduced space over $(\lambda_1,\lambda_2)\in \R^2_{\ge 0}$ is, up to a constant multiple, the volume of $C'_{\lambda_1,\lambda_2}$. When $\lambda_1=0$, then $C_{\lambda_1,\lambda_2} = C'_{\lambda_1,\lambda_2}$. When $\lambda_1>0$, $C_{\lambda_1,\lambda_2}$ is a rectangle with height $\lambda_1+\lambda_2$ and width $\lambda_1$. On the other hand, $C'_{\lambda_1,\lambda_2}$ is a trapezoid with bases $2\lambda_1+\lambda_2$ and $\lambda_2$, and height $\lambda_1$. So $C_{\lambda_1,\lambda_2}$ is of equal area as $C'_{\lambda_1,\lambda_2}$, which establishes claim~\eqref{ptexampleitem2}.

    The final claim~\eqref{ptexampleitem3} results from looking at the smooth locus of $X_{\mathsf{S}'}$: In the decomposition by faces of $\mathsf{c}(C')$, the piece of $X_{\mathsf{S}'}$ which corresponds to the interior of $\mathsf{c}(C')$ is smooth, as is the piece which corresponds to the face $\R_{>0}\cdot e_2$. By Theorem~\ref{maintheorem}(a), the map $\phi$ restricts to an isomorphism of Hamiltonian $T_\mathsf{c}$-manifolds on each of these pieces. The claim then follows from the classification of toric symplectic manifolds.
    
\end{example}

\section{Example: toric degenerations of the affine closure of base affine space}\label{section; toric degenerations of G/N} 

This section applies our results from Sections \ref{section; Toric degenerations and gradient Hamiltonian flows} and \ref{good valuation section} to the example of the the affine closure of base affine space of a  semisimple complex Lie group. This space is an important singular affine variety for which many families of valuations have been studied in the literature. As we show below, many of these valuations are good valuations in the sense of Definition \ref{definition;good}. Using our results from Section \ref{good valuation section}, this allows us to produce integrable systems on the affine closure of base affine space. This example is an important ingredient in Section \ref{section 5}, where we use it to construct collective integrable systems on arbitrary Hamiltonian spaces.

The organization of this section is as follows. Subsection \ref{G mod N} establishes the decomposed K\"ahler variety structure of the affine closure of base affine space used in our construction.  Subsection \ref{section; valuations on g mod n} describes how to produce  valuations on the affine closure of base affine space from valuations on $G/B$. We provide this section, which is somewhat expository, because valuations on $G/B$ are  prominent in the literature (see Remark~\ref{rem; other valuations}).   In Subsection \ref{general G mod N} we show that these valuations are good valuations and apply our results from Sections \ref{section; Toric degenerations and gradient Hamiltonian flows} and \ref{good valuation section} to produce the desired continuous maps.  We end this section with Example \ref{main example} which describes the case of string valuations.

\subsection{The affine closure of base affine space}\label{G mod N} This section recalls basic facts about the K\"ahler geometry of the affine closure of base affine space. Additional details and references can be found in \cite[Section 6]{guilleminSymplecticImplosion2002}. Let $G$ be a complex semisimple Lie group and import all notation from  Section~\ref{Lie theory section}. 
The affine closure of the base affine space of $G$, denoted $G\sslash N$, is the affine $G\times H$-variety whose coordinate ring is the ring  $\C[G]^N$ of invariants for the action of $1\times N \subset G\times G$. It is the affine closure of the base affine space $G/N$ studied in~\cite{bggBasic}. An alternate description of the algebra $\C[G]^N$ and its finite generating set is given in~\eqref{equation;Ninvariants} below.

View $G\sslash N$ as a $H$-variety with respect to the action of $1\times H \subset G\times H$. Then, in the notation of Section \ref{section; affine G-varieties},  $\Lambda(G\sslash N) = \Lambda_+$, $\Gamma(G\sslash N) = \ttt^*_+$, and $\Sigma = \Sigma(G\sslash N)$ is the set of open faces of  $\ttt^*_+$.  Define $(G\sslash N)^{\overline{\sigma}}$ and $(G\sslash N)^\sigma$ as in Section~\ref{section; affine G-varieties} for $\sigma\in \Sigma$.  The partition of $G\sslash N$ by the subvarieties $(G\sslash N)^\sigma$ coincides with the orbit decomposition of $G\sslash N$ with respect to the action of $G\times 1 \subset G\times H$. In particular, each $(G\sslash N)^\sigma$ is a smooth manifold in the analytic topology and this decomposition of $G\sslash N$ satisfies~\ref{D1'}. The piece $(G\sslash N)^\sigma$ is isomorphic as an algebraic $G$-homogeneous space to the quotient $G/[P_\sigma,P_\sigma]$, where $P_\sigma $ is the parabolic subgroup of $G$ with Lie algebra
 \begin{equation}\label{eq; P sigma def}
    \mathfrak{p}_\sigma = \hhh \oplus \nnn \oplus \bigoplus_{\alpha \in R_{+,\sigma}} \g_{-\alpha}, \quad R_{+,\sigma } = \{\alpha \in R_+ \mid \lambda(\alpha^\vee) = 0,~\forall \lambda \in \sigma \},
 \end{equation}
 and $[P_\sigma,P_\sigma]$ is its commutator subgroup~\cite[Lemma 6.2]{guilleminSymplecticImplosion2002}. The open dense piece of $G\sslash N$ is isomorphic to the $G$-homogeneous space $G/N$.

Fix a finite set $\Pi \subset \Lambda_+$ that generates $\Lambda_+$ as a semigroup. Define the $G\times H$-module
\begin{equation}\label{eq; definition of E given Pi}
	E =  \bigoplus_{\varpi\in \Pi} V(\varpi)
\end{equation}
where $V(\varpi)$ is the irreducible $G$-module with highest weight $\varpi$ and $1\times H\subset G\times H$ acts on $V(\varpi)$ with weight $-\varpi$. The dual vector space $E^*$ generates $\C[G]^N$ as an algebra with respect to the embedding of $G\times H$-modules,
\begin{equation}
\label{equation;Ninvariants}
	E^* =  \bigoplus_{\varpi\in \Pi} V(\varpi)^* \subset \bigoplus_{\lambda \in \Lambda_+} V(\lambda)^* \cong \C[G]^N.
\end{equation}
The isomorphism on the right is a consequence of the $G\times G$-module decomposition of $\C[G]$ and depends on a choice of highest weight vectors $v(\lambda) \in V(\lambda)$~\cite[Theorem 27.3.9]{tauvel2005lie}. Dualizing, one obtains a $G\times H$-equivariant embedding of  $G\sslash N$ as a closed subvariety of $E$. 

Recall that $K$ is a maximal compact subgroup of $G$, and $T=H\cap K$. Let $h_E$ denote the unique $K\times T$-invariant Hermitian inner product on $E$ such that $||v(\varpi)|| = 1$ for all $\varpi\in \Pi$ and let $\omega_E = -\Im h_E$. Note that since the action of $K\times T \subset G \times H$ on $(E,h_E)$ is unitary by definition, it is Hamiltonian with respect to $\omega_E$.  Since the direct summands $V(\varpi)$ are $1\times T$-weight spaces with distinct weights, the direct sum \eqref{eq; definition of E given Pi} is orthogonal with respect to $h_E$.   The embedding of $G\sslash N$ into $E$ endows it with the structure of a decomposed affine K\"ahler variety, $(G\sslash N,E,\omega_E)$. 

\subsection{Valuations on the affine closure of base affine space}\label{section; valuations on g mod n} For more details about the relation between valuations on $G\sslash N$ and valuations on $G/B$, we refer the reader to~\cite{kavehCrystalBasesNewton2015}.

Let $\Lambda$ be the weight lattice of $H$, let $m = \dim_\C(G/H)$, and let $\mathsf{c}\colon \Z^m\times \Lambda \to \Lambda$ denote projection to $\Lambda$. Use a refinement of the standard partial order on $\Lambda$ and the standard lexicographic order to define a total order on  $\Z^m\times \Lambda$ as in \cite[p.~2492]{kavehCrystalBasesNewton2015}. 

Valuations on $G\sslash N$ can be constructed from valuations on $G/B$ as follows.
Let $\C(G/B)$ be the algebra of rational functions on $G/B$, and let $\nu \colon \C(G/B)\setminus \{0\} \to \Z^m$ be a valuation with one-dimensional leaves. Assume that the image of $\nu$ generates $\Z^m$ as a group (without loss of generality, if it does not, then replace $\Z^m$ with the $\Z$-submodule generated by the image of $\nu$). Let $\iota\colon N_- \to G/B$ denote the embedding $n_- \mapsto n_-B$. This embedding restricts to an isomorphism $N_-\cong N_- B/B$, and induces an algebra isomorphism $\iota^* \colon \C[N_-B/B] \cong \C[N_-]$. For $0\ne F\in \C[N_- B/B]\subset \C(G/B)$, define
\begin{equation}\label{def; nu N minus}
    \nu\vert_{N_-} \colon \C[N_-]\setminus\{0\} \to \Z^m, \quad \nu\vert_{N_-}(\iota^*F) = \nu(F).
\end{equation}
This defines a valuation on $\C[N_-]$. Since $\nu$ was assumed to have one-dimensional leaves on $\C(G/B)$, it is easy to check that $\nu|_{N_-}$ has one-dimensional leaves on $\C[N_-]$.  Next, consider the composition
\begin{equation}\label{eq; kaveh embedding}
    j\colon N_-\times H \hookrightarrow G/N \hookrightarrow G\sslash N
\end{equation}
The first map is $(n,h) \mapsto nh N$. It identifies $N_-\times H$ with the open subset $B_-N \subset G/N$. The second map is inclusion of the dense $G$-orbit in $G\sslash N$. With respect to the isomorphism fixed in \eqref{equation;Ninvariants}, the dual map  
\begin{equation}
    j^* \colon \bigoplus_{\lambda \in \Lambda_+} V(\lambda)^* \cong \C[G\sslash N]  \to \C[N_-\times H] \cong \C[N_-]\otimes \C[H]
\end{equation}
has the property that for $z_\lambda \in V(\lambda)^*$, $j^*z_\lambda = f_\lambda \otimes \chi^\lambda$. 
Here $\chi^\lambda \in \C[H]$ denotes the character on $H$ corresponding to $\lambda$ and $f_\lambda \in \C[N_-]$ is given by $f_\lambda(n) = z_\lambda(n\cdot v(\lambda))$, where $v(\lambda) \in V(\lambda)$ is the chosen highest weight vector and $n\in N_-$. Given  $z= \sum_{\lambda \in \Lambda_+} z_\lambda$ in $\C[G\sslash N]$, let $\lambda = \max\{ \gamma | z_\gamma\neq 0\}$. Define
\begin{equation}\label{eq; general form of v}
    \mathbf{v} \colon \C[G\sslash N]\setminus \{0\} \to \Z^m \times \Lambda, \quad \mathbf{v}(z) = (\nu\vert_{N_-}(f_\lambda),\lambda).
\end{equation}
Then by \cite[Proposition 6.1]{kavehCrystalBasesNewton2015}, $\mathbf{v}$ is a valuation with one-dimensional leaves, $\mathsf{S}_\mathbf{v}$ generates $\Z^m \times \Lambda$ as a group, and $\mathsf{c}(\mathsf{S}_\mathbf{v}) = \Lambda_+$ .

The action of $H$ on  $G/B$ defines a $\Lambda$-grading of $\C(G/B)$. Many valuations on $G/B$ have the following additional property:

\begin{enumerate}[label=(*)]
    \item \label{BA1} There exists a linear map $\mathsf{a}' \colon \Z^m \to \Lambda$ such that if $f\in \C(G/B)$ is $\Lambda$-homogeneous of degree $\delta$, then $\mathsf{a}'(\nu(f)) = \delta$.
\end{enumerate}

The following lemma summarizes the important properties of $\mathbf{v}$ and the maps $\mathsf{a}$ and $\mathsf{c}$.

\begin{lemma}\label{prop; bold v satisfies properties}
Assume there exists a linear map $\mathsf{a}' \colon \Z^m \to \Lambda$ satisfying~\ref{BA1}. Then the valuation $\mathbf{v}$ and the map $\mathsf{c}$ satisfy \ref{v1}-\ref{v5}.  The valuation $\mathbf{v}$ and the map $\mathsf{a} = \mathsf{a}' - \mathsf{c}$ satisfy \ref{v4}.
\end{lemma}

\begin{proof} The valuation $\mathbf{v}$ and the map $\mathsf{c}$ satisfy \ref{v1} by definition of the total orders on $\Z^m \times \Lambda$ and $\Lambda$.  They satisfy \ref{v2} since $\mathsf{c}(\mathsf{S}_\mathbf{v}) = \Lambda_+$ and $G$ is semisimple. Recall the line bundle on $ G/P_\sigma$ associated to $\lambda \in \sigma$. There is a Newton-Okounkov body associated to this line bundle, together with the valuation $\nu$ (see \cite[Definition 3.7]{haradaIntegrableSystemsToric2015}).
Condition \ref{v3} is satisfied since $\mathsf{c}\n(\lambda)$ can be identified with the integral points of this Newton-Okounkov body. Condition \ref{v5} is satisfied by the same argument as in \cite[Proposition 2.2]{alexeevToricDegenerationsSpherical2005}.

Finally,  we show that $\mathbf{v}$ and both $\mathsf{a}$ and $\mathsf{c}$ satisfy \ref{v4}. Suppose $z \in  \C[G\sslash N]$ is $\Lambda\times \Lambda$-homogeneous of degree $(\gamma ,\lambda)$. Then $j^*z = f \otimes \chi^\lambda$ for some $f \in \C[N_-]$ that is $\Lambda$-homogeneous of degree $\gamma +\lambda$ (with respect to the conjugation action of $H \times 1$ on $N_-$).
Thus, 
\[
    \mathsf{c}(\mathbf{v}(z)) = \mathsf{c}(\nu\vert_{N_-}(f),\lambda) = \lambda \quad \text{and} \quad 
    \mathsf{a}(\mathbf{v}(z)) = \mathsf{a'}(\nu\vert_{N_-}(f)) - \mathsf{c}(\nu\vert_{N_-}(f),\lambda) = \gamma. \qedhere
\]
\end{proof}

\subsection{Good valuations and toric degenerations of the affine closure of base affine space} 
\label{general G mod N}

\begin{data}\label{def; data for valuation on g mod n} Let $H \times H = H_\mathsf{a} \times  H_\mathsf{c}$. We consider $G\sslash N$ as an affine $H\times H$-variety with respect to the action of $H\times H \subset G\times H$. Fix  data $(G\sslash N, E, h_E, \mathbf{v},\mathsf{a},\mathsf{c})$ as follows.

\begin{enumerate}[label=(\roman*),start=1]
\item Let $(E,h_E)$ be defined as in Section \ref{G mod N}. The representation of $T\times T $ on $(E,h_E)$ is unitary and the embedding $G\sslash N \hookrightarrow E$ dual to \eqref{equation;Ninvariants} is $H\times H$-equivariant.
    \item Let $\mathbf{v}\colon \C[G\sslash N]\backslash\{0\} \to L$ be a valuation constructed as in \eqref{eq; general form of v}. Assume that it has one-dimensional leaves and $\mathsf{S}_\mathbf{v}$ is finitely generated and saturated.
    \item Assume there exists a map $\mathsf{a}'$ as in Lemma~\ref{prop; bold v satisfies properties}. Let $\mathsf{c} \colon  L \to \Lambda$ be defined as in Section \ref{section; valuations on g mod n}, and let $\mathsf{a} = \mathsf{a}'-\mathsf{c}$.
 \end{enumerate}
\end{data}

Note that the maps $\mathsf{a}$ and $\mathsf{c}$ from (iii) are both surjective by definition.  In the notation of Definition \ref{definition;good}, $\Lambda(G\sslash N)= \Lambda_+$, $\Gamma = \Gamma(G\sslash N) =\ttt_+^*$, and $\Sigma=\Sigma(G\sslash N)$ is the face poset of $\ttt_+^*$. Denote the quadratic moment maps of the Hamiltonian actions of $T_\mathsf{a} = T$ and $T_\mathsf{c} = T$ on $E$ by  $\psi_\mathsf{a}$ and $\psi_\mathsf{c}$ respectively.

\begin{lemma}\label{lemma; bold v defines good valuation}
Data \ref{def; data for valuation on g mod n} define a good valuation on $G\sslash N$.
\end{lemma}

\begin{proof} We just need to check compatibility of the valuation, \ref{GV1new}, and compatibility with the decomposition, \ref{GV2new}.  Compatibility of the valuation was shown in Lemma \ref{prop; bold v satisfies properties}.  
As described in Section \ref{G mod N}, the partition of $G\sslash N$ with respect to the action of $H_\mathsf{c} = 1 \times H \subset G \times H$ coincides with the partition into orbits for the action of $G$. This finite partition endows $G\sslash N$ with the structure of a decomposed variety. In particular, this decomposition satisfies \ref{GV2new} since it is an orbit stratification for an algebraic group action with finitely many orbits.
\end{proof}

For the remainder of this section, assume we have fixed a choice of Data \ref{def; data for valuation on g mod n}. Let $\H = (\C^\times)^m \times H$ be the algebraic torus with (real) weight lattice $\Z^m\times \Lambda$. Let  $X_\mathsf{S}$ denote the toric $\H$-variety associated to the value semigroup $\mathsf{S}$.  The good valuation can be used to construct a toric degeneration $\pi\colon \X \to \C$ of $G\sslash N$ to $X_\mathsf{S}$ that embeds as a subvariety of $E\times \C$ (Proposition \ref{proposition; degen from good valuation}). This also fixes an embedding of $X_\mathsf{S}$ into $E$ such that the action of the compact subtorus $\T = (S^1)^m\times T$ on $X_\mathsf{S}$ is Hamiltonian, and generated by the restriction of the moment map $\Psi\colon E\times \C \to \Lie(\T)^*$ defined in \eqref{eqn; TT moment map def}.

Extend the action of $T\times T$ on $E$ trivially to $E\times \C$ and let  $(\psi_\mathsf{a},\psi_\mathsf{c})\colon E\times \C \to \ttt^* \times \ttt^*$ denote the moment map for this action. This map satisfies~\eqref{definition psi c}.
The action of $T\times T$ on $E\times \C$ preserves the subvariety $\X$ as well as the fibers of $\pi$. The restriction of this action to $\X_1 \cong G\sslash N$ coincides with the action of $T\times T$ on $G\sslash N$ as the maximal torus of $K\times T \subset G\times H$. The restriction of this action to $\X_0 \cong X_\mathsf{S}$ is given by the inclusion $T\times T \subset \T$ dual to $\mathsf{a} \times \mathsf{c}$.

For each $\sigma\in \Sigma$, recall that $\X_0^{\sigma} = \X_0 \cap E^{\sigma}$. Denote $U_0^\sigma = (\X_0^\sigma)^{sm}$. 

\begin{proposition}
\label{prop; DH measure}
For all $\sigma \in \Sigma$, the Duistermaat-Heckman measures of  $(U_0^\sigma,\omega_0^\sigma,\psi_\mathsf{c})$ and $(\X_1^\sigma,\omega_1^\sigma,\psi_\mathsf{c})$ are the same.
\end{proposition}

\begin{proof}
We prove the proposition for the special case where $\sigma$ is the interior of $\ttt_+^*$. The other cases are similar except that the kernel of the action of $T$ requires more careful attention.

Let $\nu_1$ denote the Duistermaat-Heckman measure of $(\X_1^\sigma,\omega_1^\sigma,\psi_\mathsf{c})$. The moment map $\psi_\mathsf{c}$ is proper as a map to $\sigma$. The action of $T$ on $\X_1^\sigma$ is free, so every value of $\psi$ is a regular value. Thus, $\nu_1=f(\lambda)dm$, where $f(\lambda)$ is the symplectic volume of the symplectic reduction of $(\X_1^\sigma,\omega_1^\sigma,\psi_\mathsf{c})$ at $\lambda$.  Because $(\X_1^\sigma,\omega_1^\sigma,\psi_\mathsf{c})$ is isomorphic as a Hamiltonian $T$-space to the open dense symplectic piece of the symplectic implosion (equipped with the torus action of $1\times T \subset K \times T$, cf.~Section \ref{section; symplectic implosion}) \cite[Theorem A.1]{hilgertContractionHamiltonianKSpaces2017}, and because reduction of the symplectic imploded space by the torus $1\times T$ produces a quotient space symplectomorphic to the reduction of $T^*K$ by $1 \times K$ (at the same value), it follows that this symplectic reduction is symplectomorphic to the coadjoint orbit $\mathcal{O}_\lambda$ equipped with its canonical symplectic form, $\omega_\lambda$. Thus, for all $\lambda\in \sigma$, $f(\lambda) =  \Vol(\mathcal{O}_\lambda,\omega_\lambda)$. This function is continuous and has the scaling property that for all $\alpha >0$ and $\lambda$, $ f(\alpha \lambda)= \alpha^m f(\lambda)$.

Let $\nu_0$ denote the Duistermaat-Heckman measure of $(U_0^\sigma,\omega_0^\sigma,\psi_\mathsf{c})$. It follows from Corollary \ref{cor; smooth locus} that $\Psi(U_0^\sigma)$ is the smooth locus of $\Psi(\X_0^\sigma)$. Both  $\Psi(U_0^\sigma)$ and $\Psi(\X_0^\sigma)$ are dense in the rational polyhedral cone $\Psi(\X_0)$. In particular, $\Psi(U_0^\sigma)$ is convex and the restricted map $\Psi\colon U_0^\sigma \to \Psi(U_0^\sigma)$ is proper. It follows by the classification of proper complexity 0 torus manifolds \cite[Proposition 6.5]{karshonNonCompactSymplecticToric2015} that the Duistermaat-Heckman measure of $\Psi$ equals $\chi dM$. Here, $\chi$ is the indicator function of $\Psi(U_0^\sigma)$ and $dM= dx \times dm$ is the Lebesgue measure on $\Lie(\T)^* = \R^m \times \ttt^*$ determined by $\Z^m \times \Lambda$. It follows by Tonelli's theorem that $\nu_0 = g(\lambda)dm$, where 
\[
    g(\lambda):= \Vol(\pr_{\R^m}(\mathsf{c}\n(\lambda) \cap \Psi(\X_0)),dx).
\]
This function is continuous has the scaling property that $g(\alpha\lambda) = \alpha^m g(\lambda)$, for all $\alpha>0$ and $\lambda\in \sigma$.

It remains to show that $f(\lambda) = g(\lambda)$ for all $\lambda \in \sigma$. Since $f$ and $g$ are both continuous and share the same scaling property, it suffices to show that $f(\lambda) = g(\lambda)$ for all $\lambda \in \sigma \cap \Lambda$.  If $\lambda$ is integral, then the coadjoint orbit  $(\mathcal{O}_\lambda,\omega_\lambda)$ is symplectomorphic to $G/B$, where $G/B$ is equipped with a Fubini-Study K\"ahler form by embedding into $\mathbb{P}(V(\lambda))$. It follows by \cite[Theorem B]{haradaIntegrableSystemsToric2015}\footnote{If $\mathbf{v}$ is constructed from a string valuation as in Example \ref{main example}, then this equality  can also be deduced from properties of canonical bases and string polytopes.} that
\[
    f(\lambda) = \Vol(\mathcal{O}_\lambda,\omega_\lambda)  = \Vol(\triangle(R_\lambda,\nu),dx),
\]
where $\triangle(R_\lambda, \nu)$ is the Newton-Okounkov body associated to $\lambda$ and $\nu$. This Newton-Okounkov body is nothing but $\pr_{\R^m}(\mathsf{c}\n(\lambda) \cap \Psi(\X_0))$, so $\Vol(\triangle(R_\lambda,\nu),dx) = g(\lambda)$. Thus $f(\lambda) = g(\lambda)$, as desired.
\end{proof}

Combining Lemma \ref{lemma; bold v defines good valuation},  Proposition \ref{prop; DH measure}, and Theorem \ref{maintheorem}, we have the following Theorem. 
\begin{theorem} \label{maintheoremGmodN} 
Let $G$ be a connected semisimple complex Lie group and fix a choice of Data \ref{def; data for valuation on g mod n}. Let $X_\mathsf{S}$ denote the toric variety constructed from these data and embedded into $(E,h_E)$ as above. 
Then, there exists a continuous, surjective, $T\times T$-equivariant, proper map $\phi \colon G\sslash N \to X_{\mathsf{S}}$ such that
	\begin{enumerate}[label=(\alph*)]
	    \item For each $\sigma\in \Sigma$, 
	    $D^\sigma := \phi\n(U_0^\sigma)$ is a dense open subset of $\X^\sigma_1$ and  
	    \[
	        \phi\colon (D^\sigma,(\psi_\mathsf{a},\psi_\mathsf{c})) \to (U_0^\sigma,(\psi_\mathsf{a},\psi_\mathsf{c}))
	    \]
	    is an isomorphism of Hamiltonian $T\times T$-manifolds.
	    
	    \item For each $\sigma\in \Sigma$, 
	    the restricted map $\Psi\circ \phi\colon D^\sigma \to \Lie(\T)^*$ generates a complexity 0 Hamiltonian $\T$-action on $D^\sigma$.
	\end{enumerate}
\end{theorem}

Theorem \ref{maintheoremGmodN} can be applied in the following example.

\begin{example}[String valuations] \label{main example}
Let $\mathbf{i}=(i_1,\dots,i_m)$ be a reduced word for the longest element of the Weyl group of $G$, expressed in terms of simple reflections associated with simple real roots $\alpha_{i_1},\dots,\alpha_{i_m}$. Let $\nu_\mathbf{i}\colon \C(G/B)\backslash\{0\} \to \Z^m$ be the valuation constructed in \cite[Section 2.2]{kavehCrystalBasesNewton2015}. Specifically, $\nu_\mathbf{i}$ is defined from the highest-term valuation associated to the standard coordinate chart on the Bott-Samelson variety associated to $\ii$. 
Define
\[
\mathsf{a}' \colon \Z^m \to \Lambda, \quad \mathsf{a}'(v_1,\dots,v_m) = \sum_{j=1}^m v_j \alpha_{i_j}.
\] 
The map $\mathsf{a}'$ and the valuation $\nu_\ii$ satisfy \ref{BA1} by \cite[Proposition 3.8]{kavehCrystalBasesNewton2015}. Define $\mathbf{v}_\mathbf{i}\colon \C[G\sslash N]\backslash\{0\} \to \Z^m \times \Lambda$ using $\nu=\nu_\ii$ as in~\eqref{eq; general form of v}.
By \cite[Proposition 6.1]{kavehCrystalBasesNewton2015}, $\mathbf{v}_\mathbf{i}$ has 1-dimensional leaves and the value semigroup $\mathsf{S}_\ii = \mathsf{S}_{\mathbf{v}_\ii}$ coincides with the set of integral points of a rational convex polyhedral cone known as the \emph{extended string cone associated to $\ii$}. The extended string cones were introduced in \cite{berensteinTensorProductMultiplicities2001} and~\cite{littelmannConesCrystalsPatterns1998} as parametrizations of the dual canonical basis. Extended string cones are known to be rational \cite{littelmannConesCrystalsPatterns1998}. It follows that   $\mathsf{S}_\ii$ is finitely generated and saturated. Thus, the tuple $(G\sslash N, E, h_E, \mathbf{v}_\ii, \mathsf{a},\mathsf{c})$ is an example of Data~\ref{def; data for valuation on g mod n}, and
Theorem~\ref{maintheoremGmodN} holds with $X_\mathsf{S} = X_{\mathsf{S}_\ii}$. In particular, every extended string cone is the image of an integrable system on $G\sslash N$.
\end{example}

\begin{example}\label{ex; flp valuations} Let $\underline{\beta} = \{ \beta_1, \dots, \beta_m\}$ be an  enumeration of the positive roots of $G$. One can use $\underline{\beta}$ to define a valuation  $\nu_{\underline{\beta}}$ on $\C(G/B)$ as in \cite[Section 3]{fangSimplicesNewtonOkounkovBodies2018}. It is easy to check that 
\[
\mathsf{a}' \colon \Z^m \to \Lambda, \quad \mathsf{a}'(v_1,\dots,v_m) = \sum_{i=1}^m v_i \beta_{i},
\] 
together with $\nu_{\underline{\beta}}$ satisfy \ref{BA1}. Define $\mathbf{v}_{\underline{\beta}}\colon \C[G\sslash N]\backslash\{0\} \to \Z^m \times \Lambda$ using $\nu=\nu_{\underline{\beta}}$ as in~\eqref{eq; general form of v}. It follows for the same reasons as in the previous example that $\mathbf{v}_{\underline{\beta}}$ has 1-dimensional leaves. In order to apply Theorem~\ref{maintheoremGmodN}, it remains to show that the value semigroup $\mathsf{S}_{\mathbf{v}_{\underline{\beta}}}$ is finitely generated and saturated. We will return to this problem in Section~\ref{ss; gromov width}.
\end{example}

\begin{remark}\label{rem; other valuations} Other examples of valuations in the literature that are  related to $G\sslash N$ include those described by Fujita-Naito~(which recover the Nakashima-Zelevinsky polytopes) \cite{fujitaNewtonOkounkovConvex2017a}, those described by Kiritchenko and Feigin-Fourier-Littelmann~(which recover the Feigin-Fourier-Littelmann-Vinberg polytopes) \cite{kiritchenkoNewtonOkounkovPolytopes2019, feiginFavourableModulesFiltrations2017}, and the cluster-theoretic valuations recently studied by Fujita-Oya in~\cite{fujitaNewtonOkounkovPolytopesSchubert2020a}.  Although we do not study these valuations here, we expect many of them will produce good valuations on $G\sslash N$.  We also remark that the valuations in the previous two examples are special cases of valuations defined by birational sequences which were studied in \cite{fangEssentialBases2017}.  Such valuations all satisfy the equivariance property \ref{BA1}. The problem of  whether the value semigroups of valuations defined by birational sequences are finitely generated and saturated is discussed in more detail in \cite{fangEssentialBases2017}.
\end{remark}

\section{Collective integrable systems on Hamiltonian $K$-spaces}\label{section 5}

Let $K$ be a compact connected Lie group with Lie algebra $\mathfrak{k}$. Let $Y$ be a singular symplectic space;Recall from Section~\ref{section; ham grp actions} that this means $Y$ is a locally compact, paracompact, Hausdorff, second countable topological space equipped with a locally finite partition by locally closed subspaces, each equipped with the structure of a connected symplectic manifold. Assume $Y$ is equipped with a Hamiltonian action of $K$ with moment map $\mu$. This section constructs collective integrable systems on $Y$. The first step is to construct maps $F \colon \kk^* \to \Lie(\T)^*$ that are integrable systems with respect to the canonical Lie-Poisson structure on $\kk^*$. This is achieved in Section~\ref{section; generalized GZ systems} using the integrable systems on the affine closure of base affine space from Section~\ref{section; toric degenerations of G/N}.  The second step is to pull back these integrable systems to $Y$ via $\mu$ and show that the resulting collective integrable systems have the desired properties. This is done in Section~\ref{section; toric contractions}. As a by-product, we introduce a new singular symplectic space $X_Y$, called the \emph{toric contraction} of $Y$, and a map $\Phi\colon Y \to X_Y$, called the \emph{toric contraction map}. The collective integrable system $F\circ \mu$ and the toric contraction map form a nice commuting diagram, \eqref{eq;toric contraction commutative diagram}, which is useful for proving convexity properties of our collective integrable systems.

\subsection{Symplectic reduction, implosion, and contraction}\label{section; symplectic implosion} This section fixes notation for symplectic reduction, implosion, and contraction and recalls a few facts which are useful in the constructions that follow. See  \cite{sjamaarStratifiedSymplecticSpaces1991,guilleminSymplecticImplosion2002,hilgertContractionHamiltonianKSpaces2017} for details. In all that follows, let $K$ be a compact connected Lie group, let $T\subset K$ be a maximal torus, and fix a positive Weyl chamber $\ttt_+^*\subset \kk^*$. 

Let $X\sslash_\lambda K$ denote the \emph{symplectic reduction} of a Hamiltonian $K$-space $(X,\mu)$ at a point $\lambda \in \kk^*$. Since $K$ is compact, the action of $K$ on $X$ is proper and $X\sslash_\lambda K$ is again a singular symplectic space. If $X$ is equipped with a Hamiltonian $J$-action which commutes with the $K$-action, then $X\sslash_\lambda K$ has an induced Hamiltonian $J$-action.

Let $\mathcal{E} X$ denote the \emph{symplectic implosion} of a  Hamiltonian $K$-space $(X,\mu)$  \cite[Definition 2.1]{guilleminSymplecticImplosion2002}. Unless otherwise specified, implosion is defined with respect to the Weyl chamber $\ttt_+^*$ fixed above.  The symplectic implosion  is a singular Hamiltonian $T$-space with moment map induced by, and denoted, $\mu$. 

Let $T^*K$ denote the cotangent bundle of $K$ and let $\omega_{\rm{can}}$ denote its canonical symplectic structure. Let $(\mu_L,\mu_R)$ denote the moment map for the action of $K\times K$ on $T^*K$ by cotangent lift of  left and right multiplication. With respect to the left-invariant trivialization $T^*K\cong K\times \kk^*$, $\mu_L(k,\xi ) = \Ad_k^* \xi$ and $\mu_R(k,\xi ) = -\xi$.

The \emph{universal symplectic cross-section}, denoted $\mathcal{E}T^*K$, is the symplectic implosion of  $(T^*K ,\mu_R)$ with respect to  $-\ttt_+^*$.  It is a singular Hamiltonian $K\times T$-space with moment map induced by, and denoted, $(\mu_L,-\mu_R = \mathcal{S}\circ\mu_L)$. Its symplectic pieces are $(K/[K_\sigma,K_\sigma]) \times \sigma$
	where $\sigma$ is an open face of $\ttt_+^*$, where $K_\sigma = K \cap P_\sigma$. Let $\omega^\sigma$ denote the symplectic 2-form on the symplectic piece associated to $\sigma$. The moment map $\mu_L$ induces a $K$-equivariant homeomorphism $\mathcal{E} T^* K/(1\times T) \cong \kk^*$. 

The \emph{symplectic contraction} of a Hamiltonian $K$-space $(X,\mu)$ is the singular symplectic space 
\begin{equation}
	X^{sc} = (\mathcal{E}X\times \mathcal{E}T^*K)\sslash_0 T,
\end{equation}
where symplectic reduction is taken with respect to the Hamiltonian $T$-action generated by  $\mu-\mathcal{S}\circ\mu_L$. It is a Hamiltonian $K\times T$-space with moment map induced by, and denoted, $(\mu_L,\mathcal{S}\circ\mu_L)$.  Let $\Phi^{sc}\colon X \to X^{sc}$ denote the \emph{symplectic contraction map}. In terms of the definition of $X^{sc}$ above,
\[
    \Phi^{sc}(x) = [[h\cdot x],[h\n,\mu(h\cdot x)]]
\]
where: $h\in K$ such that $\mu(h\cdot x) \in \ttt_+^*$, where $[h\cdot x] $ and $[h\n,\mu(h\cdot x)]$ denote equivalence classes in the respective symplectic imploded spaces, and where the outer square brackets denote an equivalence class with respect to the action of $T$. The map $\Phi^{sc}$ is continuous, proper, surjective, $K$-equivariant, fiber-connected, and has the property that $\mu_L \circ \Phi^{sc} = \mu$. See \cite[Section 4]{hilgertContractionHamiltonianKSpaces2017} for more details.

\subsection{Completely integrable systems on $\kk^*$}\label{section; generalized GZ systems} This section constructs integrable systems on $\kk^*$ (equipped with the canonical Lie-Poisson structure).  Since there are various definitions of integrable systems on Poisson manifolds in the literature, we briefly define what we mean by integrable systems on $\kk^*$.

\begin{definition}\label{def; completely integrable system on constant rank Poisson manifold}
    A collection of real valued continuous functions $f_1,\dots, f_n$ on a smooth connected Poisson manifold $M$ of constant rank $2r$ is a \emph{completely integrable system} if: 
    \begin{enumerate}[label=(\roman*)]
        \item There exists an open dense subset $D \subset M$ such that the restricted functions $f_i\vert_{D}$ are all smooth and the rank of the Jacobian of $F = (f_1,\dots,f_n)$ equals $\dim(M)- r$ on a dense subset of $D$.
        \item The restricted functions $f_i\vert_{D}$ pairwise Poisson commute, i.e.~$\{f_i\vert_{D},f_j\vert_{D}\} = 0$ for all $1\leq i,j\leq n$.
    \end{enumerate}
\end{definition}

A compact Lie group $K$ can be expressed in the form $ K= (K_{ss} \times Z)/D$, where $K_{ss}$ is a semisimple compact connected Lie group,  $Z$ is a compact connected torus, and $D \subset K_{ss} \times Z$ is a finite central subgroup such that $(1\times Z)\cap D$ is the trivial group. Let $T_{ss}$ be a maximal torus of $K_{ss}$ so that $T = (T_{ss}\times Z)/D$. Note that $\ttt_+^* = (\ttt_{ss})_+^* \times \mathfrak{z}^*$. 

Let $G$ be the complexification of $K_{ss}$. Fix $H$ and $N$ so that $H\cap K_{ss} = T_{ss}$, and so that the dominant real weights associated with $N$ generate the cone $(\ttt_{ss})_+^*$. Equip $G\sslash N$ with the structure of a Hamiltonian $K_{ss}\times T_{ss}$-space by embedding it into a Hermitian inner product space $(E,h_E)$, cf.~\eqref{eq; definition of E given Pi} and \eqref{equation;Ninvariants}. We have the following chain of isomorphisms of Hamiltonian $K\times T$ spaces:
\[
    \mathcal{E}T^*K \cong  \mathcal{E}(T^*K_{ss} \times T^*Z)/D \cong (\mathcal{E}T^*K_{ss} \times T^*Z)/D \cong (G\sslash N \times T^*Z)/D.
\]
The first two isomorphisms are canonical. The third follows by \cite[Theorem A.1]{hilgertContractionHamiltonianKSpaces2017}. 

Fix a good valuation $(G\sslash N, E, h_E, \mathbf{v}, \mathsf{a},\mathsf{c})$ as in Data \ref{def; data for valuation on g mod n}. Let $X_\mathsf{S}$ denote the affine toric variety associated to the value semigroup of $\mathbf{v}$ and construct a toric degeneration of $G\sslash N$ to $X_\mathsf{S}$ as in Section \ref{main construction section}. The maximal compact subgroup of the algebraic torus of $X_\mathsf{S}$ is $\T_{ss} = (S^1)^m \times T_{ss}$.  Equip $X_\mathsf{S}$ with the structure of a Hamiltonian $\T_{ss}$-space as follows.  Recall that $\Sigma$ denotes the poset of open faces of $(\ttt_{ss})_+^*$. Each closed face $\overline{\sigma} \subset (\ttt_{ss})_+^*$ corresponds to a torus orbit-closure $X_{\mathsf{S}}^{\overline{\sigma}} \subset X_\mathsf{S}$.  Following Whitney~\cite{whitneyElementaryStructureReal1957}, inductively define $(X_\mathsf{S}^\sigma)_0 = X_\mathsf{S}^\sigma$ and $(X_\mathsf{S}^\sigma)_{i+1} = (X_\mathsf{S}^\sigma)_{i} \setminus (X_\mathsf{S}^\sigma)_{i}^{sm}$.  The sets $(X_\mathsf{S}^\sigma)_{i}^{sm}$ are disjoint $\T_{ss}$-invariant smooth manifolds. Decompose further into connected components if necessary.  The vector space $E$ is equipped with a representation of $\T_{ss}$ and $X_\mathsf{S}$ is identified $\T_{ss}$-equivariantly with a subvariety of $E$. The symplectic structure on $E$ endows each piece $(X_\mathsf{S}^\sigma)_{i}^{sm}$ with a symplectic structure.  Thus $X_{\mathsf{S}}$ is a singular symplectic space. The action of $\T_{ss}$ on $E$ is Hamiltonian, generated by a moment map $\Psi_{ss} \colon E \to \Lie(\T_{ss})^*$ as in \eqref{eqn; TT moment map def}. Since each $(X_\mathsf{S}^\sigma)_{i}^{sm}$ is $\T_{ss}$-invariant, the restriction of $\Psi_{ss}\vert_{(X_\mathsf{S}^\sigma)_{i}^{sm}}$ is a moment map for the action of $\T_{ss}$. 

Let $X$ denote the quotient space $(X_{\mathsf{S}}\times T^*Z)/D$,
where $D$ acts on the product as a subgroup of $T_{ss} \times Z$. Let $\T = (S^1)^m \times T$ and let $\Psi\colon X \to \Lie(\T)^*$ denote the map induced by $\Psi_{ss} \times \mu_{Z,L} $. Note that $(X,\Psi)$ is a Hamiltonian $\T$-space. Let $\phi_{ss} \colon G\sslash N \to X_\mathsf{S}$ denote the map constructed as in Theorem \ref{maintheoremGmodN}  by toric degeneration and let $\phi\colon \mathcal{E}T^*K \to X$ be the map induced by $\phi_{ss} \times \id_{T^*Z}$ and the chain of isomorphisms above.  Let $F\colon \kk^* \to \Lie(\T)^*$ denote the continuous map induced by $\Psi \circ \phi$. These constructions are summarized in the following commuting diagram (where dotted arrows indicate the map is induced by a quotient).
\begin{equation}\label{eq; Psi defined for arbitrary K}
    \begin{tikzcd}
        \mathcal{E}T^*K_{ss} \times T^*Z \ar[r, "\phi_{ss} \times \id"] \ar[d,"/D"] & X_{\mathsf{S}}\times T^*Z \ar[r, "\Psi_{ss} \times \mu_{Z,L}"] \ar[d,"/D"] & \Lie(\T_{ss})^* \times \Lie(Z)^* \ar[d,equals] \\
        \mathcal{E}T^*K \ar[r,dotted,"\phi"] \ar[dr,"\mu_L"] &X \ar[r, dotted,"\Psi"]  & \Lie(\T)^*\\
        &\kk^* \ar[ru,"F",dotted]&
    \end{tikzcd}
\end{equation}

\begin{theorem}\label{thm; liek completely integrable system}
    Let $K$ be a compact connected Lie group and let $F \colon \mathfrak{k}^* \to \Lie(\mathbb{T})^*$ be constructed as above. Then, 1) $\Psi \circ \phi$ is a completely integrable system on $\mathcal{E}T^*K$ in the sense of Definition~\ref{def; integrable system on singular symplectic space}, and 2) $F$ is a completely integrable system on $\mathfrak{k}^*$ in the sense of Definition \ref{def; completely integrable system on constant rank Poisson manifold} with respect to the standard Lie-Poisson structure.
\end{theorem}

\begin{proof}
    The first conclusion follows immediately from the definition since the quotient $/D$ is a smooth local diffeomorphism onto each symplectic piece of $\mathcal{E}T^*K$. The second conclusion follows since the restriction of $\mu_L$ to each symplectic piece of $\mathcal{E}T^*K$ is a constant-rank smooth submersion onto each orbit-type stratum of $\mathfrak{k}^*$ that is Poisson with respect to the Lie-Poisson structure.
\end{proof}

\begin{example}[String valuations] \label{main example 1.2} Combining Example \ref{main example} and Theorem \ref{thm; liek completely integrable system}, we have realized every extended string cone associated to $G = K^\C$ as the image of a completely integrable system on $\mathfrak{k}^*$.
\end{example}

The following is immediate from the definitions, Theorem \ref{maintheoremGmodN}, and Proposition \ref{ReyerCor}. 
    
\begin{proposition}\label{prop; facts about maps}
    Let $X$, $\phi$, $F$, and $\Psi$ be defined as above. Then:
    \begin{enumerate}[label=(\roman*)]
        \item The map $\phi$ is continuous, proper, surjective and $T\times T$-equivariant.
        \item The map $F$ has the property that $\mathsf{a} \circ F = \pr_{\ttt^*}$ and $\mathsf{c} \circ F = \mathcal{S}$.
        \item The map $\Psi$ is proper and its fibers are connected.
        \item $F(\kk^*) = \Psi(X) = \cone(\mathsf{S})\times \Lie(Z)^*$.
        
    \end{enumerate}
\end{proposition}

We end this section with some useful topological details about the maps $\Psi$ and $\phi$. Recall from Section \ref{subsection; hamiltonian torus action} that $\Psi_{ss}(X_{\mathsf{S}}^{\overline{\sigma}}) = \cone(\mathsf{S}^{\overline{\sigma}})$. The image  $\Psi_{ss}(X_{\mathsf{S}}^{\sigma})$ is the complement in $\cone(\mathsf{S}^{\overline{\sigma}})$ of the union of cones $\cone(\mathsf{S}^{\overline{\tau}})$ such that $\tau \prec \sigma$. As such, $\Psi_{ss}(X_{\mathsf{S}}^{\sigma})$
is a convex, locally rational polyhedral set. We record several useful facts (cf.~Corollary \ref{cor; smooth locus}). 

\begin{proposition}\label{prop; facts about maps II} Let $X$, $\Psi$, $\phi$, and $F$ be defined as above. Denote $U^\sigma = ((X_\mathsf{S}^\sigma)^{sm} \times T^*Z)/D$. Then:
\begin{enumerate}[label=(\roman*)]
    \item The image $\Psi(U^\sigma)$ is the smooth locus of  $\Psi(X_{\mathsf{S}}^{\sigma}) \times \Lie(Z)^*$. In particular, it is convex.
    \item The domain and codomain restricted map $\Psi\colon U^\sigma \to \Psi(U^\sigma)$ is proper.
    \item The pre-image $\phi\n(U^\sigma)$ is a connected open dense $T\times T$-invariant subset of $K/[K_\sigma,K_\sigma] \times \sigma$ and the restricted map  
        \[
            \phi\colon (\phi\n(U^\sigma),(\pr_{\ttt^*}\circ \mu_L,\mathcal{S}\circ \mu_L)) \to (U^\sigma,(\mathsf{s}\circ \Psi, \mathsf{c}\circ \Psi))
        \]
        is an isomorphism of Hamiltonian $T\times T$-manifolds.
\end{enumerate} 
\end{proposition}

\subsection{Toric contraction and collective integrable systems}\label{section; toric contractions} This section applies the results of the previous section to construct collective integrable systems on arbitrary Hamiltonian spaces whose underlying group is compact. Let $\T$, $X$, $\Psi$, $\phi$, and $F$ be defined as in the previous section.

\begin{definition}\label{def; toric contraction space}
The \emph{toric contraction} of a Hamiltonian $K$-space $(Y,\mu)$ with respect to $(X,\Psi)$ is the space $X_Y := (\mathcal{E}Y \times X)\sslash_0 T$,
where the symplectic  reduction is taken with respect to the action of $T$ generated by the moment map $\mu-\mathsf{c} \circ \Psi$. 
\end{definition}

Let $\Psi_Y \colon X_Y \to \Lie(\T)^*$ denote the map induced by $\Psi$. Together with the induced $\T$-action, this endows $X_Y$ with the structure of a Hamiltonian $\T$-space.\footnote{Despite this terminology, the action of $\T$ on $X_Y$ is not necessarily a complexity 0 action. See  Proposition \ref{proposition; properties of toric contraction IV} below.}

By Propositions \ref{prop; facts about maps} and \ref{prop; facts about maps}, the map $\id_{\mathcal{E}Y} \times \phi$ descends to a  continuous map $\phi_Y \colon Y^{sc} \to X_Y$, as illustrated by the following commuting diagram.
\begin{equation}\label{eq; def of phiM}
\begin{tikzcd}
    \mathcal{E}Y \times \mathcal{E}T^*K \ar[r,"\id_{\mathcal{E}Y} \times \phi"] & \mathcal{E}Y \times X \\
    (\mu -\mathcal{S}\circ \mu_L)\n(0) \ar[u,hookrightarrow] \ar[r] \ar[d,"/T"]& (\mu - \mathsf{c} \circ \Psi)\n(0)\ar[d,"/T"]\ar[u,hookrightarrow]\\
    Y^{sc} \ar[r,dotted,"\phi_Y"] & X_Y.
\end{tikzcd}
\end{equation}
Both  $Y^{sc}$ and $X_Y$ inherit Hamiltonian $T\times T$-space structures generated by the moment maps $ (\pr_{\ttt^*}\circ \mu_L,\mathcal{S}\circ \mu_L)$ and $(\mathsf{a}\circ \Psi,\mathsf{c}\circ \Psi)$ respectively.  The map $\phi_Y$ is $T\times T$ equivariant by construction. 

\begin{definition}\label{def; toric contraction map}
The \emph{toric contraction map}  is the composition $\Phi = \phi_Y \circ \Phi^{sc}\colon Y \to X_Y$.
\end{definition}

We now give a series of results about toric contractions. In these results we are careful to note the various topological properties of toric contraction which will be important in applications. The first result is a direct consequence of the definitions and properties of symplectic contraction.

\begin{proposition}[Properties of toric contraction, part I]\label{proposition; properties of toric contraction I} Let $(Y,\mu)$ be a Hamiltonian $K$-space.
	\begin{enumerate}[label=(\roman*)]
		\item The toric contraction map is continuous, surjective, proper, and $T$-equivariant (with respect to the action of $T$ on $Y$ as the maximal torus $T\subset K$ and the action of $T$ on $X_Y$ as the subtorus $T\times 1 \subset T\times T$).
        \item The following diagram commutes.
        \begin{equation}\label{eq;toric contraction commutative diagram}
\begin{tikzcd}
    Y \ar[r,"\Phi"] \ar[d,"\mu"]& X_Y \ar[d,"\Psi_Y"]\\ 
    \kk^* \ar[r,"F"] & \Lie(\T)^*.
\end{tikzcd}
\end{equation}
	\end{enumerate}
\end{proposition}

An analogue of  \eqref{eq;toric contraction commutative diagram} was constructed for integral coadjoint orbits of unitary groups in \cite{nishinouToricDegenerationsGelfand2010}. Their construction used Gelfand-Zeitlin systems and toric degeneration. A similar diagram  was later constructed for all Hamiltonian $U(n)$-spaces using branching contraction \cite{hilgertContractionHamiltonianKSpaces2017}.

\begin{proposition}[Properties of toric contraction, part II]\label{proposition; properties of toric contraction II} Let $(Y,\mu)$ be a Hamiltonian $K$-manifold with principal stratum $\sigma$ (i.e. $\sigma$ is the unique stratum of $\mathfrak{t}^*_+$ such that $\mu(Y) \cap \sigma$ is dense in $\mu(Y)\cap \mathfrak{t}_+^*$). Recall that $\mu^{-1}(\sigma) \subset Y$ is the principal symplectic cross-section.
	\begin{enumerate}[label=(\roman*)]
        \item Denote $U^\sigma = ((X_\mathsf{S}^\sigma)^{sm} \times T^*Z)/D$ as in Proposition \ref{prop; facts about maps II}. Then 
        \[
            \mathcal{U} = (\mu\n(\sigma)\times U^{\sigma}) \sslash_0 T
        \]
        is a dense symplectic piece of $X_Y$.
        \item The pre-image $D = \Phi\n(\mathcal{U})$ is a connected, dense, open, $T$-invariant subset of $Y$ and the restricted toric contraction map is an isomorphism of Hamiltonian $T\times T$-manifolds,
        \[
            \Phi\colon (D,(\pr_{\ttt^*}\circ \mu,\mathcal{S} \circ \mu)) \to (\mathcal{U},(\mathsf{a}\circ\Psi_Y,\mathsf{c}\circ\Psi_Y)).
        \]
	\end{enumerate}
\end{proposition}

\begin{proof} The first order of business is to show that $\mathcal{U}$ is a piece of $X_Y$. This follows since $\mu\n(\sigma)\times U^{\sigma}$ is a piece of $\mathcal{E}Y \times X$ and the stabilizer subgroup at every point in $\mu\n(\sigma)\times U^{\sigma}$ for the $T$-action equals the kernel of the $T$-action, so the reduced space is a symplectic manifold. 

To see (ii), note that
    \[
        \phi_Y\n(\mathcal{U}) = (\mu\n(\sigma) \times \phi\n(U^\sigma))\sslash_0 T. 
    \]
    By Proposition \ref{prop; facts about maps II}, this is a connected open dense $T \times T$-invariant subset of $Y^{sc}$ and the restricted map 
    \[
        \phi_Y \colon (\phi_Y\n(\mathcal{U}),(\pr_{\ttt^*}\circ \mu_L,\mathcal{S}\circ \mu_L) ) \to (\mathcal{U},(\mathsf{a}\circ\Psi_Y,\mathsf{c}\circ\Psi_Y))
    \]
    is an isomorphism of Hamiltonian $T\times T$-manifolds. 
    Item (ii) then follows by properties of symplectic contraction. Finally, since $\Phi\n(\mathcal{U})$ is dense in $Y$ by (ii) and $\Phi\colon Y \to X_Y$ is surjective, $\mathcal{U}$ is dense in $X_Y$. This completes the proof of (i).
\end{proof}

Recall that a Hamiltonian $K$-manifold $(Y,\mu)$ is proper if there exists a $K$-invariant set $\tau \subset \kk^*$ containing $\mu(M)$ such that $\tau \cap \ttt_+^*$ is convex and  $\mu \colon M \to \tau $ is a proper map. If $(Y,\mu)$ is proper, then $\triangle_Y := \mathcal{S} \circ \mu(Y)$
is a convex locally rational polyhedral set and the fibers of $\mu$ are connected \cite[Theorem 1.1 and Remark 5.2]{lermanNonabelianConvexitySymplectic1998}.\footnote{Although we cite \cite{lermanNonabelianConvexitySymplectic1998} for a statement of the convexity theorem, it should be noted that convexity for proper moment maps is due to numerous authors including Condevaux, Dazord, Molino, Knop, Birtea, Ratiu, Ortega, Sjamaar, Karshon, and Bjorndahl.}
		
\begin{proposition}[Properties of toric contraction, part III]\label{proposition; properties of toric contraction III} Let $(Y,\mu)$ be a proper Hamiltonian $K$-manifold. Let $\sigma$ be the principal stratum of $(Y,\mu)$ and let $\mathcal{U}$ be as in Proposition \ref{proposition; properties of toric contraction II}. Then:
	\begin{enumerate}[label=(\roman*)]
		\item The set $\Psi_Y(X_Y) = \mathsf{c}\n(\triangle_Y)\cap \Psi(X)$  is a convex, locally rational polyhedral set.  Moreover, if $Y$ is  compact, then $\Psi_Y(X_Y)$ is a convex polytope.
		\item The codomain restricted map $\Psi_Y\colon X_Y \to \Psi_Y(X_Y)$ is proper.
		\item The pair $(\mathcal{U},\Psi_Y)$ is a proper Hamiltonian $\T$-manifold and 
		\[
		    \Psi_Y(\mathcal{U}) = \mathsf{c}\n(\triangle_Y\cap \sigma) \cap \Psi(U^\sigma)
		\]
		is the smooth locus of the convex locally rational polyhedral set  $\mathsf{c}\n(\triangle_Y\cap \sigma) \cap \Psi(X^\sigma)$.
		\item The fibers of $\Psi_Y$ and $F \circ \mu = \Psi_Y \circ \Phi $ are connected.
	\end{enumerate}
\end{proposition}

\begin{proof}[Proof of Proposition \ref{proposition; properties of toric contraction III}]
\begin{enumerate}[label=(\roman*)]
    \item The equality $\Psi_Y(X_Y) = \mathsf{c}\n(\triangle_Y)\cap \Psi(X)$ follows immediately from the definitions.  Since $\triangle_Y$ is a convex locally rational polyhedral set, $\mathsf{c}$ is a linear map, and $\Psi(X) = \cone(\mathsf{S})$ is a convex rational polyhedral cone by Proposition \ref{prop; facts about maps}, it follows that $\Psi_Y(X_Y)$ is a convex locally rational polyhedral set. If $Y$ is compact, then $\triangle_Y$ is a convex polytope by the non-abelian convexity theorem for Hamiltonian group actions. It follows in this case that $\Psi_Y(X_Y)$ is a convex polytope since the fibers of $\mathsf{c}\n(\lambda) \cap \cone(\mathsf{S})$ are compact.
    \item This follows by (i), since $\Psi\colon X \to \Psi(X)$ is proper (Proposition \ref{prop; facts about maps}) and $(Y,\mu)$ being proper implies that $\mathcal{S}\circ \mu \colon Y \to \triangle_Y$ is proper.
    \item It follows from the definitions that
    $\Psi_Y(\mathcal{U}) =\mathsf{c}\n(\triangle_Y\cap \sigma) \cap \Psi(U^\sigma)$. In particular, $\Psi_Y(\mathcal{U})$ is convex since it is an intersection of convex sets. It is the smooth locus of $\mathsf{c}\n(\triangle_Y\cap \sigma) \cap \Psi(X^\sigma)$ since $\triangle_Y\cap \sigma$ is smooth and $\Psi(U^\sigma)$ is the smooth locus of $\Psi(X^\sigma)$. The properness of $\Psi_Y\colon \mathcal{U} \to \Psi_Y(\mathcal{U})$ follows since $(Y,\mu)$ is proper and by Proposition \ref{prop; facts about maps II}.
    \item First, we show that the fibers of $\Psi_Y$ are connected~\cite[Theorem 1.1]{lermanNonabelianConvexitySymplectic1998}. Since $(Y,\mu)$ is proper, the fibers of  $\mu\colon \mathcal{E}Y \to \ttt^*_+$ are connected. The fibers of $\Psi$ are connected by Proposition \ref{prop; facts about maps}. Thus, for all $\xi \in \Lie(\T)^*$, $
		    \Psi_Y\n(\xi) = (\mu\n(\mathsf{c}(\xi))\times \Psi\n(\xi))/T$
    is connected.
	
	The fiber connectedness of $\Psi_Y \circ \Phi$ follows by applying \cite[Lemma 4]{laneConvexityThimmTrick2018}. By (i) and (ii),  $\Psi_Y \circ \Phi \colon Y \to \Psi_Y(X_Y)$ is a continuous proper map to the convex set $\Psi_Y(X_Y)$. By Proposition \ref{proposition; properties of toric contraction II}, $\Phi\n(\mathcal{U})$ is a dense subset of $Y$. The set $\Phi\n(\mathcal{U})$ is saturated by $\Psi_Y \circ \Phi$ since $\mathcal{U}$ is saturated by $\Psi_Y$. By (iii), $\Psi_Y(\mathcal{U})$ is convex.  The fibers of the restriction of $\Psi_Y \circ \Phi$ to $\Phi\n(\mathcal{U})$ are connected since the fibers of $\Psi_Y$ are connected, $\mathcal{U}$ is saturated by $\Psi_Y$, and $\Phi\colon \Phi\n(\mathcal{U}) \to\mathcal{U}$ is a symplectomorphism.  Finally, the restricted map $\Psi_Y \circ \Phi\colon \Phi\n(\mathcal{U}) \to \Psi_Y(\mathcal{U})$ is open since $(\mathcal{U},\Psi_Y)$ is a proper Hamiltonian $\T$-manifold.
	\end{enumerate}
\end{proof}

Recall that the \emph{complexity} of a Hamiltonian $T$-manifold $(Y,\mu)$ is
\begin{equation}\label{def; complexity part 1}
    c_T(Y) = \frac{1}{2} \dim Y - \dim T + \dim T_{\ker}
\end{equation}
where $T_{\ker}$ is the kernel of the action of $T$; this is equal to half the dimension of the symplectic reduction by $T$ at a regular value of $\mu$. The \emph{complexity}
of a Hamiltonian $K$-manifold $(Y,\mu)$ with principal symplectic cross-section  $S = \mu^{-1}(\sigma)$ is the complexity of $S$ as a Hamiltonian $T$-manifold, $c_K(Y) = c_T(S)$; this is equal to half the dimension of the symplectic reduction by $K$ at a regular value of $\mu$.

\begin{proposition}[Properties of toric contraction, part IV]\label{proposition; properties of toric contraction IV} Let $(Y,\mu)$ be a Hamiltonian $K$-manifold and let $(\mathcal{U},\Psi_Y)$ be the dense symplectic piece of  $X_Y$ as in Proposition \ref{proposition; properties of toric contraction II}.  Then, $c_\T(\mathcal{U}) = c_K(Y)$. 
\end{proposition}

\begin{proof} Let $S = \mu^{-1}(\sigma)$ be the principal symplectic cross-section of $Y$ and let $T_S$ denote the kernel of the $T$-action on $S$. 
Recall from Proposition \ref{proposition; properties of toric contraction II} that $\mathcal{U}$ is constructed as the diagonal symplectic reduction of $S \times U^\sigma$ by $T$. Let $T^{\overline{\sigma}} \subset T$ denote the connected subtorus with $\ann(\Lie(T^{\overline{\sigma}})) = \operatorname{span}_\R(\sigma)$.  The kernel of the diagonal $T$-action on $S \times U^\sigma $ is  $T^{\overline{\sigma}}$. In particular, note that $T^{\overline{\sigma}}\subset T_S$. In fact, the diagonal action of  $T/T^{\overline{\sigma}}$ on  $S \times U^\sigma$ is free. Thus,
\[
    \dim \mathcal{U} = \dim S + \dim U^\sigma - 2\dim T + 2\dim T^{\overline{\sigma}}.
\]
The action of $\T$ on $\mathcal{U}$ descends from the action of $\T$ on $U^\sigma$. The kernel of the action of $\T$ on $U^\sigma$ is the subtorus $\T^F$, where $F= \mathsf{c}\n(\sigma)$. By construction, $U^\sigma$ is a complexity 0 Hamiltonian $\T$-manifold, i.e.
\[
    0 = \frac{1}{2}\dim U^\sigma -\dim \T + \dim \T^F.
\]
By construction, the kernel $\T_{\ker}$ of the $\T$-action on $\mathcal{U}$ is the product of the subgroups $\T^F$ and $1 \times T_S$.  The intersection of these subgroups is $1 \times T^{\overline{\sigma}}$, so
\[
    \dim \T_{\ker} = \dim \T^F + \dim T_S - \dim T^{\overline{\sigma}}.
\]
Combining these facts, we have that 
\begin{equation*}
    \begin{split}
        c_\T(\mathcal{U}) = \frac{1}{2}\dim \mathcal{U} - \dim \T + \dim \T_{\ker}  = \frac{1}{2}\dim S  - \dim T   + \dim T_S  = c_K(Y)
    \end{split}
\end{equation*}
which completes the proof.
\end{proof}

\begin{remark}\label{rem;action-angle coordinates} The above propositions have the following important consequence: if $(Y,\mu)$ is a proper multiplicity-free Hamiltonian $K$-manifold, then the dense open subset $(D,F\circ \mu)$ is a proper toric manifold. Thus $(D,F\circ \mu)$ is classified up to isomorphism by its image  $F\circ \mu(D) = \mathsf{c}\n (\Delta_Y \cap \sigma) \cap \Psi(U^\sigma)$ and the kernel of the action of $\T$  \cite[Proposition 6.5]{karshonNonCompactSymplecticToric2015}.
\end{remark}

\subsection{Proof of main theorem and applications} \label{section proof of main theorem}

\begin{proof}[Proof of Theorem~\ref{intro main theorem}]
    Fix a good valuation $(G\sslash N, E, h_E, \mathbf{v}, \mathsf{a},\mathsf{c})$ as in Data \ref{def; data for valuation on g mod n}; this exists by Example~\ref{main example}. Let $F \colon \Lie(K)^*\to \Lie(\mathbb{T})^*$ be as in Section~\ref{section; generalized GZ systems}. The first claim and third claims follow directly from Proposition~\ref{proposition; properties of toric contraction III}. The second claim follows directly from Proposition~\ref{proposition; properties of toric contraction IV}.
\end{proof}

\begin{example}[Coadjoint orbits of compact Lie groups]\label{main example 2}
Let $K$ be a compact connected semisimple Lie group, let $Y = \mathcal{O}_\lambda$ be the coadjoint orbit of $K$ passing through $\lambda \in \ttt_+^*$, and let $\omega_\lambda$ denote its canonical symplectic structure. The coadjoint action of $K$ on $\mathcal{O}_\lambda$ is Hamiltonian with moment map the inclusion $\iota\colon \mathcal{O}_\lambda \hookrightarrow \kk^*$.  The symplectic contraction of this Hamiltonian space is $\mathcal{O}_\lambda^{sc} = \mathcal{O}_\lambda$. The extra $T$-action on $\mathcal{O}_\lambda^{sc}$ is trivial; its moment map is the constant map that sends every point to $\lambda$.  The symplectic contraction map is the identity. 

Continuing with this example, suppose we are given Data \ref{def; data for valuation on g mod n} and have constructed maps and spaces as in Sections \ref{section; generalized GZ systems} and \ref{section; toric contractions}.  Our construction simplifies in this case. The toric contraction is $X_Y \cong X_\mathsf{S}\sslash_\lambda T$ and the toric contraction map is induced directly from  $\phi_{ss}$:
\[
   \Phi\colon  Y \cong (G\sslash N)\sslash_\lambda T \to X_\mathsf{S}\sslash_\lambda T \cong X_Y.
\]
The image $\Psi(X_Y)$ is the convex polytope given by the intersection of $\cone(\mathsf{S})$ and the affine subspace $\mathsf{c}\n(\{\lambda\})$.  For instance, if the valuation used in the construction is a string valuation $\mathbf{v}_\ii$ as in Example \ref{main example}, then $\Psi(X_Y)$ is the intersection of the extended string cone $\cone(\mathsf{S}_\ii)$ and the affine subspace $\mathsf{c}\n(\{\lambda\})$. If $\lambda$ is integral, then $\Delta_\lambda(\ii) = \cone(\mathsf{S}_\ii)\cap \mathsf{c}\n(\{\lambda\})$ is the string polytope associated to $\lambda$ and $\ii$.  We note that the integer lattice points of $\Delta_\lambda(\ii)$ correspond to a crystal basis for the irreducible $K$-module $V_\lambda$ with highest weight $\lambda$~\cite{berensteinTensorProductMultiplicities2001}.
\end{example}

\begin{example}[Cotangent bundles of compact Lie groups]\label{main example 3} Let $K$ be a compact connected Lie group. The toric degeneration construction of \cite{haradaIntegrableSystemsToric2015} cannot be applied to the cotangent bundle $T^*K$ because it is non-compact. However, since the Hamiltonian $K\times K$-action on $T^*K$ described in Section \ref{section; symplectic implosion} is proper and multiplicity-free, applying our construction to the $K\times K$-action yields collective completely integrable systems on $T^*K$ with connected moment map fibers. 

Assume for simplicity that $K$ is also semisimple and let $\ii_1$ and $\ii_2$ be any two reduced words for the longest element $w_0$ of the Weyl group of $G = K^\mathbb{C}$. Together $\ii_1$ and $\ii_2$ form a reduced word for the longest element of the Weyl group of the product $G\times G$ which we denote $(\ii_1,\ii_2)$ where the index denotes the direct summand. We can define data as in Example \ref{main example} where, e.g.,  $\mathbf{v}_\ii = \mathbf{v}_{(\ii_1,\ii_2)}$ and $\mathsf{c} = (\mathsf{c}_1, \mathsf{c}_2)$. Applying the construction of Theorem \ref{thm; liek completely integrable system} to these data yields a completely integrable system 
\begin{equation}\label{eq; product of integrable systems}
    F_{(\ii_1, \ii_2)}\colon \kk^* \times \kk^* \to \Lie(\T)^* \times \Lie(\T)^*
\end{equation}
 in the sense of Definition \ref{def; completely integrable system on constant rank Poisson manifold}. It follows from the  definition of $ (\ii_1,\ii_2)$ that the image of $F_{(\ii_1, \ii_2)}$ is the product of extended string cones, 
\begin{equation}\label{eq; extended string cone}
    F_{(\ii_1, \ii_2)}(\kk^* \times \kk^*) = \cone(\mathsf{S}_{(\ii_1, \ii_2)}) = \cone(\mathsf{S}_{\ii_1}) \times \cone(\mathsf{S}_{\ii_2}).
\end{equation}
The image of the pullback
\begin{equation}\label{eq; collective system on T^*K}
    F_{(\ii_1, \ii_2)} \circ (\mu_L,\mu_R) \colon T^*K \xrightarrow{(\mu_L, \mu_R)}  \kk^* \times \kk^* \xrightarrow{F_{(\ii_1, \ii_2)}} \Lie(\T)^* \times \Lie(\T)^*.
\end{equation}
is easy to describe. First, the image of $(\mu_L, \mu_R)$ is 
\begin{equation}\label{eq; image of T^*K moment map}
    (\mu_L, \mu_R)(T^*K) = \{ (\lambda_1, \lambda_2) \in \kk^* \times \kk^* \mid \exists \lambda \in \mathfrak{t}_+^* \colon \lambda_1 \in \mathcal{O}_\lambda \text{ and }\lambda_2 \in \mathcal{O}_{\lambda^*}\}
\end{equation}
where $\lambda^* = -w_0\lambda$). It follows that 
\begin{equation}\label{eq; image of collective system on T^*K}
    F_{(\ii_1,\ii_2)}\circ(\mu_L, \mu_R)(T^*K) = \{ (\xi_1, \xi_2) \in \cone(\mathsf{S}_{\ii_1}) \times \cone(\mathsf{S}_{\ii_2}) \mid \mathsf{c}_1(\xi_1) = \mathsf{c}_2(\xi_1)^*\},
\end{equation}
cf.~Proposition \ref{prop; facts about maps}(ii). 

By Proposition \ref{proposition; properties of toric contraction I} and \ref{proposition; properties of toric contraction II}, $F_{(\ii_1,\ii_2)}\circ(\mu_L, \mu_R)$ is a continuous map which is a moment map for a Hamiltonian action of a big torus $\mathbb{T}$ on a dense open subset $D\subseteq T^*K$. It follows by Proposition \ref{proposition; properties of toric contraction III} that $D$ is the pre-image under $F_{(\ii_1,\ii_2)}\circ(\mu_L, \mu_R)$ of the smooth locus of the polyhedral cone \eqref{eq; image of collective system on T^*K} and the restriction of $F_{(\ii_1,\ii_2)}\circ(\mu_L, \mu_R)$ to $D$ is a proper moment map. Since the $K\times K$-action on $T^*K$ is multiplicity-free, it follows by Proposition \ref{proposition; properties of toric contraction IV} that the big torus action on $D$ generated by $F_{(\ii_1,\ii_2)}\circ(\mu_L, \mu_R)$ is completely integrable. Thus, $D$ and the moment map $F_{(\ii_1,\ii_2)}\circ(\mu_L, \mu_R)$ are classified up to isomorphism as a proper symplectic toric manifold by the smooth locus of \eqref{eq; image of collective system on T^*K}, cf.~\cite{karshonNonCompactSymplecticToric2015}. In particular, fibers over the relative interior of \eqref{eq; image of collective system on T^*K} are Lagrangian tori in $T^*K$.

We end this example by making an observation in relation to geometric quantization and the recent independence of polarization result of Crooks and Weitsman \cite{CrooksWeitsman3}. The integer lattice points of \eqref{eq; image of collective system on T^*K} are precisely the integer lattice points of products of string polytopes corresponding to dual pairs of dominant integral weights of $G$,
\begin{equation}\label{eq; integer lattice}
    F_{(\ii_1,\ii_2)}\circ(\mu_L, \mu_R)(T^*K)_\Z = \{ (\lambda, \mathbf{m}_1, \lambda^*, \mathbf{m}_2) \in \Lambda_+ \times \Z^m \times \Lambda_+ \times \Z^m \mid \mathbf{m}_i \in \Delta_{\lambda}(\ii_i)\},
\end{equation}
cf.~Example \ref{main example 2}. The elements of this set form a basis of
\begin{equation}\label{eq; peter weyl}
    \bigoplus_{\lambda \in \Lambda_+} V_\lambda \otimes V_{\lambda}^*
\end{equation}
and thus, by the Peter-Weyl theorem, of $L^2(K)$.  The latter space, $L^2(K)$, is the  quantization of $T^*K$ resulting from the vertical polarization of the cotangent bundle\footnote{We take the liberty of not introducing the machinery of geometric quantization in detail and refer the reader to \cite{CrooksWeitsman3} for details and additional references.}. On the other hand, the Lagrangian torus fibration constructed above is an example of a singular real polarization of $T^*K$. The integer lattice points in the relative interior of \eqref{eq; image of collective system on T^*K} correspond to Lagrangian tori that are Bohr-Sommerfeld fibers of this real polarization.  If we ignore basis elements of \eqref{eq; peter weyl} corresponding to lattice points in the boundary of \eqref{eq; image of collective system on T^*K}, then this can be interpreted as an independence of polarization result: the real polarization produced by our construction recovers a basis for $L^2(G)$ and, moreover, this basis is graded by the weights of $G$. This generalizes the recent results of Crooks and Weitsman who used a product of Gelfand-Zeitlin systems to prove a similar ``independence of polarization'' result for $T^*U(n)$ \cite{CrooksWeitsman3}.
\end{example}

\subsection{Gromov width of coadjoint orbits}\label{ss; gromov width}  Recall that the Gromov width of a compact symplectic manifold $M$ of dimension $2n$ is the supremum of all cross-sectional areas $\pi r^2$, $r>0$, such that the standard symplectic ball of dimension $2n$ and radius $r$ can be embedded symplectically into $M$. It is an open conjecture of Karshon and Tolman that the Gromov width of a coadjoint orbit of a connected compact simple Lie group $K$ is given by the relatively simple formula
\begin{equation}\label{eq;ktconjecture}
    \mathrm{GWidth}(\mathcal{O}_\lambda,\omega_\lambda) = \min\{ \langle \lambda, \alpha^\vee\rangle \mid \alpha \in R_+, \, \langle \lambda, \alpha^\vee\rangle > 0\}.
\end{equation}
Tight upper bounds are known in all cases \cite{castroUpperBoundGromov2016}, but a proof of tight lower bounds for all cases has remained elusive.  The goal of this section is to close Karshon and Tolman's conjecture, modulo an algebraic conjecture. 

Let $K$ be a connected compact simple Lie group with complexification $G$. Recall from Example~\ref{ex; flp valuations} that valuations $\mathbf{v}_{\underline{\beta}}$ can be defined on $\C[G\sslash N]$ from an enumeration $\underline{\beta} = \{ \beta_1, \dots, \beta_m\}$ of the positive roots of $G$. Such an enumeration is said to be a \emph{good ordering} if $i>j$ whenever $\beta_i $ is larger than $\beta_j$ (with respect to the standard partial order on positive roots) \cite{fangEssentialBases2017}.

\begin{theorem}\label{thm;gwidth}
    Let $K$ be a connected compact simple Lie group, and $G$ the complexification of $K$. If there exists a good ordering $\underline{\beta} = \{\beta_1, \dots, \beta_m\}$ of the roots of $G$ such that the value semigroup $\mathsf{S}_{\mathbf{v}_{\underline{\beta}}}$ is finitely generated and saturated, then~\eqref{eq;ktconjecture} holds for all coadjoint orbits of $K$.
\end{theorem}

\begin{proof} Let $\underline{\beta}$ be such a good ordering. Then $\mathbf{v}_{\underline{\beta}}$ gives rise to an integrable system on $G\sslash N$ whose image is the rational cone $\cone(\mathsf{S}_{\mathbf{v}_{\underline{\beta}}})$ (cf.~Example~\ref{ex; flp valuations}).  The construction of the previous section produces an integrable system on $(\mathcal{O}_\lambda,\omega_\lambda)$, for every $\lambda \in \ttt_+^*$, whose image is the convex polytope given by the intersection of $\cone(\mathsf{S}_{\mathbf{v}_{\underline{\beta}}})$ and $\mathsf{c}^{-1}(\lambda)$ (cf.~Example~\ref{main example 2}).  In the case that $\lambda$ is integral, this intersection is identified (integral affinely) with the Newton-Okounkov body $\Delta_{\lambda}(\underline{\beta})$ defined in \cite[Section 3]{fangSimplicesNewtonOkounkovBodies2018}.
    
Given $\lambda$, let $\ell_\lambda$ denote the value of the right hand side of~\eqref{eq;ktconjecture}.
It was shown in \cite[Theorem 6.2]{fangSimplicesNewtonOkounkovBodies2018} that for $\lambda$ integral it is possible to integral affinely embed an equidimensional open simplex of size $\ell_\lambda$ into $\Delta_{\lambda}(\underline{\beta})$.  It follows by a scaling and continuity argument, similar to that of \cite[Section 6]{alekseevActionangleCoordinatesCoadjoint2020}, that for all $\lambda\in \ttt_+^*$ it is possible to integral affinely embed an equidimensional open simplex of size $\ell_\lambda$  into $\cone(\mathsf{S}_{\mathbf{v}_{\underline{\beta}}})\cap \mathsf{c}^{-1}(\lambda)$.  Combining this with Remark~\ref{rem;action-angle coordinates} and \cite[Proposition 2.1]{fangSimplicesNewtonOkounkovBodies2018}, we have  that the Gromov width of $(\mathcal{O}_\lambda,\omega_\lambda)$ is $\geq \ell_\lambda$. Since tight upper bounds are already known by \cite{castroUpperBoundGromov2016}, this completes the proof.
\end{proof}

It is conjectured that good orderings whose associated value semigroups are finitely generated exist for all Lie types; see \cite[Remark 6.3]{fangSimplicesNewtonOkounkovBodies2018} and \cite{fangEssentialBases2017}.  We remark that the proof of Theorem~\ref{thm;gwidth} can be applied to any suitable valuation $\mathbf{v}$ on $G\sslash N$, provided that one can prove a lemma similar to \cite[Theorem 6.2]{fangSimplicesNewtonOkounkovBodies2018} for the related Newton-Okounkov bodies. Thus, even if it is not the case that there exist valuations $\mathbf{v}_{\underline{\beta}}$ as described in Theorem~\ref{thm;gwidth} for all Lie types, one can still hope to close the Karshon--Tolman conjecture by applying  this approach to some other family of valuations. 

We end by briefly discussing how Theorem \ref{thm;gwidth} relates to previous results.\footnote{See \cite[Section 2]{fangSimplicesNewtonOkounkovBodies2018} for a more thorough survey. Our understanding at the time of writing this is that this survey is up-to-date, with the exception of \cite{alekseevActionangleCoordinatesCoadjoint2020} which appeared some time later.} Tight lower bounds were proven using Gelfand--Zeitlin systems for arbitrary coadjoint orbits of type $A_n$, $B_n$, or $D_n$ in \cite{pabiniakGromovWidthNonregular2013}. Tight lower bounds were proven for coadjoint orbits of arbitrary type using toric degenerations in \cite{fangSimplicesNewtonOkounkovBodies2018}, but, due to the limitations of the toric degeneration machinery at the time (as discussed in our introduction), those results only hold for orbits $\mathcal{O}_\lambda$ such that $\lambda$ is a scalar multiple of an integral dominant weight (such orbits are called \emph{rational}).  Tight lower bounds for all regular coadjoint orbits in arbitrary type were given in  \cite{alekseevActionangleCoordinatesCoadjoint2020}.  Thus, the remaining open cases are all non-regular, non-rational orbits in Lie type not equal to $A_n$, $B_n$, or $D_n$. Note that every coadjoint orbit of $G_2$ is rational or regular, so the conjecture is already closed for all $G_2$ coadjoint orbits. Theorem \ref{thm;gwidth}  closes the remaining cases, modulo the existence of suitable valuations.

\appendix

\section{Proof of Theorem \ref{thm; toric degenerations main theorem}}\label{section; proof of toric deg theorem}

We first establish a general result, adopting the notation of Section~\ref{section;degenerations background}. Recall that the variety $X$ is decomposed by smooth subvarieties $X^\sigma$ indexed by elements $\sigma$ of a poset $\Sigma$.  For each $\sigma \in \Sigma$, the subfamilies  $\X^{\overline{\sigma}}, \X^\sigma \subset \X$ are defined by the decomposition of $X$ and the trivialization of $\X$ away from 0 as in \eqref{eq; definition of subfamilies}.  Denote by $\X_z$, $\X_z^\sigma$, and $\X_z^{\overline{\sigma}}$ the fiber of $\pi$ over $z\in \C$ in $\X$, $\X^\sigma$, and $\X^{\overline{\sigma}}$ respectively.  By definition, $\X^\sigma \subset \X^{\overline{\sigma}} \subset \X$ and $\X_z^\sigma \subset \X_z^{\overline{\sigma}} \subset \X_z$ for all $z$. Let $Z \subset \X$ (respectively $Z^\sigma \subset \X^\sigma$ and  $Z^{\overline{\sigma}} \subset \X^{\overline{\sigma}}$) denote the union of the singular locus of $\X$ (respectively $\X^\sigma$ and $\X^{\overline{\sigma}}$) and the  critical set of $\pi$ viewed as a map with domain $\X$ (respectively $\X^\sigma$ and $\X^{\overline{\sigma}}$). Denote $U_z = \X_z \setminus (\X_z \cap Z)$, $U_z^\sigma = \X_z^\sigma \setminus (\X_z^\sigma \cap Z^\sigma)$, and 
$U_z^{\overline{\sigma}} = \X_z^{\overline{\sigma}} \setminus (\X_z^{\overline{\sigma}} \cap Z^{\overline{\sigma}})$.

The proof of the following is an application of Lemma \ref{lemma; ham grp action preserving pi}. The first half of the proof follows the outline of \cite[Corollary~2.11]{haradaIntegrableSystemsToric2015}, and the second half is a straightforward exercise.

\begin{lemma}\label{prop; gh flow preserving dh measure} 
	Let $\X$ be a variety and let $\pi \colon \X \to \C$ be a morphism of varieties such that 
	$Z$ is contained in $\X_0$, $U_0$ is non-empty, and  $\X\setminus Z$ is K\"ahler. Let $K$ be a connected Lie group and let $\psi\colon \X \to \kk^*$ be a continuous map. Assume:
	\begin{enumerate}[label=(\roman*), start=1]
		\item 
		There is a Hamiltonian action of $K$ on $\X\setminus Z$ with moment map $\psi\vert_{\X\setminus Z}$ such that the action of $K$ preserves the fibers of $\pi$ and the K\"ahler metric.
		\item   The map $(\pi,\psi)\colon \X \to \C \times \kk^*$ is proper as a map to its image.
	\end{enumerate}
	Then,  the flow $\varphi_{t}(x)$ is defined for all $x \in U_0$ and $t\in \R$. For $t\neq 0$ fixed, $\varphi_{-t}\colon (U_0,\omega_0,\psi)\to (\X_t,\omega_t,\psi)$ is a map of Hamiltonian $K$-manifolds. If additionally:
	\begin{enumerate}[label=(\roman*), start=3]
		\item  The Duistermaat-Heckman measures of $(U_0,\omega_0,\psi)$ and $(\X_t,\omega_t,\psi)$ are equal,
	\end{enumerate}
	then $\varphi_{-t} \colon U_0 \to \X_{t}$ is a symplectomorphism onto a dense subset of $\X_{t}$.
\end{lemma}

We now turn to the proof of Theorem~\ref{thm; toric degenerations main theorem}. Throughout the remainder of this section, let $(X,M,\omega_M)$ be a decomposed K\"ahler variety and let $\pi \colon \X \to \C$ be a degeneration of $X$ that satisfies assumptions \ref{assumeA}--\ref{assumeG} as in Section \ref{section; Toric degenerations of stratified affine varieties}.

\begin{lemma}\label{lemma;immediate consequences}
	For all $\sigma \in \Sigma$, $Z_0^\sigma$ is contained in $\X_0^\sigma$. Moreover, $U_0^\sigma$ is the smooth locus of $\X_0^\sigma$. 
\end{lemma}

\begin{proof}
	The fact that $Z^\sigma_0$ is contained in $\X_0^\sigma$ is a consequence of \ref{assumeB} and the smoothness of $X^\sigma$. By \eqref{eq; definition of subfamilies}, $\X^\sigma$ is an open subset of $\X^{\overline{\sigma}}$.   Thus, $Z^\sigma = Z^{\overline{\sigma}} \cap \X^\sigma$ and $U^\sigma = U^{\overline{\sigma}} \cap \X^\sigma$. By Assumption \ref{assumeA} and Proposition \ref{cor; HK cor 2.10}, $U_0^{\overline{\sigma}}$ is the smooth locus of $\X^{\overline{\sigma}}_0$. Since $\X_0^\sigma$ is an open subset of $\X^{\overline{\sigma}}_0$, it follows that $U_0^\sigma$ is the smooth locus of $\X_0^\sigma$.
\end{proof}

\begin{lemma}\label{lemma;immediate consequences II}
    The following statements are true for all $\sigma \in \Sigma$.
    \begin{enumerate}[label=(\alph*)]
        \item The flow $\varphi_{-1}^\sigma$ is defined for all $x \in  U_0^\sigma$.
        \item The map $\varphi_{-1}^\sigma\colon (U_0^\sigma,\omega_0^\sigma,\psi) \to (\X_1^\sigma,\omega_1^\sigma,\psi)$ is a map of Hamiltonian $T$-manifolds.
        \item The set $D^\sigma = \varphi_{-1}^\sigma (U_0^\sigma)$ is dense in $\X_1^\sigma$.
        \item The map  $\varphi_{-1}^\sigma\colon (U_0^\sigma,\omega_0^\sigma) \to (\X_1^\sigma,\omega_1^\sigma)$ is a symplectomorphism onto its image.
    \end{enumerate}
\end{lemma}

\begin{proof} Fix any $\sigma \in \Sigma$. The proof is a direct application of Lemma \ref{prop; gh flow preserving dh measure} to the subfamily $\pi\colon \X^{\sigma} \to \C$.  We have $Z^\sigma \subset \X^\sigma_0$ by Lemma \ref{lemma;immediate consequences}. The set $U_0^\sigma$ is non-empty by \ref{assumeF}. The smooth subvariety $U_0^\sigma $ inherits a K\"ahler structure from the embedding into $M\times \C$ given by Assumption \ref{assumeB}. Assumptions (i)--(iii) of Lemma \ref{prop; gh flow preserving dh measure} are precisely assumptions \ref{assumeE} and \ref{assumeF}.
\end{proof}

Recall the stratified gradient Hamiltonian flow $\varphi_t$ of the stratified gradient Hamiltonian vector field $V_\pi$, defined in \eqref{eq; stratified flow def}.

\begin{lemma} \label{CTSflowlemma}
The stratified gradient 
Hamiltonian flow $\varphi_t \colon \X_1 \to \X_{1-t}$ is continuous for all $0<t<1$.
\end{lemma}

Before proving Lemma \ref{CTSflowlemma}, we note the following elementary lemma.

\begin{lemma} \label{dumbsubsequencelemma}
Let $X$ and $Y$ be metric spaces, let $f\colon X\to Y$ be a map of the underlying sets, and let $x\in X$. Assume that for every sequence $\{x_i\}_{i\in \N} \subset X$ with $\lim_{i\to \infty} x_i = x$, there is a subsequence $\{x_{i_j}\}_{j\in \N}$  with the property that $\lim_{j\to \infty} f(x_{i_j}) = f(x)$. Then, $f$ is continuous at $x$. \end{lemma}

\begin{proof}[Proof of Lemma \ref{CTSflowlemma}] Fix $\sigma\in\Sigma$, $x\in \X_1^\sigma$, and $T\in (0,1)$. We prove that $\varphi_T\colon \X_1\to \X_{1-T}$ is continuous at $x$. 

Let  $\{x_i\}_{i\in \N} \subset \X_1$ be an arbitrary sequence converging to $x$. By Lemma~\ref{dumbsubsequencelemma}, it suffices to find a subsequence $\{x_{i_j}\}_{j\in \N}$ so that $\{\varphi_{T}(x_{i_j})\}_{j\in \N}$ converges to $\varphi_{T}(x)$. By passing to a subsequence if necessary, we may assume without loss of generality that $\{x_i\}_{i\in \N}\subset \X^\tau_1$ for some $\tau \geq \sigma$.  (If $\tau = \sigma$ then the result follows immediately since the restriction of $V_\pi$ to $\X^\sigma\setminus \X_0^\sigma$ is smooth. The remainder of the proof deals with the case $\tau > \sigma$.)

Consider the sequence of paths $\{\varphi_t (x_i) \colon [0,T] \to \X^\tau\}_{i\in \N}$. Since $\varphi_t$ preserves $\psi$ and $x_i$ converges to $x$, there is a compact set $c\subset \psi(X)$ so that $\varphi_t(x_i) \in \psi\n(c) \cap \pi\n([0,T])$ for all $t\in [0,T]$ and all $i\in \N$. By Assumption \ref{assumeE}II),  $\psi\n(c) \cap \pi\n([0,T])$ is compact. By a standard diagonalization argument, by replacing $\{x_i\}_{i\in \N}$ with a subsequence, we may assume that for each $t\in \Q \cap [0,T]$
the sequence of points $\{\varphi_t(x_i)\}_{i\in \N}$ converges as $i\to \infty$.

We will show below that the sequence of time derivatives
\begin{equation} \label{sequenceofderivatives}
\frac{d\varphi_t (x_i)}{dt} = V_\pi(\varphi_t(x_i)), \qquad i\in \N,
\end{equation}
 converges uniformly on $\Q \cap [0,T]$. 
 Assuming this to be true for the moment, because $\Q \cap [0,T]$ is dense in $[0,T]$ it follows that the sequence $\{V_\pi(\varphi_t(x_i))\}_{i\in \N}$ converges uniformly on $[0,T]$. As a consequence, the paths $\varphi_t (x_i)$ converge uniformly to a $C^1$ path $\mu\colon [0,T] \to \X$. For all $t' \in [0,T],$
 \begin{align*}
V_\pi(\mu(t' ))  &= V_\pi(\lim_{i \to \infty} \varphi_{t'} (x_i)) & \text{$\varphi_t(x_i)$ converges to $\mu(t)$.}\\
& =\lim_{i\to \infty} V_\pi( \varphi_{t'} (x_i))& \text{Assumption \ref{assumeG}} \\
& = \lim_{i\to \infty} \left( \frac{d}{dt}\varphi_t (x_i)\Big|_{t=t'}\right) & \text{Definition of $\varphi_t$.}\\
& = \frac{d}{dt}\left( \lim_{i\to \infty} \varphi_t (x_i)\right) \Big|_{t=t'} & \text{Uniform convergence of derivatives.} \\
& = \frac{d}{dt}\mu(t) \Big|_{t=t'} & \text{$\varphi_t(x_i)$ converges to $\mu(t)$.}
 \end{align*}
Since $\varphi_t$ preserves $\psi$,
\begin{equation} \label{equation; where does the psi go}
    \psi(\mu(t)) = \psi\left(\lim_{i\to \infty}\varphi_t(x_i)\right) 
    = 
    \lim_{i\to \infty}\psi\left(\varphi_t(x_i)\right)
    = \lim_{i\to \infty} \psi\left(x_i\right) = \psi(x).
\end{equation}
Thus, $\mu([0,T])$ is contained in $\psi\n(\psi(x))$. It follows by Assumption \ref{assumeE}III) that $\mu([0,T])$ is contained in $\X^\sigma$.

In summary, the path $\mu(t)$ solves the same initial value problem on $\X^\sigma\setminus \X_0^\sigma$, defined by the smooth vector field $V_\pi^\sigma$ and the initial value $\mu(0) = x$, as the integral curve $\varphi^\sigma_t (x)$. It follows by uniqueness of solutions that
\[
\lim_{i\to \infty} \varphi_t(x_i) = \mu(t) = \varphi_t (x)
\]
 for all $t\in [0,T]$. In particular, this holds for $t=T$, which completes the proof (modulo the claim that~\eqref{sequenceofderivatives} converges uniformly).

It remains to show that~\eqref{sequenceofderivatives} converges uniformly on $\Q \cap [0,T]$. For each $t\in \Q\cap [0,T]$ let $\mu(t)=\lim_{i\to \infty} \varphi_t(x_i)$; we have already established this limit exists. 
Assume for the sake of contradiction that~\eqref{sequenceofderivatives} does not converge uniformly to $V_\pi(\mu)$, as a function of $t\in \Q\cap [0,T]$. Then there is some $\gamma>0$ so that, for all $N>0$, there is $i \ge N$ and $t_i\in \Q\cap [0,T]$ with 
\begin{equation} \label{boundedaway}
\gamma< || V_\pi(\varphi_{t_i}(x_i)) - V_\pi( \mu(t_i))||.
\end{equation}
By passing to a subsequence, we may assume that the sequence $\{(x_i,t_i)\}_{i\in \N}$ satisfies \eqref{boundedaway} for all $i \in \N$.
By the compactness of $[0,T]$, we may also assume that $
\lim_{i\to \infty} t_i = t_\star
$ for some $t_\star\in [0,T]$. 
Similarly, by the compactness of $\psi\n(c) \cap \pi\n(t_\star)$ (and the same argument as~\eqref{equation; where does the psi go}), we may additionally assume that 
$
\lim_{i\to \infty} \varphi_{t_\star}(x_i) =y 
$
for some $y \in \X^\sigma$. 
We first prove three preliminary claims. 

\textbf{Claim 1:} $\mu$ has a unique continuous extension to $[0,T]$. 

\textbf{Proof of Claim 1:} Let $\{s_n\}_{n\in \N}\in \Q\cap [0,T]$ be a sequence converging to some $s\in [0,T]$. Assume the sequence $\{\mu(s_n)\}_{n\in \N}$ is not Cauchy; then there is $\epsilon>0$ so that for all $N>0$ there exist $m, n\in\N$ with $||\mu(s_n)-\mu(s_m)||>\epsilon$. Since $\lim_{n\to \infty} s_n= s$, for any $L>0$ we may find $n,m$ so that $|s_n-s_m|<1/L$ and $||\mu(s_n)-\mu(s_m)||>\epsilon$. For any $\epsilon'>0$ we can pick $i\in \N$ sufficiently large that 
\[
||\varphi_{s_n}(x_i) - \mu(s_n)||<\epsilon'/2, \text{ and } \qquad ||\varphi_{s_m}(x_i) - \mu(s_m)||<\epsilon'/2.
\]
Then
$
||\varphi_{s_n}(x_i) - \varphi_{s_m}(x_i) || > \epsilon-\epsilon'.
$
By the mean value inequality applied to the path $\varphi_s(x_i)$, there is $s'$ between $s_m$ and $s_n$ so that 
\[
||V_\pi(\varphi_{s'}(x_i))||\ge \frac{||\varphi_{s_n}(x_i) - \varphi_{s_m}(x_i) ||}{|s_n-s_m|} > L(\epsilon-\epsilon').
\]
Since $L$ and $\epsilon'$ were arbitrary, this implies that $||V_\pi||$ is unbounded on the compact set $\psi\n(c) \cap \pi\n([0,T])$. By the assumption~\ref{assumeG}, this is a contradiction. Therefore $\lim_{n\to \infty} \mu(s_n)$ exists. This extension is also unique; if  $s_n\to s$ and $s_n'\to s$ are two sequences with $\lim_{n\to \infty} \mu(s_n) \ne \lim_{n\to \infty} \mu(s_n')$, then the sequence $\mu(s_1),\mu(s_1'),\mu(s_2),\mu(s_2'),\dots$ has no limit, contradicting what we have shown above. This proves Claim 1.

\textbf{Claim 2:} $\lim_{i\to \infty} \varphi_{t_\star}(x_i) = \mu(t_\star)$. 

\textbf{Proof of Claim 2:} If not, and write $\lim_{i\to \infty} \varphi_{t_\star}(x_i) = y$ as before; then $||y - \mu(t_\star)|| = \epsilon>0$. Let $L>0$, and pick $t'\in \Q\cap [0,T]$ with $|t'-t_\star|<1/L$.
For any $\epsilon'>0$, we may choose $i\in \N$ with
\[
||\varphi_{t_\star}(x_i) - y||<\epsilon'/2, \text{ and } \qquad ||\varphi_{t'}(x_i) - \mu(t')||<\epsilon'/2.
\]
Then 
$
||\varphi_{t_\star}(x_i) - \varphi_{t'}(x_i) || > \epsilon-\epsilon'.
$
By the mean value inequality applied to the path $\varphi_t(x_i)$, there is $t''$ between $t'$ and $t_\star$ so that 
\[
||V_\pi(\varphi_{t''}(x_i))||\ge \frac{||\varphi_{t_\star}(x_i) - \varphi_{t'}(x_i) ||}{|t_\star-t'|} > L(\epsilon-\epsilon').
\]
Since $L$ and $\epsilon'$ were arbitrary,  this implies that $||V_\pi||$ is unbounded on the compact set $\psi\n(c) \cap \pi\n([0,T])$. By the assumption~\ref{assumeG}, this is a contradiction. Thus $\lim_{i\to \infty} \varphi_{t_\star}(x_i) = y = \mu(t_\star)$, which establishes Claim 2.

\textbf{Claim 3:} $\lim_{i\to \infty} ||\varphi_{t_i}(x_i) -\varphi_{t_\star}(x_i)||=0$. 

\textbf{Proof of Claim 3:} If not, then there is some $\epsilon>0$ so that for all $N$ there exists $i>N$ with $||\varphi_{t_i}(x_i) -\varphi_{t_\star}(x_i)||>\epsilon$. Let $L>0$, and pick $N$ sufficiently large that $|t_i-t_\star|<1/L$ for all $i>N$. Fix $i>N$ so that $||\varphi_{t_i}(x_i) -\varphi_{t_\star}(x_i)||>\epsilon$. By the mean value inequality applied to the path $\varphi_t(x_i)$, there is some $t'$ between $t_\star$ and $t_i$ so that
\[
||V_\pi(\varphi_{t'}(x_i))||\ge \frac{||\varphi_{t_i}(x_i) - \varphi_{t_\star}(x_i) ||}{|t_i-t_\star|} > L\epsilon.
\]
Since $\epsilon$ was fixed and $L$ arbitrary, this implies that $||V_\pi||$ is unbounded on the compact set $\psi\n(c) \cap \pi\n([0,T])$. By the assumption~\ref{assumeG}, this is a contradiction. This proves Claim 3.

Now, we may complete the proof that~\eqref{sequenceofderivatives} converges uniformly on $\Q\cap [0,T]$. One has for all $i\in \N$, 
\[
|| \varphi_{t_i}(x_i) - \mu(t_\star)|| \le ||\varphi_{t_i}(x_i) - \varphi_{t_\star}(x_i) || + ||\varphi_{t_\star}(x_i) - \mu(t_\star) ||.
\]
By Claim 2 and Claim 3, both terms on the right hand side go to zero as $i\to \infty$. By the assumption~\ref{assumeG}, it follows that
\begin{equation} \label{thislimit}
\lim_{i\to \infty} V_\pi(\varphi_{t_i}(x_i)) = V_\pi(\mu(t_\star)).
\end{equation}
Similarly, by Claim 1 and assumption~\ref{assumeG}, one has
\begin{equation} \label{thislimit2}
\lim_{i\to \infty} V_\pi(\mu(t_i)) = V_\pi(\mu(t_\star)).
\end{equation}
At the same time, by~\eqref{boundedaway}
\[
0<\gamma< || V_\pi(\varphi_{t_i}(x_i)) - V_\pi(\mu(t_i))|| \le ||V_\pi(\varphi_{t_i}(x_i)) - V_\pi(\mu(t_\star)) || + ||V_\pi(\mu(t_\star) - V_\pi(\mu(t_i)) ||
\]
for all $i$. But by~\eqref{thislimit} and~\eqref{thislimit2}, the right hand side goes to zero as $i\to \infty$. This is a contradiction. Therefore, the sequence of derivatives~\eqref{sequenceofderivatives} converges uniformly on $\Q\cap [0,T]$, as desired.
\end{proof}

The proof of the following closely follows the outline of \cite[Theorem 2.12]{haradaIntegrableSystemsToric2015}. The difficulty here consists in finding a $\rho$ which works uniformly for all strata $\sigma \in \Sigma$. The details, which are an exercise in point set topology, are left to the reader.

\begin{lemma} \label{precontinuity}
For all $x\in \X_1$, the limit $
	\lim_{t\to 1^-} \varphi_t(x) $
exists and is an element of $\X_0$. For any open precompact
subset $A \subset \psi(M)$ and $\epsilon>0$, there exists $\rho>0$ such that for all $0 < s < \rho$ and $x\in \psi\n(A)\cap \X_1$,
\begin{equation*}
	||\varphi_{1-s}(x)-\lim_{t\to 1^-}\varphi_t(x) ||<\epsilon.
\end{equation*}
\end{lemma}

The following proof closely follows the second half of the proof of \cite[Theorem 2.12]{haradaIntegrableSystemsToric2015}.

\begin{proof}[Proof of Theorem \ref{thm; toric degenerations main theorem}] For each $x \in \X_1$, define $
		\phi(x) = \lim_{t\to 1^-} \varphi_t (x). $
	Since these limits exist and are elements of $\X_0$ (Lemma~\ref{precontinuity}), this defines a map $\phi\colon \X_1 \to \X_0$. 
	
	For each $\sigma \in \Sigma$, let $D^\sigma = \varphi_{-1}^\sigma(U_0^\sigma)$ (the flow $\varphi_{-1}^\sigma$ is defined at points in $U_0^\sigma$ by Lemma \ref{lemma;immediate consequences II}). By Lemma \ref{lemma;immediate consequences II}, $D^\sigma$ is dense in $\X_1^\sigma$ and  $\varphi_1^\sigma\colon D^\sigma \to U_0^\sigma$ is a symplectomorphism. In particular, for all $x \in D^\sigma$,  
	\[
		\phi(x) = \lim_{t\to 1^-} \varphi_t (x) = \lim_{t\to 1^-} \varphi_t^\sigma (x) = \varphi_1^\sigma(x).
	\]
	Thus, the restriction of $\phi$ to $D^\sigma$ coincides with $\varphi_1^\sigma$.
	
	We claim that $\phi$ is continuous. Fix some open precompact subset $A \subset \psi(M)$ and $\epsilon >0$.  By Lemma~\ref{precontinuity}, there exists $1>t >0$ such that for all $x\in \psi\n(A)\cap \X_1$, 
	$
		||\varphi_{t}(x)- \phi(x)||<\epsilon/3.
	$
	By Lemma~\ref{CTSflowlemma}, there exists $\delta>0$ such that for any $x,y \in \X_1$, if $||x-y||<\delta$, then 
	$
	||\varphi_t(x) - \varphi_t(y)||<\epsilon/3.
	$
	Combining these inequalities, we have that for all  $x,y \in \mu\n(A)\cap \X_1$ such that $||x-y||<\delta$,
	\begin{align*}
	||\phi(x)-\phi(y) || & \le || \phi(x) - \varphi_{t}(x)|| +||\varphi_t(x) -  \varphi_t(y)||+||  \varphi_{t}(y)-\phi(y)|| < \epsilon.
	\end{align*}
	Thus $\phi$ is continuous.
	
	By Lemma \ref{lemma;immediate consequences II}, the maps $\varphi_1^\sigma\colon D^\sigma \to U_0^\sigma$ are $T$-equivariant and satisfy $\psi\circ\varphi_1^\sigma = \psi$ for all $\sigma \in \Sigma$. 
	It follows that $\phi$ is $T$-equivariant and satisfies $\psi \circ \phi = \phi$ since $\phi$ is continuous, and for each $\sigma\in \Sigma$ it coincides with $\varphi_1^\sigma$ on the dense subset $D^\sigma \subset \X_1^\sigma$. 
	
	Let $c$ be a compact subset of $\X_0$. By Assumption \ref{assumeE}II), there exists a compact subset $c'\subset \ttt^*$ such that $c$ is contained in $ (\pi,\psi)\n(\{0\}\times c')$. Since $\psi \circ \phi = \psi$, the pre-image $\phi\n(c)$ is contained in $ (\pi,\psi)\n(\{1\}\times c')$. Since  $\phi\n(c)$ is a closed subset of a compact set, it is compact. Thus $\phi$ is proper.
	
    Let $x\in \X_0^\sigma$ for some arbitrary $\sigma$. Since $U_0^\sigma$ is dense in $\X_0^\sigma$ (Lemma \ref{lemma;immediate consequences}) we can find a sequence $x_i \subset U_0^\sigma$ such that  $x_i \to x$ as $i \to \infty$. Since $\phi$ is proper, there is a compact subset $c\subset \X_1$ such that $\phi\n(x_i) \in c$ for all $i \in \N$. Thus, there exist a subsequence $x_{i_k}$ and a point $y \in c$ such that $\phi\n(x_{i_k})\to y$ as $k\to \infty$. It follows by continuity of $\phi$ that $\phi(y) = x$. Thus $\phi\colon \X_1 \to \X_0$ is surjective. 
\end{proof}

\bibliographystyle{amsplain}

\bibliography{degenerationsBibliography}

\end{document}